\documentclass[10pt]{amsart}
\usepackage{amsthm,amsfonts,amssymb,amsmath,amsxtra, bm}
\usepackage{comment}
\includecomment{comment}
\usepackage[all]{xy}
\SelectTips{cm}{}
\usepackage{xr-hyper}
\usepackage[colorlinks=
   citecolor=Black,
   linkcolor=Red,
   urlcolor=Blue,
   backref=page]{hyperref}
\usepackage{verbatim}
\usepackage[top=1.15in, bottom=1.15in, left=1.4in, right=1.4in]{geometry}
\usepackage{mathrsfs}
\usepackage{tikz-cd}

\RequirePackage{xspace}
\RequirePackage{etoolbox}
\RequirePackage{varwidth}
\RequirePackage{enumitem}
\RequirePackage{tensor}
\RequirePackage{mathtools}
\RequirePackage{longtable}
\RequirePackage{multirow}

\usepackage{etoolbox}

\newcommand{\fka}{\ensuremath{\mathfrak{a}}\xspace}

\def\i{^{-1}}

\newcommand{\BA}{\ensuremath{\mathbb {A}}\xspace}

\newcommand{\BC}{\ensuremath{\mathbb {C}}\xspace}

\newcommand{\BF}{\ensuremath{\mathbb {F}}\xspace}
\newcommand{\BG}{\ensuremath{\mathbb {G}}\xspace}

\newcommand{\BP}{\ensuremath{\mathbb {P}}\xspace}
\newcommand{\BQ}{\ensuremath{\mathbb {Q}}\xspace}
\newcommand{\BR}{\ensuremath{\mathbb {R}}\xspace}

\newcommand{\BX}{\ensuremath{\mathbb {X}}\xspace}

\newcommand{\BZ}{\ensuremath{\mathbb {Z}}\xspace}

\newcommand{\bfK}{\ensuremath{\mathbf {K}}\xspace}

\newcommand{\CA}{\ensuremath{\mathcal {A}}\xspace}

\newcommand{\CG}{\ensuremath{\mathcal {G}}\xspace}

\newcommand{\CL}{\ensuremath{\mathcal {L}}\xspace}
\newcommand{\CM}{\ensuremath{\mathcal {M}}\xspace}

\newcommand{\CO}{\ensuremath{\mathcal {O}}\xspace}

\newcommand{\RH}{\ensuremath{\mathrm {H}}\xspace}

\newcommand{\RT}{\ensuremath{\mathrm {T}}\xspace}

\newcommand{\RV}{\ensuremath{\mathrm {V}}\xspace}

\newcommand{\bK}{\ensuremath{\mathbf K}\xspace}

\newcommand{\ab}{{\mathrm{ab}}}

\newcommand{\ad}{{\mathrm{ad}}}

\DeclareMathOperator{\Aut}{Aut}

\DeclareMathOperator{\charac}{char}

\DeclareMathOperator{\diag}{diag}

\DeclareMathOperator{\End}{End}

\DeclareMathOperator{\Gal}{Gal}
\newcommand{\GL}{\mathrm{GL}}

\newcommand{\smplus}{\scalebox{0.6}{+}}
\newcommand{\Phiplus}{\Phi^{\scalebox{0.6}{+}}}

\DeclareMathOperator{\Hom}{Hom}

\newcommand{\inv}{{\mathrm{inv}}}

\DeclareMathOperator{\Lie}{Lie}

\DeclareMathOperator{\Nm}{Nm}

\DeclareMathOperator{\rank}{rank}

\newcommand{\PGL}{{\mathrm{PGL}}}

\DeclareMathOperator{\Res}{Res}

\DeclareMathOperator{\Spec}{Spec}
\DeclareMathOperator{\Spf}{Spf}

\newcommand{\Sp}{{\mathrm{Sp}}}

\newcommand{\U}{\mathrm{U}}
\newcommand{\T}{\mathrm{T}}

\DeclareMathOperator{\val}{val}

\newcommand{\wt}{\widetilde}
\newcommand{\wh}{\widehat}

\newcommand{\ov}{\overline}

\newcommand{\lra}{\longrightarrow}

\newcommand{\bs}{\backslash}

\newtheorem{theorem}{Theorem}
\newtheorem{proposition}[theorem]{Proposition}
\newtheorem{lemma}[theorem]{Lemma}
\newtheorem {conjecture}[theorem]{Conjecture}
\newtheorem{corollary}[theorem]{Corollary}

\theoremstyle{definition}
\newtheorem{definition}[theorem]{Definition}

\newtheorem{remark}[theorem]{Remark}

\newtheorem{construction}[theorem]{Construction}
\newenvironment{altenumerate}
   {\begin{list}
      {\textup{(\theenumi)} }
      {\usecounter{enumi}
       \setlength{\labelwidth}{0pt}
       \setlength{\labelsep}{0pt}
       \setlength{\leftmargin}{0pt}
       \setlength{\itemsep}{\the\smallskipamount}
       \renewcommand{\theenumi}{\roman{enumi}}
      }}
   {\end{list}}
\newenvironment{altitemize}
   {\begin{list}
      {$\bullet$}
      {\setlength{\labelwidth}{0pt}
	   \setlength{\itemindent}{5pt}
       \setlength{\labelsep}{5pt}
       \setlength{\leftmargin}{0pt}
       \setlength{\itemsep}{\the\smallskipamount}
      }}
   {\end{list}}

\numberwithin{equation}{subsection}
\numberwithin{theorem}{subsection}

\newcommand{\aform}{\ensuremath{\langle\text{~,~}\rangle}\xspace}

\setitemize[0]{leftmargin=*,itemsep=\the\smallskipamount}
\setenumerate[0]{leftmargin=*,itemsep=\the\smallskipamount}

\renewcommand{\to}{%
   \ifbool{@display}{\longrightarrow}{\longrightarrow}%
   }
\let\shortmapsto\mapsto
\renewcommand{\mapsto}{%
   \ifbool{@display}{\longmapsto}{\shortmapsto}%
   }
\newlength{\olen}
\newlength{\ulen}
\newlength{\xlen}
\newcommand{\xra}[2][]{%
   \ifbool{@display}%
      {\settowidth{\olen}{$\overset{#2}{\longrightarrow}$}%
       \settowidth{\ulen}{$\underset{#1}{\longrightarrow}$}%
       \settowidth{\xlen}{$\xrightarrow[#1]{#2}$}%
       \ifdimgreater{\olen}{\xlen}%
          {\underset{#1}{\overset{#2}{\longrightarrow}}}%
          {\ifdimgreater{\ulen}{\xlen}%
             {\underset{#1}{\overset{#2}{\longrightarrow}}}
             {\xrightarrow[#1]{#2}}}}%
      {\xrightarrow[#1]{#2}}
   }
\makeatother
\newcommand{\xyra}[2][]{%
   \settowidth{\xlen}{$\xrightarrow[#1]{#2}$}%
   \ifbool{@display}%
      {\settowidth{\olen}{$\overset{#2}{\longrightarrow}$}%
       \settowidth{\ulen}{$\underset{#1}{\longrightarrow}$}%
       \ifdimgreater{\olen}{\xlen}%
          {\mathrel{\xymatrix@M=.12ex@C=3.2ex{\ar[r]^-{#2}_-{#1} &}}}%
          {\ifdimgreater{\ulen}{\xlen}%
             {\mathrel{\xymatrix@M=.12ex@C=3.2ex{\ar[r]^-{#2}_-{#1} &}}}
             {\mathrel{\xymatrix@M=.12ex@C=\the\xlen{\ar[r]^-{#2}_-{#1} &}}}}}%
      {\mathrel{\xymatrix@M=.12ex@C=\the\xlen{\ar[r]^-{#2}_-{#1} &}}}%
   }
\makeatletter
\newcommand{\xla}[2][]{%
   \ifbool{@display}%
      {\settowidth{\olen}{$\overset{#2}{\longleftarrow}$}%
       \settowidth{\ulen}{$\underset{#1}{\longleftarrow}$}%
       \settowidth{\xlen}{$\xleftarrow[#1]{#2}$}%
       \ifdimgreater{\olen}{\xlen}%
          {\underset{#1}{\overset{#2}{\longleftarrow}}}%
          {\ifdimgreater{\ulen}{\xlen}%
             {\underset{#1}{\overset{#2}{\longleftarrow}}}
             {\xleftarrow[#1]{#2}}}}%
      {\xleftarrow[#1]{#2}}
   }
\newcommand{\isoarrow}{%
   \ifbool{@display}{\overset{\sim}{\longrightarrow}}{\xrightarrow\sim}%
   }
   
\DeclareMathOperator{\Trace}{Tr}
\DeclareMathOperator{\Nilp}{Nilp}

\begin{document}

\thanks{Research of Haining Wang is partially supported by the National Natural Science Foundation of China (Grant number 12371009).}

\title[Optimal $p$-adic uniformization of unitary Shimura curves]{On optimal $p$-adic uniformization of unitary\\ Shimura curves}
\author{Michael Rapoport}
\address{Mathematisches Institut der Universit\"at Bonn, Endenicher Allee 60, 53115 Bonn, Germany}
\email{rapoport@math.uni-bonn.de}
\author{Haining Wang }
\address{Shanghai Center for Mathematical Sciences, Fudan University, No.2005 Songhu Road, Shanghai, 200438, China}
\email{wanghaining@fudan.edu.cn}

\date{\today}
\maketitle
\begin{abstract}
We prove  variants of Cherednik's theorem on $p$-adic uniformization, for two classes of Shimura curves attached to unitary groups. The first class  are PEL-type variants,  defined in \cite{RSZ} and \cite{unitshim}, of the PEL-type Shimura curves of  \cite{KRZI}. The advantage of these variants is that we can prove $p$-adic uniformization in optimal form for these variants, for any level structure prime to the distinguished $p$-adic place $v_{0}$ of the totally real field $F$ (in \cite{KRZI} this optimal form was only proved for level structures which are prime to \emph{all} $p$-adic places). Here optimality refers to the fact that we define \emph{$p$-integral models} by extending the moduli problem over $\BC$ solved by the Shimura variety, and prove  for them \emph{integral} $p$-adic uniformization  in terms of the formal Drinfeld upper half plane $\widehat\Omega_{F_{v_{0}}}$ for $F_{v_{0}}$. This greater generality compared to \cite{KRZI}  comes at the price of having to replace the reflex field by a bigger extension. 

The second class are  Shimura curves attached to unitary groups. Here the construction of $p$-integral models relies on the construction of \emph{canonical integral models} in the sense of  \cite{PR}. Again, the reflex field has to be replaced by a bigger extension--but this time due to our method of proof. We also identify the \emph{integral local Shimura curve} for an anisotropic unitary group over $\BQ_p$. 

\end{abstract}

\renewcommand{\baselinestretch}{0.4}\normalsize
\tableofcontents
\renewcommand{\baselinestretch}{1.0}\normalsize

\section{Introduction} 
The present paper is a sequel to \cite{KRZI}. Both papers  establish analogues of Cherednik's $p$-adic uniformization theorem for one-dimensional Shimura varieties which are, loosely speaking, attached to unitary groups.

Recall that Cherednik's theorem is concerned with the Shimura curve associated to a quaternion algebra over a totally real field which satisfies specific hypotheses. More precisely, let $F$ be a totally real number field. We fix an archimedean place $w_0$ of $F$ and a non-archimedean place $v_0$ of residue characteristic $p$. Let $D$ be a quaternion algebra over $F$ which is split at $w_0$ and non-split at all other archimedean places $w'\neq w_0$, and which is non-split at $v_0$. We denote by $D^\times$  the multiplicative group of $D$, but also the corresponding algebraic group  over $F$ or (by restriction of scalars) over $\BQ$. There is a natural Shimura datum $X_{D^\times}$ for $D^\times$, and an associated Shimura variety
\begin{equation}
{\rm Sh}_\bK(D^\times, X_{D^\times})=D^\times\backslash\big[ \Omega_\BR\times D^\times(\BA_f)/\bK\big] .
\end{equation}
Here $\Omega_\BR=\BP^1(\BC)\setminus \BP^1(\BR)$ denotes the union of the upper and the lower halfplane, and the action of $D^\times$ on $\Omega_\BR$ is via $D^\times\to D_{w_0}^\times$ and an identification of $D_{w_0}^\times$ with $\GL_2(\BR)$, which acts naturally on $\Omega_\BR$.  

The Shimura variety ${\rm Sh}_\bK(D^\times, X_{D^\times})$ has a canonical model ${\rm Sh}_\bK(D^\times, X_{D^\times})_F$ over the reflex field of $(D^\times, X_{D^\times})$ which can be identified with $F$ and its complex embedding given by $w_0$. For Cherednik's theorem we fix an open compact subgroup $\bK$ of the form
\begin{equation}\label{Cherprod}
\bK=\bK^{v_0}\cdot\bK_{v_0}\subset D^\times({\BA_{F, f}^{v_0}})\cdot D^\times(F_{v_0}) ,
\end{equation}
where $\bK_{v_0}$ is the unique maximal compact subgroup of $D^\times_{v_0}$. Then Cherednik's theorem asserts the existence of isomorphisms compatible with changes in $\bK^{v_0}$,
\begin{equation}\label{Cheredn}
{\rm Sh}_\bK(D^\times, X_{D^\times})_F\times_{\Spec F}\Spec \breve F_{v_0}\simeq {\bar D}^\times\backslash\big[ (\Omega_{F_{v_0}}\times_{\Sp\, F_{v_0}}\Sp\, \breve F_{v_0})\times D^\times(\BA_f)/\bK\big]. 
\end{equation}

Here $\Omega_{F_{v_0}}=\BP^1_{F_{v_0}}\setminus \BP^1(F_{v_0})$ denotes the Drinfeld upper halfplane for the local field $F_{v_0}$,  a rigid-analytic space over $F_{v_0}$. Also, $\bar D$ denotes the quaternion algebra over $F$ which is non-split at $w_0$, is split at ${v_0}$ and locally coincides with $D$ at all other places (the \emph{castling} of $D$). The action of $\bar{D}^\times$ on $\Omega_{F_{v_0}}$ is via  $\bar{D}^\times\to {\bar D}_{v_0}^\times$ and an identification of ${\bar D}_{v_0}^\times$ with $\GL_2(F_{v_0})$, which acts naturally on $\Omega_{F_{v_0}}$.  
The isomorphism \eqref{Cheredn} is to be interpreted as follows: the rigid-analytic space on the RHS  is (uniquely) algebraizable by the   projective algebraic curve over $\breve F_{v_0}$ on the left-hand side. Cherednik's theorem thus asserts that one may  pass from the tautological complex uniformization  of the Shimura variety ${\rm Sh}_\bK(D^\times, X_{D^\times})$ to $p$-adic uniformization, provided that the level structure is prime to $v_0$.

We refer to the introduction of \cite{KRZI} for a discussion of the proof of Cherednik's theorem and its generalizations. Here we only emphasize that when $F=\BQ$, the Shimura curves ${\rm Sh}_\bK(D^\times, X_{D^\times})$ are moduli spaces of abelian varieties with additional structure.  Using this,  Drinfeld \cite{Dr} gave a moduli-theoretic proof of Cherednik's theorem in this special case.  Furthermore, he proved an `integral version' of this theorem (which has the original version as a corollary).  This integral version  relies on a theorem on formal moduli spaces of $p$-divisible groups. When $F\neq \BQ$, Cherednik's Shimura curves do not represent a moduli problem of abelian varieties. In fact, ${\rm Sh}_\bK(D^\times, X_{D^\times})$ is a Shimura variety of abelian type which is not of Hodge type. Furthermore, this Shimura variety has other bad properties: the weight cocharacter $w_{D^\times}$ associated to $X_{D^\times}$ is not defined over $\BQ$ and the split rank of the connected center $Z^\circ$ over $\BQ$ and over $\BR$ differ, i.e. $\rank_\BQ(Z^{\circ})\neq\rank_\BR(Z^{\circ})$. Nonetheless, Boutot and Zink \cite{BZ} have given a proof of Cherednik's theorem by reduction to the case of PEL-Shimura varieties.  They also compare the descent data down to $F_{v_0}$ of both sides of \eqref{Cheredn}. More precisely, \cite{BZ} relates Cherednik's Shimura varieties to Shimura varieties associated to certain \emph{ fake unitary groups} corresponding to central division algebras over a CM-field extension of $F$ equipped with an involution of the second kind, for which $p$-adic uniformization is proved by Rapoport and Zink \cite{RZ}. In \cite{RZ}, these  uniformization theorems appear as a special instance of a general \emph{non-archimedean uniformization theorem}, which describes the formal completion of PEL-type Shimura varieties along a fixed isogeny class. In the case of $p$-adic uniformization, the whole special fiber forms a single isogeny class.

 The paper \cite{KRZI} proves $p$-adic uniformization for Shimura curves attached to unitary similitude groups associated to (anti-)hermitian vector spaces $V$ of dimension  $2$ over a CM-extension  $K$ of  $F$, the KRZ Shimura curves.  Analogously to  Cherednik's assumption on the quaternion algebra $D$, it is assumed in \cite{KRZI} that  $V$ is split at the archimedean
place $w_0$ of $F$ and ramified at all other archimedean places, and that  $V$ is ramified at the $p$-adic non-archimedean place $v_{0}$ (in particular, it is  assumed  that $v_0$ does \emph{not} split in $K$).  Of course, these Shimura curves are closely related to the Shimura curves considered by Cherednik (we refer to \cite{KRnew} for a general discussion of the relation between quaternion algebras and two-dimensional hermitian vector spaces). However, they are different. In particular, they  represent a moduli problem of abelian varieties, their weight cocharacter is defined over $\BQ$ and $\rank_\BQ(Z^{\circ})=\rank_\BR(Z^{\circ})$. The downside of these Shimura varieties compared to those of Cherednik is that the reflex field $F$ of Cherednik has to be replaced by a bigger reflex field $E$. The uniformization theorem of \cite{KRZI} is optimal when the level structure imposed is prime to $p$,  in the sense that it extends to an integral uniformization that allows an explicit interpretation of the points in the reduction modulo $p$.  In particular, the uniformization theorem of \cite{KRZI}  is optimal when $v_{0}$ is the only place of $F$ over $p$ because then the conditions that the level structure be prime to $p$ coincides with Cherednik's condition that the level structure be prime to $v_{0}$. 

The first aim of the present paper is to prove an \emph{optimal uniformization theorem} when there are more places over $p$, i.e., when there are \emph{banal} $p$-adic places, in the terminology of \cite[\S 7]{KRZI}. An approximation to such an optimal theorem is proved in \cite[\S 7]{KRZI}, but this is not completely satisfactory (as pointed out in the introduction of \cite{KRZI}). Here we succeed in proving such an optimal theorem by replacing the Shimura variety of \cite{KRZI} by the variant introduced in \cite{RSZ}, \cite{unitshim}, the RSZ Shimura curves. The downside of this variant is that its reflex field is (even) bigger than $E$.  

The second aim of the present paper is to  establish $p$-adic uniformization theorems for Shimura curves attached to unitary groups, under the same local assumptions.  These Shimura curves are not of PEL type but, assuming that  canonical $p$-integral models in the sense of \cite{PR} exist, we  construct $p$-integral models over $O_K$ of these Shimura curves. The optimal version of  $p$-adic uniformization for them should hold over $K$. Unfortunately, we cannot prove this optimal version when $F\neq\BQ$ but rather have to make a base change to a finite extension of $K$.  It seems a worthwhile goal to remove this extension of the reflex field. As a reward, we  expect interesting arithmetic applications (level raising, level lowering, bounding Selmer groups), cf. \cite[Introduction]{KRnew}. We also give an explicit model of the \emph{integral local Shimura variety},  in the sense of Scholze-Weinstein \cite{Sch} (comp. also \cite{PRloc}) attached to an anisotropic unitary group over a finite extension of $\BQ_p$, again at least after an extension of the reflex field.

We now summarize the layout of this article. The paper consists of four parts. In the first part, we state our main results and give some background. More precisely, in \S \ref{s:state},  we recall the definition of the Shimura curve in \cite{RSZ}, \cite{unitshim} which we will refer to as the RSZ Shimura curve. We formulate the associated moduli problem and introduce the terminology used throughout the paper. We also give the precise formulation of our main results. In section \ref{s:CM}, we review the notions of CM-types and CM-triples from earlier papers \cite{KRnew}, \cite{KRZI}.  In section \ref{s:Cmod}, we review the main theorem of complex multiplication for CM abelian varieties and its local analogue, the precise description of the Rapoport-Zink space associated to a torus with its  Weil descent datum. The proofs in the latter theory are based on the \emph{contraction functor} of \cite{KRZI} (in the case of rank one, which is  more elementary than the case of rank two treated in \cite{KRZI}). The notion of the \emph{special element} relative to $(\Phi_0^+, E)$ (Construction \ref{sp-elt-loc})  is used throughout the paper.  

The second part is devoted to the proof of our first main theorem which concerns the RSZ Shimura curves. In section \ref{s:formRSZ}, we treat the local theory. Here the case of special CM-triples is handled by citing rather directly from \cite{KRZI}. However, the case of banal CM-triples requires going back to the proofs in \cite{KRZI}. Indeed, in \cite{KRZI}, the descent datum in the banal case is not explicit. It is here that the advantage of the RSZ Shimura curves over the KRZ Shimura curves becomes apparent since in the RSZ case we are able to give the descent datum explicitly. In section \ref{s:globRSZ}, we prove the first main theorem, as a rather direct consequence of the local case. 

In the third part we go back to \cite{KRZI} and explain that the deficiencies of KRZ Shimura curves go away when extending slightly the reflex field, replacing it by an extension of degree $\leq 2$. After making this extension, we can make the descent datum completely explicit (but still the level of the Shimura curve has to be prime to $p$, like in \cite{KRZI}). This is the subject matter of \S \ref{s:revi}. In \S \ref{s:versus}, we compare KRZ Shimura curves with RSZ Shimura curves. 

The fourth part is devoted to the proof of the second main theorem which  concerns the Shimura curves attached to unitary groups. In \S \ref{s:ILSV}, we treat the local theory, by reduction to the case of the Rapoport-Zink space of RSZ CM-triples considered in \S \ref{s:formRSZ}   and the case of the  Rapoport-Zink space of a torus considered in \S \ref{s:Cmod}. Here we use the theory of integral local Shimura varieties from \cite{PRloc}. In \S \ref{s:unigl}, we treat the global case which works in complete parallel to the local case, except that we are basing ourselves on the conjectural theory of global canonical integral models of \cite{PR}. In our context, this theory is established when $p\neq 2$ and the special place $v_0$ is unramified \cite{DY}; the case of a ramified special place seems within reach. Since our approach is axiomatic, this possible deficiency of the literature plays no role in our exposition.

\subsection{Notation}
We use the following notation and conventions without comment in this paper.
\begin{itemize}
\item If $M$ is a number field, we write $\BA_{M}$ for its ad\`ele ring and denote by $\BA_{M, f}$ its finite part and by $\BA^{S}_{M,f}$ its prime to $S$ finite part for a finite collection of finite places $S$ of $M$. Let $p$ be a prime, and $v_{0}$ be a place of $M$ above $p$. We use the notation  $M_{p}=\prod_{v\mid p}M_{v}$ and  $M^{v_{0}}_{p}=\prod_{v\mid p, v\neq v_{0}}M_{v}$.

\item Let $F$ be a finite extension of $\BQ_p$ and let $K/F$ be an \'etale algebra of rank $2$, with corresponding involution $a\mapsto \bar a$.  Let  $V$ be a  free $K$-module, equipped with an alternating $\BQ_p$-bilinear form $\psi\colon V\times V\to \BQ_p$ such that $\psi(a v, v')=\psi(v, \bar a v')$ for $a\in K$ and $v, v'\in V$. Let $\Lambda$ be a $O_K$-lattice in $V$. Then the dual $O_K$-lattice is $\Lambda^\vee=\{x\in V\mid \psi(x, y)\in \BZ_p \text{ for all } y\in \Lambda\}$. The lattice $\Lambda$ is called \emph{almost self-dual} if $\Lambda=\Lambda^\vee$ if $K/F$ is a ramified field extension or an unramified field extension and the anti-hermitian form  associated to $\psi$ is split or if $K/F$ is split; if $K/F$ is an unramified field extension and the anti-hermitian form  associated to $\psi$ is non-split we ask that $\Lambda$ is contained in $\Lambda^\vee$ with colength one. 

\item  If $O$ is a discrete valuation ring, we write ${\rm Nilp}_{O}$  for the category of $O$-algebras $R$ such that  the image  of a uniformizer $\pi$ in $R$ is nilpotent. Similarly, we denote by $({\rm Sch}/\Spf O)$ the category of $O$-schemes such that $\pi\CO_S$ is a locally nilpotent ideal sheaf. Let $k$ be the residue field of $O$. If  $S$ is a scheme over $O$, we denote by $\bar{S}$ the special fiber of $S$, i.e.,  $\bar{S}=S\times_{\Spec O}\Spec k$.

\item Let $O$ be a ring, then ${\rm (LNSch)}/O$ denotes the category of locally noetherian schemes over $O$.

\end{itemize}

\part{Main results and preliminaries}
\section{Statement of the main results}\label{s:state}
\subsection{The RSZ Shimura variety}\label{ss:shimvar} Let us first recall  the Shimura curve in \cite{KRZI}. Let $K$ be a CM-field,  with  totally real subfield $F$. We denote the non-trivial $F$-automorphism of $K$ by $a\mapsto \bar a$. Let  $V$ be a two-dimensional $K$-vector space, equipped with an alternating $\BQ$-bilinear form 
\begin{equation}\label{varsigma}
\varsigma\colon V\times V\to \BQ 
\end{equation}
such that
\begin{equation}\label{introalt}
\varsigma(a x, y)=\varsigma(x, \bar a y), \quad x, y\in V, \, a\in K . 
\end{equation}
We call $\varsigma$ a \emph{anti-hermitian alternating form} (w.r.t. $K/F$).  There is a unique anti-hermitian form $\varkappa$ on $V$ such that 
 \begin{equation}\label{introback}
  \Trace_{K/\mathbb{Q}} a \varkappa(x, y) = \varsigma(ax, y), \quad x, y\in V, \, a\in K. 
\end{equation}
Conversely, the anti-hermitian form $\varkappa$ determines the alternating bilinear form $\varsigma$ with \eqref{introalt}. We  say that $\varkappa$ arises from $\varsigma$ by contraction. 

Recall that anti-hermitian spaces $V$ are determined up to isomorphism by their signature at the archimedean places of $F$ and their local invariants $\inv_v(V)\in \{\pm 1\}$ at the non-archimedean places $v$ of $F$.  Let $w_{0}$ be  an archimedean place such that $V_{w_{0}}$ has signature $(1, 1)$ and such that $V_{w}$ is definite for all archimedean places $w\neq w_{0}$. Let us be more precise. Let $\Phi = \Hom_{\text{$\mathbb{Q}$-Alg}}(K, \bar{\mathbb{Q}})$.
Let $r$ be a \emph{generalized CM-type of rank 2, special with respect to $w_{0}$}, i.e., a function 
\begin{equation}\label{GCM}
r: \Phi\lra \BZ_{\geqslant 0}, \qquad \,\, \varphi \mapsto r_\varphi,
\end{equation}
such that $r_{\varphi}+r_{\bar{\varphi}} =2$ for all $\varphi\in \Phi$, and such that  for  the extensions $\{\varphi_0, \bar{\varphi}_0\}$ of $w_{0}$ we have  $r_{\varphi_0} = r_{\bar{\varphi}_0} = 1$ and  with  $r_\varphi\in \{0, 2\}$ for $\varphi\notin \{\varphi_0, \bar\varphi_0\}$, comp. \cite{KRnew}. Then we demand that the signature of $V_\varphi=V\otimes_{K, \varphi} \BC$ be equal to  $(r_\varphi, 2-r_\varphi)$.  
\begin{remark}
Note that, because we are in very small rank, the generalized CM-type $\Phi$ is of \emph{fake Drinfeld type} relative to $\varphi_0$ and relative to $\bar{\varphi}_0$, in the sense of \cite[Ex. 2.3, (i)]{unitshim}.
\end{remark}
We denote the reflex field of $r$  by $E_r$. It is a subfield of $\bar\BQ$, the algebraic closure of $\BQ$ in $\BC$. Note  that $F$ embeds via $\varphi_0$ (or $\bar\varphi_0$) into $E_{r}$, and that the  archimedean place of $F$ induced by
\begin{equation}\label{introFintoE}
  F \overset{\varphi_0}{\longrightarrow} E_{r} \longrightarrow \BC
  \end{equation}
is equal to $w_0$.  If $F=\BQ$, then $E_{r}=F$.

Associated to these data, there is a {Shimura pair} $(G, X_G)$. Here    $G $ denotes the group of unitary 
similitudes of $V$, with similitude factor in $\BG_m$, an algebraic subgroup of ${\rm GSp}(V, \psi)$ over $\BQ$. The corresponding Shimura variety ${\rm Sh}_\bK(G, X_G)$ is the one appearing in \cite{KRZI}. We also refer to this Shimura variety as the KRZ Shimura curve.

Now we are going to define the Shimura variety of \cite{unitshim}, the RSZ\emph{ Shimura curve}. We fix a complex embedding  $\varphi_0\in \{\varphi_0, \bar\varphi_0\}$. We also fix  a \emph{classical} CM-type, i.e., a half-system of complex embeddings  $\Phiplus\subset \Phi$. We make the assumption  
$\varphi_0\in \Phi^{\smplus}$. 

An example of $\Phiplus$ is given by 
\begin{equation}\label{defPhi}
\quad \Phiplus\setminus \{\varphi_0\}=\{\varphi\in\Phi\mid r_\varphi=2\}.
\end{equation} 
This classical CM-type $\Phiplus=\Phi^{+}_{\varphi_0, r}$ is called the \emph{CM-type canonically associated to $r$ and $\varphi_0$}. Then $r$ is the \emph{strict fake Drinfeld generalized CM-type associated to} $(\Phiplus, \varphi_0)$, in the terminology of \cite[Ex. 2.3, (ii)]{unitshim}.

Let
$$
Z^\BQ=\{ z\in \Res_{K/\BQ}(\BG_m)\mid \Nm_{K/F}(z)\in \BG_m \} .
$$
The choice of $\Phiplus$ defines a Shimura datum $(Z^\BQ, X_{Z^\BQ})$ with reflex field $E_{\Phi^+}$, where $X_{Z^\BQ}$ is   the diagonal homomorphism
$$
h_{Z^\BQ}\colon \Res_{\BC/\BR}(\BG_m)\to Z^\BQ\otimes_\BQ \BR=\prod\nolimits_{\Phiplus} \Res_{\BC/\BR}(\BG_m)..
$$
We therefore obtain a Shimura variety ${\rm Sh}_{\bK_{Z^\BQ}}(Z^\BQ, X_{Z^\BQ})$. This Shimura variety will only play an auxiliary role, and we will take for $\bK_{Z^\BQ}$ the unique maximal compact subgroup $\bK_{Z^\BQ}^{\circ}$ of $Z^\BQ(\BA_f)$,
\begin{equation}
\bK_{Z^\BQ}^{\circ}=\{ z\in (O_K\otimes\wh{\BZ})^\times\mid \Nm_{K/F}(z)\in \wh{\BZ}^\times\} .
\end{equation}
The Shimura variety in which we are really interested is associated to the fiber product group
\begin{equation}\label{tildeG-1}
\wt G=G\times_{\BG_m}Z^{\BQ},
\end{equation}
where the maps to $\BG_m$ are given by $\Nm_{K/F}$, resp. the similitude character of $G$. The second part $X_{\wt G}$ of the Shimura datum is given by the conjugacy class of homomorphisms, 
$$
h_{\wt G}\colon \Res_{\BC/\BR}(\BG_m)\xrightarrow{(h_G, h_{Z^\BQ})} \wt G_\BR . 
$$
For an open compact subgroup $\bK_{\wt G} \subset \wt G(\BA_f)$, there is a Shimura variety ${\rm Sh}_{\bK_{\wt G}}(\wt G, X_{\wt G})$, 
whose complex points are given by 
$${\rm Sh}_{\bK_{\wt G}}(\wt G, X_{\wt G})(\BC) \simeq \wt G(\BQ) \bs [\Omega_\BR \times \wt G(\BA_f)/{\bK_{\wt G}}] .$$
Here  $\Omega_\BR$ is acted on by
$\wt G(\BR)$ via the projection to ${\rm GU}(V_{w_{0}})_{\rm ad}$ and a fixed
isomorphism  
${\rm GU}(V_{w_{0}})_{\rm ad}(\BR) \simeq \PGL_2(\BR)$.  

The reflex field $E$ of $(\wt G, X_{\wt G})$ is the composite $E=E_{r, \Phi^+}=E_r E_{\Phiplus}$, cf. \cite[(3.4)]{unitshim}.  In the case when $\Phiplus=\Phi_{\varphi_0, r}$ is the canonical  CM-type associated to  $\varphi_0, r$,   it follows from \eqref{defPhi} that $E=\varphi_0(K)E_r$, and is either $E_r$ or  a quadratic extension of $E_r$. The Shimura variety ${\rm Sh}_{\bK_{\wt G}}(\wt G, X_{\wt G})$ has a canonical model ${\rm Sh}_{\bK_{\wt G}}(\wt G, X_{\wt G})_{E}$ over $E$. 

The torus $Z^\BQ$ embeds naturally as a central subgroup of $G$, which gives rise to a product decomposition
\begin{equation}\label{tildeG}
\wt G\simeq \U\times Z^\BQ ,
\end{equation}
where $\U\subset G$ is the kernel of the multiplier map, cf. \cite[(3.5)]{unitshim}. Here $\U$ is considered as algebraic group over $\BQ$. We will only consider compact open subgroups $\bK_{\wt G}\subset \wt G(\BA_f)$ which, in terms of the product decomposition \eqref{tildeG} are of the form
\begin{equation}\label{decprod-intro}
\bK_{\wt G}= \bK_\U\times \bK^{\circ}_{Z^\BQ} ,
\end{equation} 
where $\bK_\U\subset \U(\BA_f)$.

\subsection{The RSZ moduli problem over $E$}\label{ss:modoverE}

The Shimura variety ${\rm Sh}_{\bK_{\wt G}}(\wt G, X_{\wt G})$ is of PEL-type. We will now formulate the corresponding moduli problem. We will first define a moduli problem for ${\rm Sh}_{\bK^{\circ}_{Z^\BQ}}(Z^\BQ, X_{Z^\BQ})$, cf. \cite[\S 3]{unitshim}. For  economy of exposition,  we will actually define this moduli problem over $\Spec O_{E_{\Phiplus}}$. 

 We fix an auxiliary non-zero ideal $\frak a$ of $O_K$. Let $\CA^{\frak a}_0$ be the category fibered in groupoids over ${\rm (LNSch)}/O_{E_{\Phiplus}}$ which associates to each $O_{E_{\Phiplus}}$-scheme $S$ the groupoid of triples $(A_0, \iota_0, \lambda_0)$, where 
\begin{altitemize}
\item[(1)] $A_0$ is an abelian scheme over $S$;
\item[(2)] $\iota_0\colon O_K\to\End(A_0)$ is an action of $O_K$ on $A_0$;
\item[(3)] $\lambda_0$ is a polarization on $A_0$ .
\end{altitemize}
We impose the conditions that $\ker \lambda_0=A_0[\frak a]$ and that for the Rosati involution of $\lambda_0$, 
\begin{equation}\label{Ros}
   {\rm Ros}_{\lambda_0} \bigl(\iota_0(a)\bigr)=\iota_0(\bar a)
   \quad\text{for all}\quad
   a\in O_K,
\end{equation}
and that $A_0$ is of CM-type $\Phiplus$, i.e.
\begin{equation}\label{kottcond}
    \charac\bigl(\iota_0(a) \mid \Lie A_0\bigr) = \prod\nolimits_{\varphi \in \Phiplus} (T-\varphi(a))
    \quad\text{for all}\quad
    a\in O_K.
\end{equation}
Here the left-hand side in \eqref{kottcond} denotes the characteristic polynomial of the action of $\iota_0(a)$ on the locally free $\CO_S$-module $\Lie A_0$;  the right-hand side, which is a priori a polynomial with coefficients in $O_{E_{\Phiplus}}$, is regarded as an element of $\CO_S[T]$ via the structure morphism. 
The morphisms $(A_0, \iota_0, \lambda_0) \to (A_0', \iota_0', \lambda'_0)$ in this groupoid are the  $O_K$-linear isomorphisms $\phi_{0}\colon A_0\isoarrow A_0'$ such that the pullback of $\lambda'_0$ is  $\lambda_0$.

This moduli problem is representable by a Deligne-Mumford stack $\CA^{\frak a} _{0}$ which is finite and \'etale over $\Spec O_{E_{\Phiplus}}$, cf. \cite[\S 3.4]{unitshim}. For suitable choice of $\frak a$, the stack $\CA^{\frak a}_0$ is non-empty. We fix such a choice of $\frak a$ in the sequel, which we assume to be prime to the prime number $p$ to be fixed later, and write $\CA_0$ for $\CA^{\frak a}_0$. 

The moduli stack $\CA_0\otimes_{O_{E_{\Phiplus}}}\BC$ is a disjoint sum of copies of the Shimura variety ${\rm Sh}_{\bK^{\circ}_{Z^\BQ}}(Z^\BQ, X_{Z^\BQ})$. In order to single out one copy, we fix a locally free $O_K$-module $\Lambda_0$ of rank one equipped with a nondegenerate alternating form $\aform_0\colon \Lambda_0 \times \Lambda_0 \to \BZ$ such that $\langle ax, y \rangle_0 = \langle x, \bar a y\rangle_0$ for all $x,y \in \Lambda_0$ and $a \in O_K$, such that the dual lattice $\Lambda_0^\vee$ of $\Lambda_0$ inside $\Lambda_0 \otimes_\BZ \BQ$ equals $\fka\i \Lambda_0$, and such that the anti-hermitian form on $\Lambda_0\otimes_{O_K, \varphi}\BC$ obtained  by contraction from $\langle \, , \, \rangle_0$  is  of signature $(1, 0)$ for every $\varphi\in\Phiplus$.   We denote by $\CA_0^{[\Lambda_0]}$ the open and closed substack of $\CA_0$ such that
$$
\CA_0^{[\Lambda_0]}(\BC)=\{(A_0, \iota_0, \lambda_0)\in \CA_0(\BC)\mid \RH_1(A_0, \BZ)\sim \Lambda_0 \} .
$$
Here the $O_K$-module $\RH_1(A_0, \BZ)$ is equipped with the natural $\BZ$-valued Riemann form induced by the polarization, and the $\sim$-sign means that $\RH_1(A_0, \BZ) \otimes_\BZ \wh\BZ$ and $\Lambda_0 \otimes_\BZ \wh\BZ$ are $\wh O_K$-linearly similar up to a factor in $\wh\BZ^\times$ and that $\RH_1(A_0, \BZ) \otimes_\BZ \BQ$ and $\Lambda_0 \otimes_\BZ \BQ$ are $K$-linearly similar up to a (necessarily positive) factor in $\BQ^\times$. Then there is an identification
\begin{equation}
\CA_0^{[\Lambda_0]}(\BC)={\rm Sh}_{\bK^{\circ}_{Z^\BQ}}(Z^\BQ, X_{Z^\BQ})(\BC),
\end{equation}
cf. \cite[(3.15)]{unitshim}.

Now we can formulate a moduli problem for the Shimura variety ${\rm Sh}_{\bK_{\wt G}}(\wt G, X_{\wt G})$, i.e., a functor over ${\rm (LNSch)}/E$.   We introduce the \emph{$K/F$-hermitian} space of dimension $2$,
\begin{equation}
\wt V=\Hom_K(V_0, V) ,
\end{equation}
where $V_0=\Lambda_0\otimes_{\BZ}\BQ$, with its anti-hermitian form obtained by contraction from $\langle \, , \, \rangle_0$. 
\begin{definition}\label{def:overE}Let $\wt\CA_{\bK_{\wt G}, E}$ be the category fibered in groupoids over ${\rm (LNSch)}/E$  which associates to each $E$-scheme $S$ the groupoid of tuples $(A_0,\iota_0,\lambda_0,A,\iota,\lambda,\bar\eta)$, where
\begin{altitemize}
\item[(1)] $(A_0, \iota_0, \lambda_0)$ is an object of $\CA^{[\Lambda_0]}_0(S)$; 
\item[(2)] $A$ is an abelian scheme over $S$;
\item[(3)] $\iota\colon K\to\End^0(A)$ is an action of $K$ on $A$ up to isogeny satisfying the Kottwitz condition $(\mathrm{KR}_{r})$:
\begin{equation}\label{kottwitzcondi}
\charac\bigl(\iota(a) \mid \Lie A\bigr) = \prod_{\varphi \in \Phi} (T-\varphi(a))^{r_\varphi}
\quad\text{for all}\quad a\in K;
\end{equation} 
\item[(4)] $\lambda$ is a quasi-polarization on $A$ whose Rosati involution satisfies 
  ${\rm Ros}_{\lambda} \bigl(\iota(a)\bigr)=\iota(\bar a)$ for $a\in K$;
\item[(5)] $\bar\eta$ is a $\bK_{\wt G}$-orbit (equivalently, a ${\bK_{\U}}$-orbit, where $\bK_{\wt G}$ acts through its projection $\bK_{\wt G} \to \bK_\U$) of isometries of $\BA_{K, f}/\BA_{F,f}$-hermitian modules 
\begin{equation}\label{eta}
   \eta\colon \wh \RV(A_0, A) \isoarrow \wt V\otimes_K\BA_{K, f}.
\end{equation}
\end{altitemize}
Here we define
\begin{equation}\label{V(A_0,A)}
   \wh \RV(A_0,A) := \Hom_{\BA_{K,f}}\bigl(\wh \RV(A_0), \wh \RV(A)\bigr) ,
\end{equation}
where $\wh \RV(A_0)$ and $\wh \RV(A)$ are the product of rational Tate modules of $A_0$ and $A$. Here $ \wh \RV(A_0,A)$ is endowed with its natural $\BA_{K,f}$-valued hermitian form $h$,
\begin{equation}\label{h_A def}
h(x, y) := \lambda_0\i \circ y^\vee\circ\lambda\circ x\in\End_{\BA_{K, f}}\bigl(\wh \RV(A_0)\bigr)=\BA_{K, f}, \quad x,y \in \wh \RV(A_0,A) ,
\end{equation}
cf. \cite[\S 3]{unitshim}, where $y^\vee \colon \wh\RV(A^\vee) \to \wh\RV(A_0^\vee)$ denotes the adjoint of $y$ with respect to the Weil pairings on $\wh\RV(A) \times \wh\RV(A^\vee)$ and $\wh\RV(A_0) \times \wh\RV(A_0^\vee)$. Furthermore, for any geometric point $\bar s \to S$, the orbit $\bar\eta$ is required to be $\pi_1(S,\bar s)$-stable with respect to the $\pi_1(S,\bar s)$-action on the fiber 
\begin{equation*}
\wh \RV(A_0,A)(\bar s)=\Hom_{\BA_{K,f}}(\wh \RV(A_{0, \bar s}), \wh \RV(A_{\bar s})).
\end{equation*} 

A morphism $(A_0,\iota_0,\lambda_0,A,\iota,\lambda,\bar\eta) \to (A'_0,\iota'_0,\lambda'_0,A',\iota',\lambda',\bar\eta')$ in this groupoid is  given by an  isomorphism $\mu_0\colon (A_0, \iota_0, \lambda_0) \overset{\sim}\longrightarrow (A'_0, \iota'_0, \lambda'_0)$ in $\CA_0^{[\Lambda_0]}(S)$  and a $K$-linear quasi-isogeny $\phi\colon A\to A'$ pulling $\lambda'$ back to $\lambda$ and $\bar\eta'$ back to $\bar\eta$. 
\end{definition}
\begin{proposition}
The moduli problem $\wt\CA_{\bK_{\wt G}, E}$ is representable by a Deligne-Mumford stack $\wt \CA_{\bK_{\wt G}, E}$ over $\Spec E$, and
\begin{equation*}
\wt \CA_{\bK_{\wt G}, E}(\BC)={\rm Sh}_{\bK_{\wt G}}(\wt G, X_{\wt G}) ,
\end{equation*}
compatible with changing $\bK_{\wt G}$ of the form \eqref{decprod-intro}. If $\bK_{\U}$ is small enough, then $\wt\CA_{\bK_{\wt G}, E}$ is a  scheme. 
 
\end{proposition}
\begin{proof}
The representability statement is standard, and the identification of the $\BC$-points is  \cite[Theorem 3.7]{unitshim} (the main point being the Hasse principle for hermitian spaces). For the last statement, we have to show that for a point $x=(A_0,\iota_0,\lambda_0,A,\iota,\lambda,\bar\eta)$ of $\wt\CA_{\bK_{\wt G}, E}$, the automorphism group of $x$ is trivial.
However, an element of this automorphism group defines an element  in an arithmetic subgroup of a unitary group which is compact at all archimedean places of $F$. It is therefore of finite order. By making  $\bK_{\U}$ small enough, its eigenvalues are forced to be congruent to $1$ modulo some big enough integer, hence are equal to $1$. The assertion follows. 
\end{proof}

\subsection{The $p$-integral  RSZ moduli problem}
Let $p$ be a prime number and fix an embedding $\bar\BQ\to \bar\BQ_p$. Let $v_{0}$ be  a $p$-adic place of $F$ which is non-split in $K$ and such that $V_{v_{0}}$ is a non-split $K_{v_{0}}/F_{v_{0}}$-anti-hermitian space, i.e., $\inv_{v_{0}}(V)=-1$. Throughout the paper, we always assume\footnote{In  light of the results of Kirch \cite{Ki}, it should be possible to remove this blanket assumption.} $p\neq 2$ if $v_{0}$ is ramified in $K$. Let $\nu$ be the place of $E$ induced by the embedding $\nu:\bar{\BQ}\rightarrow \bar{\BQ}_{p}$. We assume that the place $v_{0}$ is induced by $\nu$ under the inclusion $F\hookrightarrow E$ given by $\varphi_0$. In the sequel, we consider $F_{v_0}$ and $K_{v_0}$ as subfields of $\bar\BQ_p$ (recall that $\varphi_0:K\to\bar\BQ$ is fixed).

Let us write $\U(\BQ_{p})=\prod_{v\mid p}\U_{v}(\BQ_{p})$, where we denote by $\U_v={\rm Res}_{F_v/\BQ_p}(\U(V_v))$ the corresponding unitary group over $\BQ_{p}$ for $V_v$. We define $\U^{v_{0}}(\BA_{f}):=\U^{v_{0}}(\BQ_{p})\times \U(\BA^p_{f})$ where $\U^{v_{0}}(\BQ_{p})=\prod_{v\neq v_{0}}\U_{v}(\BQ_{p})$. We define $\wt G^{v_{0}}(\BA_{f})=\U^{v_{0}}(\BA_{f})\times Z^\BQ(\BA_{f})$.

In the product decomposition
\begin{equation*}
\bK_{\wt G}= \bK_\U\times \bK^{\circ}_{Z^\BQ},
\end{equation*} 
we fix $\bK^{\circ}_{Z^\BQ}$ and vary the open compact subgroup $\bK_\U$ of $\U(\BA_{f})$ which we write in the form  
\begin{equation*}
{\bK}_\U = {\bK}_{\U, p} \cdot {\bK}^{p}_{\U},
\end{equation*}
where $\bK_{\U}^{p}$ is an open compact subgroup of $\U(\BA^p_{f})$ and $\bK_{\U, p}$ is an open compact subgroup of $\U(\BQ_{p})$.  We impose that $\bK_{\U, p}$ further decomposes as $\bK_{\U, p}=\bK_{\U, v_{0}}\cdot\bK^{v_{0}}_{\U, p}$, where ${\bK}_{\U, p}^{v_{0}}$ is an  open compact subgroup of $\U^{v_{0}}(\BQ_{p})=\prod_{v\neq v_{0}}\U_{v}(\BQ_{p})$. We assume that  ${\bK}_{\U, v_{0}}$ is a maximal open compact and therefore ${\bK}_{\U, v_{0}}= \U_{v_{0}}(\BQ_{p})$, as follows from our assumptions that $ \U_{v_{0}}(\BQ_{p})$ is a  compact group.   To simplify the formulation of the moduli problem, we assume that 
\begin{equation*}
\bK_{\U, p}^{v_{0}}\subset \prod\nolimits_{v\mid p, v\neq v_{0}}\bK_{\U, v}^{\circ}, 
\end{equation*}
where $\bK_{\U, v}^{\circ}$ is the stabilizer of a  fixed  lattice $\Lambda_{v}$ in $V_{v}$ which is almost selfdual (see the Notation section at the end of the Introduction for the meaning of this terminology).  Let ${\bK}^{v_{0}}_\U = {\bK}^{v_{0}}_{\U, p} \cdot {\bK}^{p}_{\U}$ and ${\bK}^{v_{0}}_{\wt{G}} = {\bK}^{v_{0}}_{\U} \times \bK^{\circ}_{Z^{\BQ}}$.

\begin{remark}\label{parornot}
 The subgroup ${\bK}_{\U, v_{0}}$ is a parahoric in $\U_{v_0}(\BQ_p)$ when $K_{v_0}/F_{v_0}$ is unramified; it contains a parahoric with index $2$ when $K_{v_0}/F_{v_0}$ is ramified.
 \end{remark}
 
We are now going to describe a model $\wt\CA_{\bK_{\wt G}}$ of $\wt\CA_{\bK_{\wt G}, E}$ over $\Spec O_{E, (\nu)}$ by formulating a moduli problem over $({\rm LNSch})/O_{E, (\nu)}$.

\begin{definition}\label{def:overOE}
Let $\wt\CA_{\bK_{\wt G}}$ be the category fibered in groupoids over ${\rm (LNSch)}/O_{E, (\nu)}$ which associates to each $O_{E, (\nu)}$-scheme $S$ the groupoid of tuples $(A_0,\iota_0,\lambda_0,A,\iota,\lambda,\bar\eta^p, \bar\eta_p^{v_{0}})$, where
\begin{altitemize}
\item[(1)] $(A_0, \iota_0, \lambda_0)$ is an object of $\CA^{[\Lambda_0]}_0(S)$; 
\item[(2)] $A$ is an abelian scheme over $S$;
\item[(3)] $\iota\colon O_{K, (p)}\to\End_{(p)}(A)=\End(A)\otimes \BZ_{(p)}$ is an action up to prime-to-$p$ isogeny on $A$  satisfying the Kottwitz condition \eqref{kottwitzcondi} on $O_{K, (p)}$; 
\item[(4)] $\lambda\in \Hom_{(p)}(A, A^\vee)$ is a quasi-polarization on $A$ whose Rosati involution satisfies 
  ${\rm Ros}_{\lambda} \bigl(\iota(a)\bigr)=\iota(\bar a)$ for $a\in O_{K, (p)}$;
\item[(5)] $\bar\eta^p$ is a $\bK_{\U}^p$-orbit of isometries of $\BA_{K, f}^p/\BA_{F,f}^p$-hermitian modules 
\begin{equation}\label{eta}
   \eta^p\colon \wh \RV^p(A_0, A) \isoarrow \wt V\otimes_K\BA_{K, f}^p,
\end{equation}
where $\wh \RV^p(A_0, A)=\Hom_{\BA_{K, f}^p}(\wh \RV^p(A_0), \wh \RV^p(A))$ (as usual, a Galois-invariance property is imposed, comp. Definition \ref{def:overE});
\item[(6)] $\bar\eta_p^{v_{0}}$ is a $\bK_{\U, p}^{v_{0}}$-orbit of isometries of $K_{p}^{v_{0}}/F_{p}^{v_{0}}$-hermitian modules 
\begin{equation}\label{eta}
\eta_p^{v_{0}}\colon \wh \RV_p^{v_{0}}(A_0, A) \isoarrow \wt V\otimes_K K_p^{v_{0}}.
\end{equation}
\end{altitemize}
Here  the notation is as follows. Under the decomposition of $p$-divisible groups induced by the action of $O_{F}\otimes\BZ_p \cong \prod_{v\mid p} O_{F,v}$, we have 
\begin{equation*}\label{decofpdivgp}
   A_0[p^\infty] =  \prod\nolimits_{v \mid p} A_0[v^\infty],\quad \text{resp.}\quad  A[p^\infty]=\prod\nolimits_{v\mid p} A[v^\infty] .
\end{equation*}
 We set 
\begin{equation}\label{primetov}
A_0[p^\infty]^{v_{0}}=\prod\nolimits_{v\mid p, v\neq v_{0}}A_0[v^\infty],\quad \text{resp.}\quad A[p^\infty]^{v_{0}}=\prod\nolimits_{v\mid p, v\neq v_{0}}A[v^\infty] .
\end{equation}
Furthermore, 
\begin{equation*}
\wh \RV_p^{v_{0}}(A_0, A):=\underline{\Hom}_{O_{K_{p}^{v_{0}}}}(A_0[p^\infty]^{v_{0}}, A[p^\infty]^{v_{0}})\otimes _{O_{K_{p}^{v_{0}}}} K_p^{v_{0}}
\end{equation*}
is the $p$-adic \'etale sheaf of \cite[Definition 7.4.1]{KRZI}.  More precisely, \eqref{eta}
 is an equivalence class of
isogenies of (pro-)\'etale sheaves 
\[\underline{\Hom}_{O_{K_p^{v_0}}}(A_0[p^\infty]^{v_{0}}, A[p^\infty]^{v_{0}})\to\CL,
\]
where $\CL\subset \wt V\otimes_K K_p^{v_0}$ is a variable
lattice. We refer to the datum \eqref{eta} as a \it{CL}-level structure on $A$.

We impose the following further conditions on the above tuples:
\begin{altenumerate}
\item The action of $O_{K, (p)}$ on $\Lie A$ satisfies the Eisenstein condition $({\rm EC}_r)$ relative to $r$, cf. \cite[\S 2.2]{KRZI};

\item\label{sglob cond i}  For every $v\mid p$, the  polarization  of $p$-divisible groups induced by $\lambda$ 
\begin{equation*}
\lambda_{v} \colon A[v^\infty] \to A^\vee[v^\infty] \cong A[v^\infty]^\vee 
\end{equation*}
is principal, except when $v$ is unramified and $\inv(V_v)=-1$, in which case  $\ker\lambda_{v}$  is contained in $A[\pi_{v}^\infty][\pi_v]$ of rank $\#(\Lambda_{v}^\vee/ \Lambda_{v})$; 

\item\label{sglob cond ii} We require that at every geometric point $\bar s$ of $S$ the following \emph{sign condition} holds for every non-split place $v\mid p$,
\begin{equation}\label{condsign}
   \inv^r_{v}(A_{0, \bar s},\iota_{0, \bar s},\lambda_{0, \bar s},A_{\bar s},\iota_{\bar s},\lambda_{\bar s})=\inv_{v}(\wt V),
\end{equation}
where the right-hand denotes the Hasse invariant of the hermitian space $\wt V$ at $v$. Here  the left-hand side is the sign from \cite[App. A]{RSZ}.  To make the connection with the definition of \cite[\S 3]{RSZ}, we point out that $\inv_{v}(V)=\inv_{v}(\wt V)$.
\end{altenumerate}
A morphism  $(A_0, \iota_0, \lambda_0, A, \iota, \lambda, \bar\eta^p, \bar\eta_p^{v_{0}}) \to (A_0', \iota_0', \lambda_0', A', \iota', \lambda', \bar\eta'^p, \bar\eta_p^{\prime v_{0}})$ in this groupoid is given by an isomorphism $\phi_0\colon (A_0,\iota_0,\lambda_0) \isoarrow (A_0',\iota_0',\lambda_0')$ in $\CA^{[\Lambda_0]}_{0}(S)$ and an $O_{K,(p)}$-linear quasi-isogeny $\phi\colon A \to A'$ inducing an isomorphism $A[p^\infty] \isoarrow A'[p^\infty]$, pulling $\lambda'$ back to $\lambda$, and pulling $\bar\eta^{\prime p}$ back to $\bar\eta^p$ and pulling $\bar\eta_p^{\prime v_{0}}$ back to $\bar\eta_p^{v_{0}}$. 
\end{definition}

Let $(\bar{V}, \bar{\varkappa})$ be the nearby anti-hermitian space relative to $w_{0}$ and $v_{0}$. This means $\bar{V}$ is isomorphic to $V$ at all places except for $w_{0}$ and $v_{0}$ and such that it has signature $(2,0)$ at $w_{0}$ and $\inv_{v_{0}}(\bar{V})=+1$. Let $J_{\U}=\mathrm{Res}_{F/\BQ}\U(\bar{V})$ and let $J={G}(\bar{V}, \bar{\varkappa})$, the  unitary similitude group  of $(\bar{V}, \bar{\varkappa})$ with rational similitude factor. Then we define 
\begin{equation}\label{wtJ-group}
\wt{J}=J_{\U}\times Z^{\BQ}= J\times_{\BG_{m}} Z^{\BQ}.
\end{equation}
Then $\U(\bar{V})$ is anisotropic at the archimedean place $w_{0}$ and quasi-split  at $v_{0}$, and  locally coincides 
with $\U(V)$ at all places $\neq v_{0}, w_{0}$ of $F$. Then  $\wt{J}$ is an inner form of $\wt{G}$. We write 
\[J_{\U}(\BQ_{p})=\prod\nolimits_{v\mid p}J_{\U,v}(\BQ_{p})=\prod\nolimits_{v\mid p}\U(\bar{V}_{v})(F_{v}).
\]
 We fix an identification of $J_{\U,v_{0}, \mathrm{ad}}(\BQ_{p})\simeq \PGL_2(F_{v_{0}})$.

Now we can formulate the first main theorem of this paper. 

\begin{theorem}\label{maintilde}
Assume that $\bK_\U^p$ is small enough. 
\begin{altenumerate}
\item The functor $\wt\CA_{\bK_{\wt G}}$ is representable by a projective flat $O_{E, (\nu)}$-scheme $\wt\CA_{\bK_{\wt G}}$ over $\Spec O_{E, (\nu)}$. Its generic fiber is identified with $\wt\CA_{\bK_{\wt G}, E}$, which is the canonical model of the Shimura variety ${\rm Sh}_{\bK_{\wt G}}(\wt G, X_{\wt G})$. 
\item Let $(\wt\CA_{\bK_{\wt G}})^\wedge$ be the formal completion of $\wt\CA_{\bK_{\wt G}}$ along its special fiber, which is a formal scheme over $\Spf O_{E_\nu}$. Then there exists an isomorphism of formal schemes over $\Spf O_{\breve {E}_\nu}$,
\begin{equation*}
(\wt\CA_{\bK_{\wt G}})^\wedge\times_{{\rm Spf}\,O_{E_{\nu}}}{\rm Spf}\, O_{\breve {E}_\nu}  \simeq   \wt{J}(\BQ)\bs\big[\big(\widehat{\Omega}_{F_{v_{0}}} \times_{{\rm Spf}\,O_{F_{v_{0}}}}{\rm Spf}\, O_{\breve {E}_\nu}\big) \times
 \wt G(\BA_f)/{\bK}_{\wt G}\big]  \, .
\end{equation*}
The quotient on the right-hand side is defined using the composition of maps $$\wt J(\BQ)\to J_\U(\BQ)\to J_{\U, v_0}(\BQ_p)\to  J_{\U, v_0, {\rm ad}}(\BQ_p)$$ and an identification of the adjoint group $J_{ \U,  v_{0},{\rm ad}}(\BQ_{p})$  with $\PGL_2(F_{v_{0}})$ and through an action of  $\wt{J}(\BQ)$ on $\wt G(\BA_f)/{\bK}_{\wt G}$. For varying $\bK^{v_{0}}_{\wt G}$, this isomorphism is compatible with the action of $\wt{G}^{v_{0}}(\BA_{f})$  through Hecke correspondences on both sides. 
\end{altenumerate}
Furthermore, the natural descent datum down to $O_{E_{\nu}}$ on the left-hand side is given on the right-hand side by 
\begin{equation*}
(\xi, h)\rightarrow (\omega_{\Omega, {E_{\nu}}}(\xi), \wt{w}h), \quad  h\in \wt G(\BA_f).
\end{equation*}
 Here $\omega_{\Omega, {E_{\nu}}}$  is the natural descent datum down to $O_{E_\nu}$ on the first factor and  $\wt{w}={\frak r}_{\Phi^+, E_\nu}(\varpi_\nu)$ is the \emph{special element} in the center of $\wt{G}(\BQ_{p})$,  which is considered as an element in $\wt G(\BA_f)$ ({special} relative to $(\Phi^{+}, E_{\nu})$ for the uniformizer $\varpi_{\nu}$, as in Construction \ref{consemil}).
\end{theorem}
\begin{remark}
As in \cite[Theorem  7.3.3]{KRZI}, one can characterize the relative curve $\wt\CA_{\bK_{\wt G}}$ over  $O_{E, (\nu)}$  as  the unique  stable relative curve in the sense of Deligne-Mumford with generic fiber $\wt\CA_{\bK_{\wt G}, E}$. 
\end{remark}

From Theorem \ref{maintilde}  we obtain the analogue of Cherednik's $p$-adic uniformization theorem in the generic fiber.
\begin{corollary}
There is an $\wt{G}^{v_{0}}(\BA_{f})$-equivariant isomorphism of rigid analytic spaces
\begin{equation*}
(\wt \CA_{\bK_{\wt G}, E}\times_{\Spec {E}}\Spec {\breve{E}}_\nu)^{\rm rig}  \simeq   \wt{J}(\BQ)\bs\big[\big({\Omega}_{F_{v_{0}}} \times_{{\rm Sp}\, {F_{v_{0}}}} {{\rm Sp}\, {\breve{E}}_\nu}\big) \times
 \wt G(\BA_f)/{\bK}_{\wt G}\big]  \, 
\end{equation*}
over ${\rm Sp}\, {\breve{E}}_\nu$ with compatible descent data. Here   the left-hand side  is equipped with the natural descent datum  down to $E_\nu$ and on the right-hand side the descent datum is given by 
\begin{equation*}
(\xi, h)\rightarrow (\omega_{\Omega, {E_{\nu}}}(\xi), \wt{w}h),  \quad  h\in \wt G(\BA_f),
\end{equation*}
where $\wt{w}$ is the special element  in the center of $\wt{G}(\BQ_{p})$,  which is considered as an element in $\wt G(\BA_f)$ (special relative to $(\Phi^{+}, E_{\nu})$ for the uniformizer $\varpi_{\nu}$,as in Construction \ref{consemil}).
\end{corollary}
\begin{remark}
We therefore see that $\wt{\CA}_{\bK_{\wt G}}$ has $p$-adic uniformization at all places $\nu$ of $E$ over $v_{0}$. This is quite remarkable, when compared to the $p$-adic uniformization theorem  for  fake unitary groups of \cite{RZ}. Indeed, to simplify let us take $F=\BQ$  in \emph{loc. cit.}. Then for $n\geq 3$, the reflex field $E$ is a quadratic extension of $\BQ$ in which $p$ is split. At one of the two places of $E$ above $p$, the Shimura variety has $p$-adic uniformization; at the other place the Shimura variety does not have $p$-adic uniformization (the special fiber contains more than one isogeny class).  
\end{remark}
As usual, this global uniformization theorem is a consequence of a corresponding local result, which is the conjunction of Proposition \ref{LS-weil-special} and \ref{LS-weil-banal}.

\subsection{The Shimura variety for the unitary group}\label{ss:staU} Our second main result concerns  the Shimura curve associated to the unitary group $\U$ and the Shimura datum ${X}_{\U}$ given by  the  conjugacy class of homomorphisms
\begin{equation}\label{shi-da}
h_{\U}\colon \Res_{\BC/\BR}(\BG_m)\longrightarrow\U_{\BR} 
\end{equation}
which we now define. Fix a CM type $\Phi^+$ for $K/F$ with $\varphi_0\in\Phi^+$. We choose for each $\varphi\in \Phi^+$ a $\BC$-basis of $V\otimes_{K, \varphi}\BC$  such that the anti-hermitian form is given by $\diag(i\cdot 1_{r_\varphi}, (-i)\cdot 1_{r_{\bar\varphi}})$. Then  we can write 
\begin{equation*}\label{shi-daprod}
\U\otimes_\BQ\BC=\prod\nolimits_{\Phi^+} \GL_2/\BC . 
\end{equation*}
 Correspondingly $h_{\U}=(h_\varphi)_{\varphi\in\Phi^+}$. We set
\begin{equation}\label{shi-daprod}
h_\varphi(z)=
\begin{cases}
\begin{aligned}
\diag(1, z/\bar z),  &\quad \varphi=\varphi_0 \\
1  , &\quad \varphi\neq\varphi_0.\\
\end{aligned}
\end{cases}
\end{equation}
The associated reflex field is $E(\U, X_{\U})=K$ (embedded via $\varphi_0$ in $\bar\BQ$). 
For an open compact subgroup $\bK_{\U} \subset \U(\BA_f)$, there is a Shimura variety ${\rm Sh}_{\bK_{\U}}(\U, X_{\U})$ over $K$, 
whose complex points are given by 
\begin{equation*}
{\rm Sh}_{\bK_{\U}}(\U, X_{\U})(\BC) \simeq \U(\BQ) \bs [\Omega_\BR \times \U(\BA_f)/{\bK_{\U}}].
\end{equation*}
Note that this Shimura variety is of abelian type but not of PEL type.   

We continue to assume that the place $v_0$ of $F$ remains prime in $K$ and that the anti-hermitian space $V_{v_0}$ is non-split.  Assume that $\bK_{\U}$ is of the form $\bK_{\U}=\bK_{\U}^p\cdot\bK_{\U, p}^{v_0}\cdot\bK_{\U, v_0}$, where $\bK_{\U, v_0}=\U_{v_0}(\BQ_p)$  (recall that, by our assumption, this is a compact group), where $\bK_{\U, p}^{v_0}=\prod_{v\neq v_{0}}\bK_{\U, v}\subset\prod_{v\neq v_{0}}\U_{v}(\BQ_{p})$ is a product of parahoric stabilizers of the fixed almost selfdual lattice $\Lambda_v$, and where $\bK_{\U}^{p}\subset \U(\BA_{f}^p)$ is a sufficiently small open compact subgroup. We consider the \emph{canonical integral model} $\CA_{\bK_{\U}}$ of this Shimura variety of abelian type over $O_{K_{v_{0}}}$, in the sense of   \cite{PR}.

 \begin{remark}
 The canonical integral model is known to exist when $v_0$ is unramified and $p\neq 2$, cf. \cite{DY}. When $v_0$ is ramified or $p=2$, we will assume the existence of the canonical integral model.  We refer to Remark \ref{cani} for a more thorough discussion.  
 \end{remark}

We introduce the following finite extension of $K$ contained in $\bar\BQ$,
\begin{equation}\label{defE}
E(\varphi_0)=\bigcap\nolimits_{\varphi_0\in \Phi^{\prime, +}}E_{\Phi^{\prime,+}}K,
\end{equation}
(intersection of the join with $K$ of the reflex fields of all CM-types $ \Phi^{\prime, +}$ of $K$ containing $\varphi_0$). 
\begin{remark}\label{newfield}
If $F=\BQ$, then $E(\varphi_0)=K$. Indeed, then $ \Phi^{\prime, +}= \Phi^{ +}$ is unique and $E_{\Phi^{ +}}K=K$. If $F$ is a quadratic extension, then there is another possibility for $ \Phi^{\prime, +}$ besides  $\Phi^{ +}$; then $ E_{\Phi^{\prime, +}}K= E_{\Phi^{ +}}K$ is a quadratic extension of $K$ and hence $E(\varphi_0)$ is a quadratic extension of $K$. 
\end{remark}
Here now is the second main result of this paper. 
\begin{theorem}\label{unifor-ab-intro}
Let $(\CA_{\bK_{\U}})^\wedge$ be the formal completion of $\CA_{\bK_{\U}}$ along its special fiber, which is a formal scheme over $\Spf O_{K_{v_{0}}}$. Then there exists an isomorphism of formal schemes over $\Spf O_{E(\varphi_0)_{\nu}}$,
\begin{equation*}
(\CA_{\bK_{\U}})^\wedge\times_{\Spf O_{K_{v_{0}}}} \Spf O_{E(\varphi_0)_{\nu}} \simeq   {J}_{\U}(\BQ)\bs\big[\big(\widehat{\Omega}_{F_{v_{0}}} \times_{\Spf O_{F_{v_{0}}}} \Spf O_{E(\varphi_0)_{\nu}}\big) \times
 \U(\BA_f)/\bK_{\U}\big] .
\end{equation*}
The quotient on the right hand side is defined using an identification of the adjoint group ${J}_{\U, {v_{0}}, {\rm ad}}(\BQ_{p})$ with $\mathrm{PGL}_2(F_{v_{0}})$ and through an action of ${J}_{\U}(\BQ)$ on $\U(\BA_f)/{\bK}_{\U}$.  For varying $\bK_{\U}^{v_{0}}$, this isomorphism is compatible with the action of $\U^{v_{0}}(\BA_{f})$ through Hecke correspondences on both sides.
\end{theorem}
We note that the extension of scalars from $K$ to $E(\varphi_0)$ is a result of our proof, in the optimal result it should not appear.
\begin{conjecture}
The isomorphism in Theorem \ref{unifor-ab-intro} is the base change by $\Spf O_{E(\varphi_0)_\nu}\to \Spf O_{K_{v_{0}}}$ of an isomorphism
\begin{equation*}
(\CA_{\bK_{\U}})^\wedge \simeq   {J}_{\U}(\BQ)\bs\big[\big(\widehat{\Omega}_{F_{v_{0}}} \times_{\Spf O_{F_{v_{0}}}} \Spf O_{{K_{v_0}}}\big) \times
 \U(\BA_f)/\bK_{\U}\big] .
\end{equation*}
\end{conjecture}
We also extend this result to the case when $\bK_{\U, p}^{v_0}$ is not necessarily parahoric by constructing an integral model $\CA_{\bK_{\U},O_{E(\varphi_0)_{\nu}}}$ over $O_{E(\varphi_0)_{\nu}}$ and proving a corresponding $p$-adic uniformization theorem in this context, cf. Theorem \ref{unifor-para}. 

  As usual, the global results about integral models of Shimura curves follow from corresponding local results. We formulate as follows the main local result behind the previous global result, which we find quite striking. 
\begin{theorem}\label{Thm;locU}
Let $F/\BQ_p$ be a finite extension. Let $K/F$ be a quadratic extension, and let $\U'$ be an anisotropic $K/F$-hermitian group of size $2$ over $F$. Fix an embedding $\varphi_0: K\to\bar\BQ_p$. Let $\U=\Res_{F/\BQ_p}(\U')$ and let $\mu$ be defined as in the global case using a fixed local CM type $\Phi^+$ with $\varphi_0\in\Phi^+$, with associated reflex field $K\subset\bar\BQ_p$. Let $E(\varphi_0)$ be the extension of $K$ in $\bar\BQ_p$, defined as in the global case. Also, let $\bfK_{\U}=\U(\BQ_p)$ (a quasi-parahoric subgroup). Let $b$ be a representative of the unique element $[b]\in B(\U, \mu^{-1})$. Let $\mathcal{M}^{\mathrm{int}}_{\mu, b, \bfK_{\U}}$ be the associated integral local Shimura variety, cf. \cite{PRloc}. Then $\mathcal{M}^{\mathrm{int}}_{\mu, b, \bfK_{\U}}$ is a formal scheme over $O_K$ and there is an isomorphism 
\[
\mathcal{M}^{\mathrm{int}}_{\mu, b, \bfK_{\U}}\times_{\Spf O_{{K}}} \Spf O_{E(\varphi_0)} \simeq \wh\Omega_F\times_{\Spf O_{F}} \Spf O_{E(\varphi_0)}. 
\]
\end{theorem}
As in the global case, we conjecture that this isomorphism comes by base change from an isomorphism 
\[
\mathcal{M}^{\mathrm{int}}_{\mu, b, \bfK_{\U}} \simeq \wh\Omega_F\times_{\Spf O_{F}} \Spf O_{K}. 
\]
\begin{remark}
We therefore obtain an explicit identification of the  integral local Shimura variety $\mathcal{M}^{\mathrm{int}}_{\mu, b, \bfK}$ for the case of the unitary group of an anisotropic hermitian space of dimension two over $F=\BQ_p$:
\begin{equation}
\mathcal{M}^{\mathrm{int}}_{\mu, b, \bfK}=\wh\Omega_{\BQ_p}\times_{\Spf \BZ_p}\Spf O_K .
\end{equation}
Previously, there were only two (indecomposable) cases, when an integral local Shimura variety $\mathcal{M}^{\mathrm{int}}_{\mu, b, \bfK}$ was explicitly known:   the Lubin-Tate case corresponding to $G=\GL_n$ and $\mu=(1, 0,\ldots, 0)$ and the unique basic element $[b]\in B(G, \mu^{-1})$ and the maximal compact subgroup $\bfK=\GL_n(\BZ_p)$, and the Drinfeld case  corresponding to $G=D^\times_{1/n}$ and $ \mu=(1, 0,\ldots, 0)$ and the unique basic element $[b]\in B(G, \mu^{-1})$ and the maximal compact subgroup $\bfK=O^\times_{D_{1/n}}$. In these cases, the integral local Shimura varieties are of the form
\begin{equation*}
\mathcal{M}^{\mathrm{LT},\mathrm{int}}_{\mu, b, \bfK}\times_{\Spf \BZ_p} \Spf \breve\BZ_p =\Spf \breve \BZ_p[[T_1,\ldots,T_{n-1}]]\times\BZ
\end{equation*}
and 
\begin{equation*}
\mathcal{M}^{\mathrm{Dr},\mathrm{int}}_{\mu, b, \bfK}\times_{\Spf \BZ_p} \Spf \breve \BZ_p =(\wh\Omega^n_{\BQ_p}\times_{\Spf \BZ_p}\Spf \breve \BZ_p)\times\BZ 
\end{equation*}
with explicit Weil descent data down to $\BZ_p$. Note that in these cases, the  integral local Shimura varieties have infinitely many connected components as formal schemes over  ${\breve \BZ_p}$.

Note that, by definition, integral local Shimura varieties are moduli spaces of shtukas, cf. \cite{Sch}. However, in order to prove the explicit identifications of the above integral local Shimura varieties, we have to pass through moduli spaces of $p$-divisible groups. This is especially true of the Drinfeld case, in which case we have to appeal to \cite{Dr} (see \cite{BC}, \cite{RZ} for expositions of Drinfeld's paper).  It would be interesting to prove the explicit identifications directly from the definitions. Possibly  the papers by S.~Bartling \cite{Ba} and A.~Vanhaecke \cite{Van} are a step in this direction.

\end{remark}

\section{CM-types and CM-triples}\label{s:CM}

In this section we recall some  notions that we will use throughout the paper. 

\subsection{Special and banal local CM-types}\label{ss:specialbanal} 

We fix a prime number $p$ and an algebraic closure $\bar \BQ_p$  
of $\BQ_p$. Let $F$ be a finite field extension of $\BQ_p$,  with residue class field $\kappa_F$ of cardinality $q$. 
We set $d=[F:\BQ_p]$, 
 $f = [\kappa_F : \mathbb{F}_p]$ and define $e$ through $d = ef$.

We  let $K/F$ be  an \'etale algebra of degree $2$. We
denote the non-trivial automorphism of $\Gal(K/F)$ by $a \mapsto
\bar{a}$.  In the case where $K/F$ is a ramified extension of local fields
(ramified case) we choose a prime element $\Pi \in O_K$ 
such that
$\bar{\Pi} = - \Pi$. Then $\pi = -\Pi^2$ is a prime element 
of $F$. In the case where $K/F$ is an unramified extension of local
fields (unramified case) or where $K \cong F \times F$ (split case) we
choose a prime element $\pi \in F$ and we set $\Pi = \pi$. 

Let $m\in\{1,2\}$. Let $r$ be a generalized local CM-type of rank $m$ relative to
$K/F$ in the sense of \cite[Definition 2.1]{KRnew}. Let $\Phi=\Hom_{\text{$\BQ_p$-Alg}}(K, \bar\BQ_p)$, then $r$ is a function 
\begin{equation}\label{loc-CM-type}
r: \Phi\lra \BZ_{\geqslant 0}, \qquad \,\, \varphi \mapsto r_\varphi,
\end{equation}
such that $r_{\varphi}+r_{\bar\varphi} =m$ for all $\varphi\in \Phi$. Here $\bar\varphi(a) = \varphi(\bar a)$ (we recall that
 $a\mapsto \bar a$ is the non-trivial automorphism of $K$ over $F$). In this paper, we refer to a  generalized local CM-type of rank $m$ relative to
$K/F$ simply as a local CM-type of rank $m$. 
The reflex field $E_r$ of $r$ is the subfield of $\bar\BQ_{p}$ fixed by 
$$\Gal(\bar\BQ_{p}/E_r):=\{ \tau\in \Gal(\bar\BQ_{p}/\BQ_{p})\mid r_{\tau\varphi} = r_\varphi, \ \forall \varphi\}.$$
Let  $O_{E_r}$ be the ring of integers of $E_{r}$.

When we fix an embedding $\varphi_0:F\to \bar \BQ_p$, we denote by $\varphi_0, \bar\varphi_0$ the two extensions of $\varphi_0$ to $K$ by abuse of notations.  

\begin{definition}\label{loc-special-type} 
A local CM-type\index{special local CM-type} $r$ of rank $2$  is called {\it special} relative to  $\varphi_0: F\to \bar\BQ_p$ if $K/F$ is a field extension and 
$$
r_{\varphi_0}=r_{\bar\varphi_0}=1, \text{ and } r_\varphi\in \{ 0, 2 \}, \quad \text { for all } \varphi\in \Phi\setminus \{\varphi_0, \bar\varphi_0 \} . 
$$
\end{definition}
 If $ r_\varphi\in \{ 0, 2 \}$,  for all $\varphi\in \Phi$, then we call $r$ a {\it banal} CM-type. A local CM-type\index{special local CM-type} of rank $1$ will always be regarded as a banal CM-type. There is a bijection between local CM-types of rank $1$ and classical CM-types, via $\Phi^+(r)=\{ \varphi\in\Phi\mid r_\varphi=1\}$.

\subsection{The Kottwitz and the Eisenstein conditions}\label{ss:kotteis}
Let $S$ be an $O_{E_{r}}$-scheme, and $\CL$ be a locally free
$\CO_S$-module, equipped with an action
\begin{equation*}
\iota: O_K \longrightarrow \End_{\mathcal{O}_S} \mathcal{L}
  \end{equation*}
of $O_K$.

We say that $(\mathcal{L}, \iota)$ satisfies the Kottwitz condition $({\rm KC}_r)$ relative to $r$ if the identity of polynomials with
coefficients in $\CO_S$ holds,    
\begin{equation}\label{signature.condition}
  {\rm char} (T, \iota (a) \vert \CL) = i\big(\underset{\varphi \in
    \Phi} \prod (T-\varphi(a))^{r_\varphi}\big),\quad \text{ for all
    $a\in O_K$ }, 
\end{equation}
where $i:O_{E_r}\longrightarrow \mathcal O_S$ is the structure homomorphism.  We say $(\mathcal{L}, \iota)$ satisfies the rank condition $({\rm RC}_r)$ if the condition in \cite[(2.2.5)]{KRZI} is satisfied  and say $(\mathcal{L}, \iota)$ satisfies the Eisenstein condition $({\rm EC}_r)$ if $(\mathcal{L}, \iota)$ satisfies $({\rm RC}_r)$ and the conditions in \cite[(2.2.12)]{KRZI} are fulfilled. Since we will not use the explicit form of the definition, we refer  to \cite{KRZI} for a thorough discussion.

\subsection{Local  CM-pairs and CM-triples}\label{ss:loctriples}
 Let $r$ be a local CM-type of rank $m\in\{1,2\}$, with reflex field $E_{r}$. Let $S$ be an $O_{E_{r}}$-scheme such that $p$ is locally nilpotent i.e. a scheme over $\Spf O_{E_{r}}$.
A {\it local CM-pair of type} $r$ with respect to $K/F$ over $S$ is a pair $(X, \iota)$ where
$X$ is a $p$-divisible group over $S$ of height $2md$ and dimension $md$ and  
$\iota$ is an $\mathbb{Z}_p$-algebra homomorphism 
\begin{equation*}
  \iota: O_K \longrightarrow \End X 
\end{equation*}
such that the rank condition $({\rm RC}_r)$ is satisfied for the
induced action of $O_K$ on $\Lie X$. In the split case
$O_K = O_F \times O_F$ we require moreover that in the induced
decomposition $X = X_1 \times X_2$ each factor is a $p$-divisible
group of height $md$. We say the pair $(X, \iota)$ satisfies the Kottwitz condition $({\rm KC}_{r})$ resp. the Eisenstein condition $({\rm EC}_{r})$ if the 
induced action of $O_K$ on $\Lie X$ satisfies the Kottwitz condition $({\rm KC}_{r})$ resp. the Eisenstein condition $({\rm EC}_{r})$.

To each local CM-pair $(X, \iota)$ we define the
{\it conjugate dual} $(X^{\vee}, \iota^{\wedge})$. Here $X^{\vee}$ is
the dual $p$-divisible group of $X$ but we change the action dual to $\iota$
by the conjugation of $K/F$, i.e.,  $\iota^{\wedge}(a)=\iota^\vee(\bar a)$. We
will denote the conjugate dual simply by $X^{\wedge}$. The  conjugate dual of a CM-pair $(X, \iota)$ of type $r$ is again a local CM-pair of type $r$. If $(X, \iota)$ satisfies the Kottwitz condition $({\rm KC}_r)$, resp., the Eisenstein conditions $({\rm EC}_r)$, then so does its conjugate dual. These assertions follow from \cite[Lemma 2.3.2]{KRZI}.

The notion of a   {\it local CM-triple of type} $r$ with respect to $K/F$ over $S$ was
introduced in \cite{KRnew}.  This is a triple $(X, \iota, \lambda)$,
where $(X, \iota)$ is a local CM-pair of type $r$ over $S$ and
$\lambda: X\longrightarrow X^{\wedge}$ is an anti-symmetric isogeny (also
called a {\it polarization})  such that the corresponding Rosati involution
induces the non-trivial automorphism on $K/F$. In this paper, we will refer to a local CM-triple of type $r$ with respect to $K/F$ as a {\it local CM-triple of type $(K/F, r)$}. We say the local CM-triple $(X, \iota, \lambda)$ satisfies the Kottwitz condition $({\rm KC}_{r})$ resp. the Eisenstein condition $({\rm RC}_{r})$, if the local CM-pair $(X, \iota)$ satisfies the Kottwitz condition $({\rm KC}_{r})$ resp. the Eisenstein condition $({\rm RC}_{r})$. 
\begin{definition}\label{a-p-p}
We say that  the local CM-triple $(X, \iota, \lambda)$ is \emph{principally polarized} if the homomorphism $\lambda \colon X\to X^{\wedge}$ is an isomorphism. We say that $(X, \iota, \lambda)$ is \emph{almost principally polarized} if $K/F$ is unramified and the kernel of $\lambda_{X}$ is contained in $X[\pi]$ and has order $q^2$.
\end{definition}
\subsection{Global and semi-local CM-triples}\label{ss:semiglob}
 Let $K/F$ now be a CM-field. Let $m\in\{1,2\}$. Let $r$ be a {\it generalized CM-type of rank} $m$, i.e., a function 
\begin{equation}\label{GCM}
r: \Phi\lra \BZ_{\geqslant 0}, \qquad \,\, \varphi \mapsto r_\varphi,
\end{equation}
such that $r_{\varphi}+r_{\bar{\varphi}} =m$ for all $\varphi\in \Phi$.  As in the local case, we will simply refer to a generalized CM-type of rank of $m$ as a CM-type of rank $m$. 

The corresponding reflex field $E_r$ is the subfield of $\bar{\BQ}$ fixed by 
$$\Gal(\bar{\BQ} /E_r):=\{ \tau\in \Gal(\bar{\BQ}/\BQ)\mid r_{\tau\varphi} = r_\varphi, \ \forall \varphi\}.$$
Let  $O_{E_r}$ be the ring of integers of $E_r$. 
\begin{definition}\label{glob-special-type}
Let $w$ be an archimedian place of $F$. We say that   a CM-type of rank $2$ is {\it special} with respect to $w$ if  for the extensions $\{\varphi_0, \bar{\varphi}_0\}$ of $w$ we have  $r_{\varphi_0} = r_{\bar{\varphi}_0} = 1$ and  such that  $r_\varphi\in \{0, 2\}$ for $\varphi\notin \{\varphi_0, \bar\varphi_0\}$.
\end{definition}
As in the local case, we can speak of \emph{banal} CM-types. We always regard a CM-type of rank $1$ as a banal CM-type. There is a bijection between  CM-types of rank $1$ and classical CM-types, via $\Phi^+(r)=\{ \varphi\in\Phi\mid r_\varphi=1\}$.

Let $p$ be a prime and let $v\mid p$ be the $p$-adic places of $F$.  We fix an embedding $\nu: 
\bar{\BQ}\rightarrow \bar{\BQ}_{p}$.  Then we obtain a decomposition 
\begin{equation*}\label{Phi-decom}
\Phi=\Hom_{\text{$\mathbb{Q}$-Alg}}(K, \bar{\mathbb{Q}})= \bigsqcup\nolimits_{v\mid p}\Hom_{\text{$\mathbb{Q}_{p}$-Alg}}(K_{v}, \bar{\mathbb{Q}}_{p})= \bigsqcup\nolimits_{v\mid p}\Phi_{v} .
\end{equation*}
 We denote by $r_{v}$ the restriction of $r$ to $\Phi_{v}$ which is a local CM-type of rank $m $ for $\Phi_{v}$. If $r$ is a CM type of rank $1$ which gives rise to a classical CM type $\Phi^{+}=\Phi^{+}(r)$, then $r_{v}$ defines a classical local CM type 
\begin{equation}\label{Phiplusv}
\Phi^{+}_{v}=\Phi^{+}_{v}(r_{v})\subset \Phi_{v}
\end{equation}
such that $\Phi^{+}=\bigsqcup\nolimits_{v\mid p}\Phi^{+}_{v}$.
Let $\nu$ also denote the place of $E_r$ over $p$ induced by the embedding $\nu$. For each $v$, the reflex field  of $r_v$ is a subfield $E_r(v)$ of $E_{r, \nu}$. 

There is the obvious  notion of a \emph{semi-local CM-triple} $(X, \iota, \lambda)$ of type $r$ with respect to
$K \otimes \mathbb{Q}_p/ F \otimes \mathbb{Q}_p$.
Let $S$ be an $O_{E_{r, \nu}}$-scheme such that $p$ is locally nilpotent i.e.
a scheme over $\Spf O_{E_{r,\nu}}$. 
Let  $(X, \iota)$ be a $p$-divisible
group over $S$ with an action
\begin{equation*}
  \iota: O_K \otimes \mathbb{Z}_p \longrightarrow \End X.  
\end{equation*}
The decomposition
\begin{equation*}
  O_F \otimes \mathbb{Z}_p = \prod\nolimits_{v\mid p} O_{F_{v}}
\end{equation*}
induces the decomposition $X = \prod_{v\mid p} X_{v}$ of the $p$-divisible group $X$. Let $\lambda$ be a polarization of $X$ which induces the conjugation
on $K/F$. Then the decomposition extends to 
\begin{equation}\label{Unif1e}
  (X, \iota, \lambda) =
  \prod\nolimits_{v\mid p} (X_{v}, \iota_{v}, \lambda_{v}). 
\end{equation}
We call $(X, \iota, \lambda)$ a semi-local CM-triple of type $r$ with respect to $K \otimes \mathbb{Q}_p/ F \otimes \mathbb{Q}_p$ if each 
$(X_{v}, \iota_{v}, \lambda_{v})$ is a local
CM-triple of type $r_{v}$ with respect to $K_{v}/F_{v}$.  In the following, we will refer to a semi-local CM-triple $(X, \iota, \lambda)$ of type $r$ with respect to $K \otimes \mathbb{Q}_p/ F \otimes \mathbb{Q}_p$ simply as a {\it CM-triple of type $(K \otimes \mathbb{Q}_p / F \otimes \mathbb{Q}_p, r)$} over $S$. A CM-triple of type $(K\otimes \BQ_{p}/F\otimes\BQ_{p}, r)$ is said to satisfy the conditions $(\mathrm{KR}_{r})$ and $(\mathrm{EC}_{r})$ if each local CM-triple $(X_{v},\iota_{v},\lambda_{v})$ satisfies the conditions $(\mathrm{KR}_{r_{v}})$ and $(\mathrm{EC}_{r_{v}})$.

\section{Formal moduli space of toric CM-triples}\label{s:Cmod}

\subsection{Complex multiplication} Let $K/F$ be a CM-field. Let $\Phi^{+} \subset \Phi=\Hom_{\text{$\mathbb{Q}$-Alg}}(K, \mathbb{C})$
be a classical CM-type. We denote the reflex field of $\Phi^{+}$ by $E$. We define an algebraic
torus $\T$ over $\mathbb{Q}$, with $\BQ$-valued points given by
\begin{equation*}
\T(\mathbb{Q}) = \{a \in K^{\times} \; | \; a \bar{a} \in \mathbb{Q}^{\times} \}. 
\end{equation*}

We recall the reciprocity law. We define the homomorphism
\begin{equation*}
\mu: \mathbb{C}^{\times} \rightarrow (K \otimes_{\mathbb{Q}} \mathbb{C})^{\times} \cong
  \prod\nolimits_{\varphi\in\Phi} \mathbb{C}^{\times}. 
\end{equation*}
The element $\mu(z)$, for $z \in \mathbb{C}$, has component $z$ for
$\varphi \in \Phi^+$ and has component $1$ for $\varphi \notin \Phi^+$.  We find
$\mu \bar{\mu}(z) = 1 \otimes z \in (K \otimes_{\mathbb{Q}} \mathbb{C})^{\times}$.
We obtain a homomorphism of algebraic tori
\begin{equation*}\label{mu}
\mu: \mathbb{G}_{m, \mathbb{C}} \rightarrow \T_{\mathbb{C}}.  
  \end{equation*}
This homomorphism is defined over $E$, and hence we obtain 
\begin{equation*}
\mu: \mathbb{G}_{m,E} \rightarrow \T_E. 
\end{equation*}
From this we define the \emph{reflex norm map} by
\begin{equation}
\mathfrak{r}_{\Phi^{+}, E}: \mathrm{Res}_{E/\mathbb{Q}}( \mathbb{G}_{m,E}) 
\overset{\mu}{\longrightarrow} \mathrm{Res}_{E/\mathbb{Q}} (\T_{E} )
\overset{\Nm_{E/\mathbb{Q}}}{\longrightarrow} \T.  
\end{equation}

We consider over the algebraic closure $\bar{E} = \bar{\mathbb{Q}}$ the set of
tuples $(A, \iota, \bar{\lambda}, \kappa)$, where $(A, \iota)$ is an abelian
variety over $\bar{E}$ of CM-type $\Phi^{+}$, endowed with a
$\mathbb{Q}$-homogeneous polarization $\bar{\lambda}$ such that its Rosati involution induces on $K$
the conjugation over $F$ and an isomorphism
$\kappa: \hat{V}(A) \rightarrow  V_{0}\otimes \mathbb{A}_f$ of
$K \otimes \mathbb{A}_f$-modules. We call a second tuple
$(A', \iota', \bar{\lambda}', \kappa')$ equivalent to
$(A, \iota, \bar{\lambda}, \kappa)$ if there is  a quasi-isogeny
\begin{equation}\label{KM3e}
    \alpha: (A, \iota, \bar{\lambda}) \rightarrow
    (A', \iota', \bar{\lambda}')
    \end{equation}
such that the pull-back of $\bar{\lambda}'$ is $\bar{\lambda}$ and the pull-back of $\bar{\kappa}'$ is $\bar{\kappa}$.  In this case, we also say that $(A, \iota, \bar{\lambda}, \kappa)$ is quasi-isogenous to $(A', \iota', \bar{\lambda'}, \kappa')$.

Let $\mathcal{A}_{\Phi^{+}}$ be the set of tuples $(A, \iota,\bar{\lambda},\kappa)$
up to equivalence.
Let $\sigma \in \Gal(\bar{E}/E)$. Taking the inverse image of
$(A, \iota, \bar{\lambda}, \kappa)$ by
$\hat{\sigma} := \Spec \sigma: \Spec \bar{E} \rightarrow  \Spec \bar{E}$ 
gives a left action $(A, \iota, \bar{\lambda}, \kappa)\mapsto\sigma (A, \iota, \bar{\lambda}, \kappa)$ of $\Gal(\bar{E}/E)$ on $\mathcal{A}_{\Phi^{+}}$. 

We formulate the main theorem of complex multiplication of Shimura and
Taniyama \cite[Theorem  4.19]{De}. 
\begin{theorem}\label{ShimTan} 
  The Galois group $\Gal(\bar{E}/E)$ acts on $\mathcal{A}_{\Phi^{+}}$ via its maximal
  abelian quotient $\Gal({E}^{\rm ab}/E)$. Let
  $e \in (E \otimes \mathbb{A})^{\times}$ and let
  $\mathrm{rec}(e) \in \Gal({E}^{\rm ab}/E)$ be the automorphism given by the
  reciprocity law of class field theory. The following tuples are equivalent:
  \begin{equation*}
    \mathrm{rec}(e) (A, \iota, \bar{\lambda}, \kappa) \equiv 
    (A, \iota, \bar{\lambda}, \mathfrak{r}_{\Phi^+, E}(e_f) \kappa)  ,
    \end{equation*}
    where $e_f$ is the finite part of the id\`ele $e$.  
\end{theorem}

Let $(E^{\times})^{\wedge} \subset (E \otimes \mathbb{A}_f)^{\times}$ be the closure
of $E^{\times}$. We deduce a homomorphism 
\begin{equation}\label{reci-global}
  \Gal(\bar{E}/E) \rightarrow (E\otimes \mathbb{A}_f)^{\times}/(E^{\times})^{\wedge}
  \overset{\mathfrak{r}_{\Phi^{+}, E}}{\longrightarrow}
  \T(\mathbb{A}_f)/\T(\mathbb{Q}) ,
\end{equation}
where the first arrow is deduced from class field reciprocity normalized such that the arithmetic Frobenius is sent to a uniformizer and the second
arrow exists because $\T(\mathbb{Q})=\T(\mathbb{Q})^{\wedge}$. To see this last fact, we note that 
the group of units in $\T(\BQ)$  is finite. Indeed, the units are elements of
$K^{\times}$ with all absolute values equal to $1$ at all places including the
infinite ones. Therefore
$\T(\mathbb{Q}) = \T(\mathbb{Q})^{\wedge}$ by Chevalley's theorem. Theorem \ref{ShimTan} says
that the action of $\Gal(\bar{E}/E)$ on $\mathcal{A}_{\Phi^{+}}$ is via (\ref{reci-global}).
  In the rest of this section, we discuss a local analog of this theorem.
 
\subsection{Formal moduli space for toric CM-triples}
We fix a prime number $p$ and an algebraic closure $\bar \BQ_p$  of $\BQ_p$. Let $F$ be a finite field extension of $\BQ_p$,  with residue class field $\kappa_F$. 
Let $d=[F:\BQ_{p}]=ef$ where $e$ is the ramification index and $f$ is the inertia degree of $F/\BQ_{p}$. We  let $K/F$ be  an \'etale algebra of degree $2$. We denote the non-trivial element in $\Gal(K/F)$ by $a \mapsto
\bar{a}$.  

Let $\Phi=\Hom_{\text{$\BQ_p$-Alg}}(K, \bar\BQ_p)$. Let $\Phi^{+}_{0}$ be a classical local CM-type with respect to $K/F$, with corresponding local CM-type $r_{0}$ of rank $1$. The reflex field of $r_{0}$ will be denoted by $E_{0}=E(r_{0})$ which is also the reflex field $E_{\Phi^{+}_{0}}$ of the local CM-type $\Phi^{+}_{0}$. Let $\kappa_{E_{0}}$ be the residue field of $E_{0}$.

We define $\T$ to be the algebraic group torus over $\BQ_{p}$ with $\T(\BQ_{p})=\{a\in K^{\times}: a\bar{a}\in\BQ^{\times}_{p}\}$. Then the maximal open compact subgroup of $\T(\BQ_p)$ is $\T(\BZ_{p})=\{a\in O^{\times}_{K}: a\bar{a}\in \BZ^{\times}_{p}\}$.

\begin{construction}\label{sp-elt-loc}
Let $E$ be a finite extension of $E_{0}$ and $\Phi^{+}_{0}$ be a classical local CM-type with respect to $K/F$. We define the \emph{local reflex norm}
\begin{equation}\label{reflex-norm-0}
\mathfrak{r}_{\Phi^{+}_{0}, E}: \mathrm{Res}_{E/\BQ_{p}}\mathbb{G}_{m, E}\longrightarrow \T
\end{equation}
with respect to the data $(\Phi^{+}_{0}, E)$ similarly as in the global case. More precisely, we consider the homomorphism
\begin{equation*}
\mu: E^{\times}\rightarrow (K\otimes E)^{\times}\xrightarrow{\sim}\prod\nolimits_{\Phi}E^{\times}
\end{equation*}
given by sending $a\in E^{\times}$ to $a$ for $\varphi\in \Phi^{+}_{0}$ and to $1$ for $\varphi\not\in \Phi^{+}_{0}$. This defines a morphism of tori
\begin{equation*}
\mu: \mathbb{G}_{m, E}\longrightarrow \T_{E}.
\end{equation*}
The map $\mathfrak{r}_{\Phi^{+}_{0}, E}$ is the composite
\begin{equation*}
\mathfrak{r}_{\Phi^{+}_{0}, E}:\mathrm{Res}_{E/\BQ_{p}}\mathbb{G}_{m, E}\overset{\mu}\longrightarrow \mathrm{Res}_{{E}/\BQ_{p}}\T_{E} \xrightarrow{\Nm_{E/\BQ_{p}}}  \T.
\end{equation*}

Let $\varpi$ be a uniformizer of $E$. We  define the element  
\begin{equation}\label{w-r-0}
w_{\Phi^{+}_{0}, E}:=\mathfrak{r}_{\Phi^{+}_{0}, E}(\varpi)\in \T(\BQ_{p}).
\end{equation} 
We refer to $w_{\Phi^{+}_{0}, E}$ as the \emph{special element in $\T(\BQ_{p})$ relative to $(\Phi^{+}_{0}, E)$ for the uniformizer $\varpi$ of $O_E$} or simply as the special element in $\T(\BQ_{p})$ if the data $(\Phi^{+}_{0}, E)$ and $\varpi$ are clear from the context. 

If $E=E_{0}$ and $r_{0}$ is the local CM-type of rank $1$ associated to a classical local CM-type $\Phi^{+}_{0}$, then we write $w_{\Phi^{+}_{0}, E}$ simply as $w_{r_{0}}$.
\end{construction}
\begin{remark}\label{sp-elt-rmk}
Note that the composite of 
\begin{equation*}
\mathfrak{r}_{\Phi^{+}_{0}, E}: E^{\times}\longrightarrow \T(\BQ_{p})\subset K^{\times}
\end{equation*}
with the norm map $\Nm_{K/F}$ is given by $\Nm_{E/\BQ_{p}}$. In particular, we see that $\Nm_{K/F}(w_{\Phi^{+}_{0}, E})$ agrees with $p^{f_{E}}$ up to unit in $\BZ_{p}$.  Here $f_E$ denotes the inertia index of $E$.

Let $E^{\prime}/E$ be a finite extension with ring of integers $O_{E^{\prime}}$ and inertia degree $f_{E^{\prime}}$. Let $w_{\Phi^{+}_{0}, E^{\prime}}$ be the special element relative to $(\Phi^{+}_{0}, E^{\prime})$ for a uniformizer $\varpi^{\prime}$ of $E'$.  Then $\Nm_{K/F}(w_{\Phi^{+}_{0}, E^{\prime}})=p^{f_{E^{\prime}}}$ up to a unit in $\BZ_{p}$ and hence $\Nm_{K/F}(w_{\Phi^{+}_{0}, E^{\prime}})$ agrees with $\Nm_{K/F}(w_{\Phi^{+}_{0}, E})^{f_{E^{\prime}}/f_{E}}$ up to a unit in $\BZ_{p}$. It follows that $w_{\Phi^{+}_{0}, E^{\prime}}$ differs from $(w_{\Phi^{+}_{0}, E})^{f_{E^{\prime}}/f_{E}}$ by an element in $\T(\BZ_{p})$.
\end{remark}

\begin{definition}\label{RSZ-LS-Tori}
We define for each $i\in\BZ$ the functor $\CM_{0}(i)$ on $\Nilp_{O_{\breve{E}_{0}}}$. Fix a local CM-triple $(\BX_{0}, \iota_{\BX_0}, \lambda_{\BX_0})$ of type $(K/F, r_{0})$  over $\bar{\kappa}_{E_{0}}$, satisfying the condition $(\mathrm{KC}_{r_{0}})$. For a scheme $S$ over $\Spf O_{\breve{E}_{0}}$,  a point of $\CM_{0}(i)(S)$ is given by the following data:
\begin{enumerate}
\item[(1)] A local CM-triple $(X_{0}, \iota_{0}, \lambda_{0})$ of type $(K/F, r_{0})$ over $S$ satisfying the condition $(\mathrm{KC}_{r_{0}})$;
\item[(2)] An $O_{K}$-linear quasi-isogeny
\begin{equation*}    
\rho_{0}: \bar{X}_{0} := X_{0} \times_{S} \bar{S} \longrightarrow \mathbb{X}_{0} \times_{\Spec k} \bar{S}
\end{equation*}      
such that, locally on $S$, there exists $u\in\BZ_p^\times$ such that $\rho_{0}^*({\lambda}_{\mathbb{X}_{0}})=up^{i}\cdot{\lambda}_{0}$ on $\bar{X}_{0}$.
\end{enumerate}
Two tuples $(X_{0}, \iota_{0}, \lambda_{0}, \rho_{0})$ and $(X'_{0}, \iota'_{0}, \lambda'_{0}, \rho'_{0})$ define the same point in $\CM_{0}(i)$ if and only if there is an isomorphism $\alpha_{0}: (X_{0}, \iota_{0})\rightarrow (X'_{0}, \iota'_{0})$ such that $\rho'_{0}\circ \alpha_{0}=\rho_{0}$. Then  $\alpha^{\ast}_{0}(\lambda_0')=u\lambda_0$ for some $u\in\BZ_p^\times$. 
\end{definition}

The formal scheme representing this functor is also denoted by $\CM_{0}(i)$. We define the formal scheme $\CM_{0}$ over $\Spf O_{\breve{E}_{0}}$ by
\begin{equation}\label{M0}
{\CM}_{0}:=\bigsqcup\nolimits_{i\in\BZ}\CM_{0}(i).
\end{equation}
Then $\CM_0$ parametrizes tuples $(X_{0}, \iota_{0}, \lambda_{0}, \rho_0)$ such that $\rho_0^*({\lambda}_{\mathbb{X}_{0}})=\nu\lambda_0$ with $\nu\in\BQ_p^\times$.  Let 
\begin{equation}\label{def:JT}
J_{\T}(\BQ_{p})=\{\alpha_0\in \mathrm{Aut}^{\circ}_{K}(\mathbb{X}_{0})\mid \alpha_0^{\ast}(\lambda_{\mathbb{X}_{0}})= \mu_{0}(\alpha_0)\lambda_{\mathbb{X}_{0}}, \,\,\mu_{0}(\alpha_0)\in \BQ^{\times}_{p}\}.
\end{equation} Then $J_{\T}(\BQ_{p})$ acts naturally on $\CM_{0}$ by $(X_{0}, \iota_{0}, \lambda_{0}, \rho_{0})\mapsto (X_{0}, \iota_{0}, \lambda_{0}, \alpha_{0}\circ\rho_{0})$.

We have the polarized contraction functor $\mathfrak{C}_{r_{0}}^{\rm pol}$ in \cite[Theorem  4.5.11]{KRZI},  trivially modified to the case when $r_{0}$ is of rank $1$. It associates to a local CM triple $(X_{0}, \iota_{0}, \lambda_{0})$ of type $(K/F, r_{0})$ over a scheme $S\in \Nilp_{O_{\breve{E}_{0}}}$ a triple $(C, \iota, \phi)$ where $C$ is a $p$-adic \'etale sheaf in $\BZ_p$-modules over $S$ with $\rank_{\mathbb{Z}_p} C = 2d$ (where we recall $d=[F: \BQ_{p}]$) with an $O_{K}$-action $\iota: O_{K}\rightarrow \End_{\BZ_{p}}(C)$, and where $\phi: C\times C\rightarrow O_{F}$ is an $O_F$-bilinear pairing on $C$ such that $\phi(\iota(a)c_{1}, c_{2})=\phi(c_{1}, \iota(\bar{a})c_{2})$ for $c_{1},c_{2}\in C$ and $a\in O_{K}$. Let $(\underline{C}_{\mathbb{X}_{0}}, \iota_{0}, \lambda_{0})$ be the image of $(\mathbb{X}_{0}, \iota_{\mathbb{X}_{0}}, \lambda_{\mathbb{X}_{0}})$ under this functor. Here $\underline{C}_{\mathbb{X}_{0}}$ is regarded as the constant
sheaf on $S$ associated to an $O_{K}$-module $C_{\BX_{0}}\simeq O_{K}$ described similarly as in \cite[Remark 4.5.12]{KRZI} and \cite[(7.4.14)]{KRZI} (trivially modified to the case when $r_{0}$ is a local CM-type of rank $1$).

\begin{definition}\label{C-moduli}
We consider the following functor $\mathcal{G}_{0}(i)$ on the category of schemes $S$ over $\Spf O_{\breve{E}_{0}}$. A point of $\mathcal{G}_{0}(i)(S)$ is given
by the following data:
\begin{enumerate}
\item[(1)]A locally
  constant $p$-adic \'etale sheaf $C$ on $S$ which is
  $\mathbb{Z}_p$-flat with $\rank_{\mathbb{Z}_p} C = 2d$ and with an action
\begin{equation*}
\iota: O_{K} \longrightarrow \End_{\mathbb{Z}_p} C;
\end{equation*}

\item[(2)]A perfect alternating $O_F$-bilinear pairing
  \begin{equation*}
  \phi: C \times C \longrightarrow O_F, 
  \end{equation*}
  such that $\phi(\iota(a) c_1, c_2) = \phi (c_1, \iota(\bar{a})c_2)$ for
  $c_1, c_2 \in C$ and $a \in O_{K}$;

\item[(3)]A quasi-isogeny of $O_K$-module sheaves on $S$ 
\begin{equation*}\label{KneunBa8e}
\rho: (C, \iota) \longrightarrow (\underline{C}_{\mathbb{X}_{0}}, \iota_{0}). 
\end{equation*}
\end{enumerate}

Another set of data $(C', \iota',\phi', \rho')$ defines the same point
iff there is an isomorphism $\alpha: C \isoarrow C'$ such that
$\rho' \circ \alpha = \rho$. Then $\alpha$ respects $\phi$ and $\phi'$ up
to a factor in $O_{F}^{\times}$.
\end{definition}

The existence of the quasi-isogeny implies that $C$ is locally
constant for the Zariski topology. Therefore locally on $S$ the sheaf
$C$ is the constant sheaf associated to an $O_K$-submodule
$C \subset C_{\mathbb{X}_{0}} \otimes_{\mathbb{Z}_p} \mathbb{Q}_p$ and $\rho$ is
given by the last inclusion. 

The polarized contraction functor $\mathfrak{C}_{r_{0}}^{\rm pol}$ in \cite[Theorem  4.5.11]{KRZI} defines a isomorphism of functors
\begin{equation}\label{KneunBa13e} 
\mathcal{M}_{0}(i) \isoarrow \mathcal{G}_{0}(i).
\end{equation}

To describe the functor $\mathcal{G}_{0}(i)$, we may restrict to the case
where the sheaf $C$ is given by an $O_K$-submodule of
$C_{\mathbb{X}_{0}} \otimes_{\mathbb{Z}_p} \mathbb{Q}_p$. 
Then $C$ defines a point of $\mathcal{G}_{0}(i)(S)$ iff
$(1/p^i ) \phi_{\mathbb{X}_{0}}$ is a perfect alternating pairing on $C$.
We define an algebraic group over $\mathbb{Z}_p$ whose $\BZ_p$-rational points are given by
\begin{equation*}
\T'(\mathbb{Z}_p) = \{g \in \GL_{O_{K}}( C_{\mathbb{X}_{0}} )\; | \;
  \phi_{\mathbb{X}_{0}}(gc_1, gc_2) = u\cdot \phi_{\mathbb{X}_{0}}(c_1, c_2) \;
  \text{for some} \; u \in O^{\times}_{F}\}. 
  \end{equation*}
Since all such $C$ are isomorphic to $O_{K}$, cf. \cite[7.4.14]{KRZI}, there is an isomorphism
\begin{equation}\label{g-T-1}
g: (C_{\mathbb{X}_{0}}, \phi_{\mathbb{X}_{0}}) \longrightarrow (C, \frac{1}{p^i } \phi_{\mathbb{X}_{0}}).
\end{equation}
This means that $gC_{\mathbb{X}_{0}} = C$ and
\begin{equation}\label{g-T-2}
\phi_{\mathbb{X}_{0}}(gc_1, gc_2) = p^i \cdot\phi_{\mathbb{X}_{0}}(c_1, c_2) , \;
\quad  c_1, c_2 \in C_{\mathbb{X}_{0}} \otimes_{\mathbb{Z}_p} \mathbb{Q}_p .
\end{equation}
We define 
\begin{equation*}
\T'(i) = \{ g \in \GL_{K}(C_{\mathbb{X}_{0}} \otimes_{\mathbb{Z}_p} \mathbb{Q}_p)\; | \;
\phi_{\mathbb{X}_{0}}(gc_1, gc_2) = p^i u\cdot \phi_{\mathbb{X}_{0}}(c_1, c_2), \;
 \; u \in O^{\times}_{F}\}.
\end{equation*}
This construction gives us a functor isomorphism
\begin{equation*}
\mathcal{G}_{0}(i) \isoarrow \T^{\prime}(i)/\T^{\prime}(\mathbb{Z}_p),
\end{equation*}
where the right hand side is considered as the restriction of the constant
sheaf to $\Nilp_{O_{\breve{E}_{0}}}$. Hence we obtain  an isomorphism
\begin{equation}\label{RZ-tori-prime-i}
\mathcal{M}_{0}(i)\overset{\sim}\longrightarrow \T^{\prime}(i)/\T^{\prime}(\mathbb{Z}_p).
\end{equation}

Let $\T^{\prime}(\BQ_{p})\subset \GL_{K}(C_{\mathbb{X}_{0}} \otimes_{\mathbb{Z}_p} \mathbb{Q}_p)$ be the union of $\T^{\prime}(i)$. Using the
functor $\mathfrak{C}_{r_{0}}^{\rm pol}$, we may write
\begin{equation}\label{T-prime}
\T^{\prime}(\BQ_{p})= \{\alpha \in \Aut_{K} \mathbb{X}_{0} \; |\;
\alpha^{\ast} (\lambda_{\mathbb{X}_{0}}) = \mu(\alpha) \lambda_{\mathbb{X}_{0}} \;
\text{for some} \; \mu(\alpha) \in p^{\mathbb{Z}} O^{\times}_{F}\}. 
\end{equation}
Therefore $\T^{\prime}(\BQ_{p})$ acts via $\rho$ on the formal scheme ${\mathcal{M}}_{0}$.
The  isomorphism of functors on $\Nilp_{O_{\breve{E}_{0}}}$ obtained from \eqref{KneunBa13e} and \eqref{RZ-tori-prime-i} is then a $\T^{\prime}(\BQ_{p})$-equivariant isomorphism,
\begin{equation}\label{T-prime-coset-i}
{\mathcal{M}}_{0} \isoarrow \T^{\prime}(\BQ_{p})/\T^{\prime}(\mathbb{Z}_p). 
\end{equation}

Now we define
\begin{equation}\label{Ti-group}
\T(i) = \{ g \in \GL_{K}(C_{\mathbb{X}_{0}} \otimes_{\mathbb{Z}_p} \mathbb{Q}_p)\; | \;
\phi_{\mathbb{X}_{0}}(gc_1, gc_2) = p^if\cdot \phi_{\mathbb{X}_{0}}(c_1, c_2), \;
\text{for some} \; f \in \BZ^{\times}_{p}\}.
\end{equation}
Using an identification $C_{\BX_{0}}\simeq O_{K}$, we have
\begin{equation*}
\T(\BQ_{p})=\bigsqcup\nolimits_{i\in\BZ} \T(i)
\end{equation*} 
and, similarly,
\begin{equation*}
\T(\mathbb{Z}_p) = \{g \in \GL_{O_{K}}( C_{\mathbb{X}_{0}} )\; | \;
\phi_{\mathbb{X}_{0}}(gc_1, gc_2) = f\cdot \phi_{\mathbb{X}_{0}}(c_1, c_2) \;
\text{for some} \; f \in \BZ^{\times}_{p}\}. 
\end{equation*}
From this description, it is clear that $J_{\T}(\BQ_{p})=\T(\BQ_{p})$. 

Using \eqref{g-T-1} and \eqref{g-T-2}, it is easy to see that  $\T^{\prime}(i)/\T^{\prime}(\mathbb{Z}_p)$ is in bijection with $\T(i)/\T(\mathbb{Z}_p)$. Thus we have by \eqref{RZ-tori-prime-i}
\begin{equation}\label{T-coset-space-i}
{\mathcal{M}}_{0}(i)\isoarrow \T(i)/\T(\BZ_{p}),
\end{equation}
 and hence an isomorphism 
\begin{equation}\label{T-coset-space}
{\mathcal{M}}_{0}\isoarrow \T(\BQ_{p})/\T(\BZ_{p})
\end{equation}
which is equivariant for the action of $J_{\T}(\BQ_{p})=\T(\BQ_{p})$ on both sides. 

Next we discuss the Weil descent datum on ${\mathcal{M}}_{0}$ from $O_{\breve{E}_{0}}$ down to $O_{E_{0}}$.   Let $\tau_{E_0}\in \Gal(\breve E_0/E_0)$ be the Frobenius. It is enough to restrict the functor $\CM_{0}$
to affine schemes $S = \Spec R$ over $\Spf O_{\breve{E}_{0}}$ . We write
$\varepsilon: O_{\breve{E}_{0}} \longrightarrow R$ for the given algebra structure.
We write  $R_{[\tau_{E_{0}}]}$  for the ring $R$ with the new algebra structure
$\varepsilon \circ\tau_{E_{0}}$.  
By base change to $\bar{\kappa}_{E_{0}}$, we obtain
\begin{equation*}
\bar{\varepsilon}: \bar{\kappa}_{E_{0}} \longrightarrow \bar{R} :=
R \otimes_{O_{\breve{E}_{0}}} \bar{\kappa}_{E_{0}}. 
\end{equation*}
    
Let $(X_{0}, \iota_{0}, \lambda_{0}, \rho_{0})$ be a point of $\CM_{0}(\bar{R})$  where $\rho_{0}$ is a quasi-isogeny
\begin{equation*}
\rho_{0}: X_{0, \bar{R}} \longrightarrow \bar{\varepsilon}_{\ast} \mathbb{X}_{0}= \BX_{0, \bar{R}}. 
\end{equation*}
Since the notion of a CM-triple depends only on the induced $O_{E}$-algebra  structure on $R$, we may regard $(X_{0}, \iota_{0}, \lambda_{0}, \rho_{0})$
as a CM-triple on $R_{[\tau_{E_{0}}]}$. We set 

\begin{equation}\label{frame-Fr}
\tilde{\rho}_{0}: X_{0}\xrightarrow{\rho_{0}} \BX_{0, \bar{R}}\xrightarrow{F_{\BX_{0}, \tau_{E_{0}}}} \tau_{E_{0}\ast}\BX_{0, \bar{R}}.
\end{equation}
 
The assignment 
\begin{equation*}
(X_{0}, \iota_{0}, \lambda_{0}, \rho_{0}) \mapsto (X_{0}, \iota_{0}, \lambda_{0}, \tilde{\rho}_{0})
\end{equation*}
defines a morphism
\begin{equation}\label{WD-M-0}
\omega_{{\mathcal{M}}_{0}}: \CM_{0}(i)(\bar{R})\rightarrow \CM_{0}(i+f_{E_{0}})(\bar{R}_{[\tau_{E_{0}}]}).
\end{equation}
Here we note that the inverse image
of the polarization $(\tau_{E_{0}})_{\ast} \lambda_{\mathbb{X}_{0}}$ on
$(\tau_{E_{0}})_{\ast} \mathbb{X}_{0}$ by the Frobenius isogeny
\begin{equation*}
F_{\mathbb{X},\tau_{E_{0}}}: \mathbb{X}_{0} \longrightarrow (\tau_{E_{0}})_{\ast} \mathbb{X}_{0}
\end{equation*}
is given by $p^{f_{E_{0}}} \lambda_{\mathbb{X}_{0}}$. Then  \eqref{WD-M-0} defines the Weil descent datum
\begin{equation*}\label{WD-Mr}
\omega_{\CM_{0}}: \CM_{0}\rightarrow \CM_{0}^{(\tau_{E_{0}})}.
\end{equation*}

\begin{proposition}\label{LS-tori-weil}
There is an isomorphism of formal schemes over $\Spf O_{\breve{E}_{0}}$,
\begin{equation*}
{\mathcal{M}}_{0} \overset{\sim}{\longrightarrow} \T(\mathbb{Q}_p)/\T(\BZ_{p}) 
\end{equation*}
which is equivariant for the action of $J_{\T}(\BQ_{p})=\T(\BQ_{p})$ on both sides.

The above Weil descent datum on ${\mathcal{M}}_{0}$ corresponds on the right-hand side to the Weil descent datum given by multiplication by the special element $w_{r_{0}}$ in $\T(\BQ_{p})$ relative to $(\Phi^{+}_{0}, E_{0})$ for a uniformizer $\varpi_{0}$.
\end{proposition}
\begin{proof}
The isomorphism 
\begin{equation*}
{\mathcal{M}}_{0} \overset{\sim}{\longrightarrow} \T(\mathbb{Q}_p)/\T(\BZ_{p})
\end{equation*}
is already proved in \eqref{T-coset-space}. Recall from the beginning of \S \ref{ss:specialbanal} the meaning of $\Pi$ and $\pi$. We define the element $w'_{r_{0}}\in \T^{\prime}(\BQ_{p})$ by the following recipe:
\begin{enumerate}
\item[(1)] If $K/F$ is ramified,  we define $w'_{r_{0}}$ as the multiplication
\begin{equation*}
\Pi^{ef_{E_{0}}}: C_{\mathbb{X}_{0}} \otimes_{\mathbb{Z}_p} \mathbb{Q}_p
\longrightarrow  C_{\mathbb{X}_{0}} \otimes_{\mathbb{Z}_p} \mathbb{Q}_p;
\end{equation*}
    
\item[(3)] If $K/F$ is unramified, we define $w'_{r_{0}}$ as the multiplication
\begin{equation*}
\pi^{ef_{E_{0}}/2}: C_{\mathbb{X}_{0}} \otimes_{\mathbb{Z}_p} \mathbb{Q}_p \longrightarrow C_{\mathbb{X}_{0}} \otimes_{\mathbb{Z}_p} \mathbb{Q}_p,
\end{equation*}
where we note that $ef_{E_0}$ is even by the same computation as in \cite[Lemma 6.4.4]{KRZI}.

\item[(4)] If $K = F \times F$ is split, we have the decomposition 
\begin{equation*}
C_{\mathbb{X}_{0}}\otimes_{\BZ_{p}}\BQ_{p}=C_{\mathbb{X}_{0}, 1}\otimes_{\BZ_{p}}\BQ_{p}\oplus C_{\mathbb{X}_{0}, 2}\otimes_{\BZ_{p}}\BQ_{p}.
\end{equation*}
We consider the numbers
$a_{1}$ and $a_{2}$ defined by $\dim \mathbb{X}_{0,1}=a_{1}$ and  $\dim \mathbb{X}_{0,2}=a_{2}$ and set
\begin{equation*}
a_{1,E_{0}} = a_{1}\frac{f_{E_{0}}}{f}, \quad \text{ resp. }\quad a_{2,E_{0}} = a_{2}\frac{f_{E_{0}}}{f} .
\end{equation*}
Then we have $a_{1,E_{0}} + a_{2,E_{0}} = ef_{E_{0}}$. We define $w'_{r_{0}}$ to be the multiplication by $\pi^{a_{1,E_{0}}}$ on $C_{\mathbb{X}_{0}, 1}\otimes_{\BZ_{p}}\BQ_{p}$ and the multiplication by $\pi^{a_{2,E_{0}}}$ on $C_{\mathbb{X}_{0}, 2}\otimes_{\BZ_{p}}\BQ_{p}$. 
\end{enumerate}
Then by the same calculation on Dieudonn\'e modules as in \cite[Propositions 6.3.2, 6.4.3, 6.5.1]{KRZI} adapted to the case of rank $1$ local CM-types, the natural descent datum on the left-hand side of 
\begin{equation*}
{\mathcal{M}}_{0} \isoarrow \T^{\prime}(\BQ_{p})/\T^{\prime}(\mathbb{Z}_p)
\end{equation*}
corresponds on the right-hand side to the Weil descent datum given by multiplication by $w^{\prime}_{r_{0}}$. Indeed, for example, suppose $K/F$ is ramified. Let $\tau$ be the Frobenius automorphism of $\bar{\kappa}_{E_{0}}$ over $\BF_{p}$, then there is an identification 
$C_{\mathbb{X}_{0}} = C_{\tau_{\ast} \mathbb{X}_{0}}$. We apply the functor $\mathfrak{C}_{r_{0}}^{\rm pol}$ to the Frobenius isogeny 
$F_{\mathbb{X}_{0}}: \mathbb{X}_{0} \longrightarrow \tau_{\ast} \mathbb{X}_{0}$. Consider the Dieudonn\'e module $P_{\mathbb{X}_{0}}$ of $\mathbb{X}_{0}$, we have
$C_{\mathbb{X}_{0}} = \{c \in P_{\mathbb{X}_{0}} \; | \; Vc = \Pi^e c \}$.
The map
\begin{equation*}
P_{\mathbb{X}_{0}} \longrightarrow W(\bar{\kappa}_{E_{0}})\otimes_{\tau, W(\bar{\kappa}_{E_{0}})} P_{\mathbb{X}_{0}}, \quad c \mapsto 1 \otimes c 
\end{equation*}
defines the identification $C_{\mathbb{X}_{0}} = C_{\tau_{\ast} \mathbb{X}_{0}}$. The Frobenius $F_{\mathbb{X}_{0}}$ induces on the Dieudonn\'e modules
\begin{displaymath}
P_{\mathbb{X}_{0}} \overset{V^{\sharp}}{\longrightarrow} W(\bar{\kappa}_{E_{0}}) \otimes_{\tau, W(\bar{\kappa}_{E_{0}})} P_{\mathbb{X}_{0}}, \quad x \mapsto 1 \otimes Vx. 
\end{displaymath}
For $c \in C_{\mathbb{X}_{0}}$ we obtain $V^{\sharp}c = 1 \otimes Vc = 1 \otimes \Pi^ec$. Therefore $F_{\mathbb{X},\tau_{E_{0}}}: \mathbb{X}_{0} \longrightarrow (\tau_{E_{0}})_{\ast} \mathbb{X}_{0}$ corresponds to multiplication by $\Pi^{ef_{E_{0}}}$ on $C_{\mathbb{X}_{0}}$. The other cases follow similarly.

Since the element $w_{r_{0}}$ has norm $\Nm_{K/F}(w_{r_{0}})=p^{f_{E_{0}}}u$ for some unit $u\in \BZ_{p}$ and  $w'_{r_{0}}$ has norm $\Nm_{K/F}(w'_{r_{0}})=\pi ^{ef_{E_{0}}}$ when we view them as elements in $K^{\times}$, they differ by a unit in $O^{\times}_{K}$.
Therefore, in the bijection
\begin{equation*}
\T(\BQ_{p})/\T(\mathbb{Z}_p)\xrightarrow{\sim} \T^{\prime}(\BQ_{p})/\T^{\prime}(\mathbb{Z}_p),
\end{equation*} 
the class of $w_{r_{0}}$ is sent to the class of $w'_{r_{0}}$ and hence the descent datum is given as in the statement of the proposition.
\end{proof}

\begin{remark}
Let $E$ be a finite extension of $E_{0}$ and $O_{E}$ be its ring of integers. Let $f_{E}$ be the inertia degree of $E$ and $\varpi$ be a uniformizer of $O_{E}$, then it is clear from the proof of the above proposition and Remark \ref{sp-elt-rmk} that we have an isomorphism $\CM_{0, O_{\breve{E}}}\overset{\sim}{\longrightarrow} \T(\mathbb{Q}_p)/\T(\BZ_{p})$ and the Weil descent datum on the left hand side relative to $O_{\breve{E}}/O_{E}$ is given on the right-hand side by multiplication by the special element $w_{\Phi^{+}_{0}, E}\in \T(\BQ_{p})$ for the uniformizer $\varpi$.
\end{remark}

\part{The RSZ unitary Shimura curves}
\section{Formal moduli space of RSZ CM-triples}\label{s:formRSZ}
 Let $F$ be a finite field extension of $\BQ_p$,  with residue class field $\kappa_F$. 
We set $d=[F:\BQ_p]$, 
 $f = [\kappa_F : \mathbb{F}_p]$ and define $e$ through $d = ef$.
We  let $K/F$ be  an \'etale algebra of degree $2$. 

Let $r$ be a local CM-type of rank $2$ relative to $K/F$ whose reflex field is denoted by $E_{r}$.  Let $\Phi^{+}_{0}$ be a classical local CM-type relative to $K/F$ whose reflex field is denoted by $E_{0}$ and let $r_{0}$ be the local CM-type of rank $1$ associated to $\Phi^{+}_{0}$. Let  $E=E_{0}E_{r}$ be the composite field of $E_{0}$ and $E_{r}$. We denote by $O_{E}$ the valuation ring of $E$ and by $f_{E}$  the inertia degree of $E$ over $\BQ_{p}$. Let $\kappa=\kappa_{E}$ be residue class field of $E$. We fix a uniformizer $\varpi_{E}$ of $E$.

We define $\T$ to be the torus over $\BQ_{p}$ with $\T(\BQ_{p})=\{a\in K^{\times}\colon a\bar{a}\in\BQ^{\times}_{p}\}$. 
Let  $V$ be a free $K$-module of rank $2$ equipped with an alternating $\BQ$-bilinear form 
\begin{equation}\label{varsigma-loc}
\varsigma\colon V\times V\to \BQ_{p} 
\end{equation}
such that
\begin{equation}\label{local-alt}
\varsigma(a x, y)=\varsigma(x, \bar a y), \quad x, y\in V, \, a\in K . 
\end{equation}
 There is a unique anti-hermitian form $\varkappa$ on $V$ such that 
 \begin{equation}\label{introback}
  \Trace_{K/\mathbb{Q}_{p}} a \varkappa(x, y) = \varsigma(ax, y), \quad x, y\in V, \, a\in K. 
\end{equation}
Conversely, the anti-hermitian form $\varkappa$ determines the alternating bilinear form $\varsigma$ with \eqref{local-alt}. We  say that $\varkappa$ arises from $\varsigma$ by contraction.
Let $\Lambda$ be an almost selfdual $O_{K}$-lattice in $V$ (see the Notation section for this term). Let $\mathrm{U}(V)$ be the unitary group over $F$ for the anti-hermitian space $(V, \varkappa)$ and set $\U=\Res_{F/\BQ_p}(\U(V))$. We define the algebraic group $\wt{G}$ over $\BQ_{p}$ by $\wt{G}=\U\times \T$. Let $G=\mathrm{GU}(V, \varsigma)$ be the unitary similitude group of $V$ with similitude factor in $\BG_{m}$. Then $\wt{G}$ is also isomorphic to the fiber product group $G\times_{\BG_{m}}\T$, similarly as in the global case, cf. \eqref{tildeG}.

Let $(X, \iota, \lambda)$ be a local CM triple of type $r$, cf. \S \ref{ss:loctriples}. We say that $(X, \iota, \lambda)$ is $(V, \varsigma)$-principally polarized if $(X, \iota, \lambda)$ is principally polarized unless $K/F$ is unramified and $\inv(V)=-1$, in which case it is almost principally polarized, cf. Definition \ref{a-p-p}.

\subsection{Formal moduli space of RSZ CM-triples: the special case} 
In this subsection, we will study a formal moduli space of CM-triples associated to the group $\wt{G}$ and describe it in terms of the Drinfeld upper-half plane with its descent datum.  We assume that $K/F$ is a field extension. 

We fix an embedding $\varphi_{0}: K\rightarrow \bar{\BQ}_{p}$. We suppose in this subsection that the local CM-type $r$ is special relative to ${\varphi_{0}}_{| F}$, cf. Definition \ref{loc-special-type}. We also let $r_{0}$ be a local CM-type of rank $1$ whose associated local CM-type is denoted by $\Phi^{+}_{0}$. We assume that $\varphi_0\in\Phi^+_0$.

In this subsection, we will  assume that $V$  is anisotropic, i.e., $\inv(V)=-1$. Therefore  $\bfK^{\circ}_{\U}:=\U(\BQ_{p})$ is a maximal quasi-parahoric. Let $\T(\BZ_{p})=\{a\in O^{\times}_{K}: a\bar{a}\in \BZ^{\times}_{p}\}$  be the maximal open compact subgroup of $\T(\BQ_{p})$ which we will also denote by $\bfK^{\circ}_{\T}$. Let $\bfK^{\circ}_{\wt{G}}=\bfK^{\circ}_{\T}\times \bfK^{\circ}_{\U}$, it is a maximal open compact subgroup of $\wt{G}(\BQ_{p})$. Note that  we have 
\begin{equation}\label{wtverT}
\wt{G}(\BQ_{p})/\bfK^{\circ}_{\wt{G}}\simeq \U(\BQ_{p})/\bfK^{\circ}_{\U}\times \T(\BQ_{p})/\bfK^{\circ}_{\T}\simeq \T(\BQ_{p})/\bfK^{\circ}_{\T}.
\end{equation}

We fix a local CM-triple $(\mathbb{X}_{0}, \iota_{\mathbb{X}_{0}}, \lambda_{\mathbb{X}_{0}})$ of type $(K/F, r_{0})$ over $\bar{\kappa}$ satisfying the condition $({\rm KC}_{r_{0}})$ and a local CM triple $(\mathbb{X}, \iota_{\mathbb{X}}, \lambda_{\mathbb{X}})$ of type $(K/F, r)$ over $\bar{\kappa}$ which satisfies the conditions $({\rm KC}_{r})$ and $({\rm EC}_{r})$ and is $(V,\varsigma)$-principally polarized. We also assume that $\inv^r(\mathbb{X}, \iota_{\mathbb{X}}, \lambda_{\mathbb{X}})=-1=\inv(V)$.

\begin{definition}\label{RSZ-LS-special}
We define for each $i\in \BZ$ the functor $\wt{\CM}_{\bK^{\circ}_{\wt G}}(i)$ on $\Nilp_{O_{\breve{E}}}$. For a scheme $S$ over $\Spf O_{\breve{E}}$, a point of $\wt{\CM}_{\bK^{\circ}_{\wt G}}(i)(S)$ consists of the following data:

\begin{enumerate}
\item[(1)] A local CM-triple $(X_{0}, \iota_{0}, \lambda_{0})$ of type $(K/F, r_{0})$ satisfying the condition $({\rm KC}_{r_{0}})$;
      
\item[(2)] An $O_{K}$-linear quasi-isogeny
\begin{equation*}    
\rho_{0}: \bar{X}_{0} := X_{0} \times_{S} \bar{S} \longrightarrow \mathbb{X}_{0} \times_{\Spec \bar{\kappa}} \bar{S}
\end{equation*}
such that locally on $\bar S$, there exists $u\in\BZ_p^\times$ such that $\rho_{0}^{\ast}(\lambda_{\BX_{0}})=up^{i}\lambda_{0}$;

\item[(3)] A local CM-triple $(X, \iota, \lambda_{X})$ of type
$(K/F ,r)$ over $S$
which satisfies the conditions $({\rm KC}_{r})$ and $({\rm EC}_{r})$ and is $(V, \varsigma)$-principally polarized;

\item[(4)] An $O_{K}$-linear quasi-isogeny 
\begin{equation*}
\rho: \bar{X} := X\times_{S} \bar{S} \longrightarrow \mathbb{X}\times_{\Spec \bar {\kappa}} \bar{S} 
\end{equation*}
such that $\rho^{\ast}(\lambda_{\BX})=up^{i}{\lambda_{\bar{X}}}$, locally on $\bar S$,  with the same  $u\in\BZ_p^\times$ as in datum $(2)$.
\end{enumerate}
\end{definition}

We denote these data by a tuple $(X_{0}, \iota_{0}, \lambda_{0},\rho_{0}, X, \iota, \lambda_{X}, \rho)$. Two tuples $(X_{0}, \iota_{0}, \lambda_{0}, \rho_{0}, X, \iota, \lambda_{X}, \rho)$ and $(X^{\prime}_{0}, \iota^{\prime}_{0}, \lambda^{\prime}_{0}, \rho^{\prime}_{0}, X^{\prime}, \iota^{\prime}, \lambda_{X^{\prime}}, \rho^{\prime})$ define the same point of $\wt{\CM}_{\bK^{\circ}_{\wt G}}(i)$ iff there exist isomorphisms of $O_{K}$-modules $\alpha_{0}: (X_{0}, \iota_{0})\isoarrow (X^{\prime}_{0}, \iota^{\prime}_{0})$ and $\alpha: (X, \iota)\isoarrow (X^{\prime}, \iota^{\prime})$  such that $\rho^{\prime}_{0}\circ\alpha_{0, \bar{S}}=\rho_{0}$ and $\rho^{\prime}\circ\alpha_{\bar{S}}=\rho$. 

\begin{remark}\label{simi-factor-rmk}
We could replace the condition in $(2)$ by $\rho_{0}^{\ast}(\lambda_{\BX_{0}})=p^{i}\lambda_{0}$ and the condition in $(4)$ by $\rho^{\ast}(\lambda_{\BX})=p^{i}\lambda_{X}$ without changing the formal moduli space $\wt{\CM}_{\bK^{\circ}_{\wt G}}(i)$. We could also replace the condition in $(2)$ by the condition that $\rho_{0}^{\ast}(\lambda_{\BX_{0}})$ differs from $p^{i}{\lambda}_{0}$ on $\bar{X}_{0}$ by a unit $u$ in $O_{F}$ and the condition in $(4)$ by the condition that $\rho^{\ast}(\lambda_{\BX})$ differs from $p^{i}{\lambda}_{\bar{X}}$ on $\bar{X}$ by the same unit. This follows from the same reasoning as in \cite[Remark 7.2.3]{KRZI}. 
\end{remark}

The formal scheme over $\Spf O_{\breve{E}}$ representing this functor is also denoted by $\wt{\CM}_{\bK^{\circ}_{\wt G}}(i)$.   We define the formal scheme
\begin{equation}\label{wtCM-special}
\wt{\CM}_{\bK^{\circ}_{\wt G}}=\bigsqcup\nolimits_{i\in\BZ}\wt{\CM}_{\bK^{\circ}_{\wt G}}(i). 
\end{equation}
Then $\wt{\CM}_{\bK^{\circ}_{\wt G}}$ parametrizes tuples $(X_{0}, \iota_{0}, \lambda_{0}, \rho_{0}, X, \iota, \lambda, \rho)$ such that $\rho_{0}$ respects the polarization $\lambda_0$ on $\bar{X}_{0}$ and ${\lambda}_{\mathbb{X}_{0}}$ on $\BX_{0,\bar{S}}$ up to a factor in $\mathbb{Q}_p^{\times}$ and $\rho$ respects the polarization $\lambda_{\bar{X}}$ on $\bar{X}$ and ${\lambda}_{\mathbb{X}}$ on $\BX_{\bar{S}}$ up to the same factor in $\mathbb{Q}_p^{\times}$.

We define 
\begin{equation}\label{def:Js}
\begin{aligned}
J_{\T}(\BQ_{p})&=\{\alpha_0\in \mathrm{Aut}^{\circ}_{K}(\mathbb{X}_{0})\mid \alpha_0^{\ast}\lambda_{\mathbb{X}_{0}}= \mu_{0}(\alpha_0)\lambda_{\mathbb{X}_{0}}, \,\,\mu_{0}(\alpha_0)\in \BQ^{\times}_{p}\};\\
J(\BQ_{p})&=\{\alpha\in \mathrm{Aut}^{\circ}_{K}(\mathbb{X})\mid  \alpha^{\ast}\lambda_{\mathbb{X}}=\mu(\alpha)\lambda_{\mathbb{X}},\,\,  \mu(\alpha)\in \BQ^{\times}_{p}\};\\
\wt{J}(\BQ_{p})&=J(\BQ_{p})\times_{\BQ^{\times}_{p}} J_{\T}(\BQ_{p}).
\end{aligned}
\end{equation}
  Note that $\wt{\CM}_{\bK^{\circ}_{\wt G}}$ affords an action of the group $\wt{J}(\BQ_{p})$ via the framing data $\rho$ and $\rho_{0}$ in $(4)$ and $(2)$ by 
  \[
 \wt\alpha=(\alpha_0, \alpha): (X_{0}, \iota_{0}, \lambda_{0},\rho_{0}, X, \iota, \lambda_{X}, \rho)\mapsto (X_{0}, \iota_{0}, \lambda_{0},\alpha_{0}\circ\rho_{0}, X, \iota, \lambda_{X}, \alpha\circ\rho) .
  \]

Note that by  \cite[Lemma 5.2.2, Lemma 5.3.2]{KRZI}, $J(\BQ_{p})$ is the group of $\BQ_p$-points of a unitary similitude group $J$ with rational similitude factor of a anti-hermitian space  whose invariant is $+1$, cf. \cite[(7.2.6)]{KRZI}. Also, $J_{\T}(\BQ_{p})$ is isomorphic to $\T(\BQ_{p})$, see \eqref{Ti-group}. Therefore $\wt{J}=J\times_{\BG_{m}}\T$ is an inner form of $\wt{G}$ such that $\wt J_\ad$ is isomorphic to $\Res_{F/\BQ_p}(\PGL_2)$.

Next, we consider the Weil descent datum on $\wt\CM_{\bK^{\circ}_{\wt G}}$. It is enough to restrict the functor $\wt\CM_{\bfK^{\circ}_{\wt G}}$ to  affine schemes  $S = \Spec R$ over $\Spf O_{\breve{E}}$. We write
$\varepsilon: O_{\breve{E}} \longrightarrow R$ for the given algebra structure. Let $\tau_{E}$ be the Frobenius automorphism of $\bar{\kappa}$ over $\kappa$. We write  $R_{[\tau_{E}]}$ for the ring $R$ with the new algebra structure
$\varepsilon \circ\tau_{E}$.  
By base change to $\bar{\kappa}$, we obtain $\bar{\varepsilon}: \bar{\kappa} \longrightarrow \bar{R}:= R \otimes_{O_{\breve{E}}} \bar{\kappa}$. 

Let $(X_{0}, \iota_{0}, \lambda_{0}, \rho_{0}, X, \iota, \lambda_{X}, \rho)$ be a point of $\wt\CM_{\bK^{\circ}_{\wt G}}(\bar{R})$  where $\rho_{0}$ and $\rho$ are  quasi-isogenies
\begin{equation*}
\rho_{0}: X_{0, \bar{R}} \longrightarrow \bar{\varepsilon}_{\ast} \mathbb{X}_{0}= \BX_{0, \bar{R}}, \quad \text{ resp.}\quad \rho: X_{\bar{R}} \longrightarrow \bar{\varepsilon}_{\ast} \mathbb{X}=\BX_{\bar{R}}. 
\end{equation*}
We set
\begin{equation*}
\tilde{\rho}_{0}: X_{0,\bar{R}}\xrightarrow{\rho_{0}} \BX_{0, \bar{R}}\xrightarrow{F_{\BX_{0}, \tau_{E}}} \tau_{E \ast}\BX_{0, \bar{R}}, \quad \text{ resp.}\quad \tilde{\rho}: X_{\bar{R}}\xrightarrow{\rho} \BX_{\bar{R}}\xrightarrow{F_{\BX, \tau_{E}}} \tau_{E \ast}\BX_{\bar{R}}.
\end{equation*}
 
The assignment 
\begin{equation}\label{Weil-map-special}
(X_{0}, \iota_{0}, \lambda_{0}, \rho_{0}, X, \iota, \lambda_{X}, \rho) \mapsto (X_{0}, \iota_{0}, \lambda_{0}, \tilde{\rho}_{0}, X, \iota, \lambda_{X}, \tilde{\rho})
\end{equation}
defines a map
\begin{equation}\label{WD-diamond-special}
\omega_{\wt{\CM}_{\bK^{\circ}_{\wt G}}}: \wt{\CM}_{\bK^{\circ}_{\wt G}}(i)(\bar{R})\rightarrow \wt{\CM}_{\bK^{\circ}_{\wt G}}(i+f_{E})(\bar{R}_{[\tau_{E}]}).
\end{equation}
Here we note that the inverse image
of the polarization $(\tau_{E})_{\ast} \lambda_{\mathbb{X}}$ on
$(\tau_{E})_{\ast} \mathbb{X}$ by the Frobenius isogeny 
$$F_{\mathbb{X},\tau_E}: \mathbb{X} \longrightarrow (\tau_{E})_{\ast} \mathbb{X}$$
is $p^{f_{E}}\lambda_{\mathbb{X}}$ and similarly for $\BX_{0}$ and $\lambda_{\BX_{0}}$. 

From \eqref{WD-diamond-special} we obtain the Weil  descent datum
\begin{equation}\label{WD7e}
\omega_{\widetilde{\CM}_{\bK^{\circ}_{\wt G}}}: \widetilde{\CM}_{\bK^{\circ}_{\wt G}} \longrightarrow
{\widetilde{\CM}^{(\tau_{E})}_{\bK^{\circ}_{\wt G}}} ,
\end{equation}
where the upper index $(\tau_{E})$ denotes the base change via
$\Spec \tau_{E} : \Spf O_{\breve{E}} \longrightarrow \Spf O_{\breve{E}}$.

We recall  from \cite{KRZI} the definition of the formal moduli space ${\CM}_{r}$ of local CM-triples associated to the group $G$, cf. \cite[Definition 2.6.1]{KRZI}.  We denote by $O_{E_{r}}$ the valuation ring of $E_{r}$ and by $f_{E_{r}}$  the inertia degree of $E_{r}$ over $\BQ_{p}$. Let $\kappa_{r}$ be residue class field of $E_{r}$.
 
\begin{definition}\label{recall-M-r}
We define for each $i\in \BZ$ the functor $\CM_{r}(i)$ on $\Nilp_{O_{\breve{E}_{r}}}$. For a scheme $S$ over $\Spf O_{\breve{E}_{r}}$, a point of $\CM_{r}(i)(S)$ consists of the following data:
\begin{enumerate}
\item[(1)] A local CM-triple $(X, \iota, \lambda_{X})$ of type
$(K/F ,r)$ over $S$ which satisfies the conditions $({\rm KC}_{r})$ and $({\rm EC}_{r})$ and is $(V, \varsigma)$-principally polarized;

\item[(2)] An $O_{K}$-linear quasi-isogeny 
\begin{equation*}
\rho: \bar{X} := X\times_{S} \bar{S} \longrightarrow \mathbb{X}\times_{\Spec \bar{\kappa}_{r}} \bar{S} 
\end{equation*}
such that, locally on $\bar S$,  there exists $u\in\BZ_p^\times$ with  $\rho^{\ast}(\lambda_{\BX})=up^{i}{\lambda}_{\bar{X}}$.
\end{enumerate}
We denote these data by a tuple $(X, \iota, \lambda_{X}, \rho)$. Two tuples $(X, \iota, \lambda_{X}, \rho)$ and $(X^{\prime}, \iota^{\prime}, \lambda_{X^{\prime}}, \rho^{\prime})$ define the same point of $\CM_{r}(i)$ iff there exists an isomorphism  of $O_{K}$-modules $\alpha: (X, \iota)\isoarrow (X^{\prime}, \iota^{\prime})$ such that $\rho^{\prime}\circ\alpha_{\bar{S}}=\rho$.  Note that the requirement on $\rho$ in $(2)$ above implies that 
\begin{equation*}
2\mathrm{height}(\rho)=\mathrm{height}(p^{i}\mid X)=4di.
\end{equation*}

The formal scheme over $\Spf O_{\breve{E}_{r}}$ representing this functor is also denoted by $\CM_{r}(i)$.  
We define the formal scheme $\CM_{r}$ over $\Spf O_{\breve{E}_{r}}$ by
\begin{equation}
\CM_{r}=\bigsqcup\nolimits_{i\in\BZ}\CM_{r}(i). 
\end{equation}
\end{definition}
Then $\CM_r$ parametrizes tuples $(X, \iota, \lambda_{X}, \rho)$ such that $\rho$ respects the polarization $\lambda_{\bar{X}}$ on $\bar{X}$ and
${\lambda}_{\mathbb{X}}$ on $\BX$ up to a factor in $\mathbb{Q}_p^{\times}$. We define a map 
\begin{equation}\label{frakvr}
\mathfrak{v}_{r}: {\CM}_{r}\rightarrow \BZ
\end{equation}
sending the tuple $(X, \iota, \lambda_{X}, \rho)$ to $(2d)^{-1}\mathrm{height}(\rho)$ and hence $\mathfrak{v}_{r}({\CM}_{r}(i))=i$. The group $J(\BQ_{p})$ acts naturally on $\CM_{r}$ via the framing $\rho$ by $(X, \iota, \lambda_{X}, \rho)\mapsto (X, \iota, \lambda_{X}, \alpha\circ\rho)$.

Next, we consider the Weil descent datum on $\CM_{r}$. It is enough to restrict the functor $\CM_{r}$ to an affine scheme $S = \Spec R$ over $\Spf O_{\breve{E}_{r}}$. We write
$\varepsilon: O_{\breve{E}_{r}} \longrightarrow R$ for the given algebra structure. Let $\tau_{E_{r}}$ be the Frobenius automorphism of $\bar{\kappa}_{r}$ over $\kappa_{r}$. We write  $R_{[\tau_{E_{r}}]}$ for the ring $R$ with the new algebra structure
$\varepsilon \circ\tau_{E_{r}}$.  
By base change to $\bar{\kappa}_{r}$, we obtain
\begin{equation*}
\bar{\varepsilon}: \bar{\kappa}_{r} \longrightarrow \bar{R}:= R \otimes_{O_{\breve{E}_{r}}} \bar{\kappa_{r}}. 
\end{equation*}
    
Let $(X, \iota, \lambda_{X}, \rho)$ be a point of $\CM_{r}(\bar{R})$  and we set
\begin{equation*}
\tilde{\rho}: X_{\bar{R}}\overset{\rho}\longrightarrow \BX_{\bar{R}}\xrightarrow{F_{\BX, \tau_{E_{r}}}} \tau_{E_{r} \ast}\BX_{\bar{R}}.
\end{equation*}
 
The assignment 
\begin{equation}\label{Weil-map-special}
(X, \iota, \lambda_{X}, \rho) \mapsto (X, \iota, \lambda_{X}, \tilde{\rho})
\end{equation}
defines a map
\begin{equation}\label{WD-Mr-i}
\omega_{\CM_{r}}: \CM_{r}(i)(\bar{R})\rightarrow \CM_{r}(i+f_{E_r})(\bar{R}_{[\tau_{E_{r}}]}).
\end{equation}
Here we note that the inverse image
of the polarization $(\tau_{E_{r}})_{\ast} \lambda_{\mathbb{X}}$ on
$(\tau_{E_{r}})_{\ast} \mathbb{X}$ by the Frobenius isogeny 
\begin{equation*}
F_{\mathbb{X},\tau_{E_{r}}}: \mathbb{X} \longrightarrow (\tau_{E_{r}})_{\ast} \mathbb{X}
\end{equation*}
is $p^{f_{E_{r}}}\lambda_{\mathbb{X}}$. Then  \eqref{WD-Mr-i} defines the Weil descent datum
\begin{equation*}\label{WD-Mr}
\omega_{\CM_{r}}: \CM_{r}\rightarrow \CM_{r}^{(\tau_{E_{r}})}.
\end{equation*}

We fix an isomorphism $J_{\mathrm{ad}}(\BQ_{p})\simeq \mathrm{PGL}_{2}(F)$. Then ${J}(\BQ_{p})$ acts on $\hat{\Omega}_{F} \times_{\Spf O_{F}} \Spf O_{\breve{E}_{r}}$ via this isomorphism. We extend this to an action of $J(\BQ_p)$ on $(\hat{\Omega}_{F} \times_{\Spf O_{F}} \Spf O_{\breve{E}_{r}})\times \BZ$ by the translation by $\val(\mu(\alpha))$ for $ \alpha\in J(\BQ_p)$ on the second factor. 

\begin{proposition}\label{Mr-thm}
There exists an isomorphism of formal schemes over $\Spf O_{\breve{E}_{r}}$
\begin{equation*}
\CM_{r}\overset{\sim}\longrightarrow (\hat{\Omega}_{F} \times_{\Spf O_{F}} \Spf O_{\breve{E}_{r}})\times \BZ
\end{equation*}
which is equivariant for the action of $J(\BQ_p)$ and such that the Weil descent datum $\omega_{\CM_{r}}$ on the left hand side induces on the right hand side the Weil descent datum 
\begin{equation*}
(\xi, i)\mapsto (\omega_{\Omega, E_{r}}(\xi), i+f_{E_r}).
\end{equation*}
Here $\omega_{\Omega, E_{r}}$ denotes the action of $\tau_{E_{r}}$ on the formal scheme $\hat{\Omega}_{F} \times_{\Spf O_{F}} \Spf O_{\breve{E}_{r}}$ through the second factor. 
\end{proposition}
\begin{proof}
This is proved in \cite[Corollary 6.1.3]{KRZI} in the case when $K/F$ is ramified and in  \cite[Corollary 6.2.5]{KRZI} when $K/F$ is unramified.
\end{proof}

Recall the space $\mathcal{M}_{0}(i)$ in Definition \ref{RSZ-LS-Tori}  and the space $\mathcal{M}_{0}$ in \eqref{M0}. We have a morphism of formal moduli spaces over $\Spf O_{\breve{E}}$
\begin{equation}\label{vartheta-special}
\vartheta: \widetilde{\mathcal{M}}_{\mathbf{K}^{\circ}_{\wt G}} \rightarrow \mathcal{M}_{r, O_{\breve{E}}}\times {\mathcal{M}}_{0, O_{\breve{E}}} ,
\end{equation}
given by sending $(X_{0}, \iota_{0}, \lambda_{0}, \rho_{0}, X, \iota, \lambda_{X}, \rho)$ to $(X, \iota, \lambda_{X}, \rho)\times (X_{0}, \iota_{0}, \lambda_{0}, \rho_{0})$. Here \[\mathcal{M}_{r, O_{\breve{E}}}=\mathcal{M}_{r}\times_{\Spf  O_{\breve{E}_r}}\Spf  O_{\breve{E}}, \quad  
{\mathcal{M}}_{0, O_{\breve{E}}}={\mathcal{M}}_{0}\times_{\Spf O_{\breve E_0}}\Spf O_{\breve{E}} ,
\]
with their induced Weil descent data $\omega_{\CM_r, E}$ and $\omega_{\CM_0, E}$ down to $O_E$. Also, there is an action 
 of $\wt{J}(\BQ_{p})$ on the right-hand side via  the map $\wt{J}(\BQ_{p})\rightarrow J(\BQ_{p})\times J_{\T}(\BQ_{p})$.

\begin{lemma}\label{wtM-immersion}
The morphism $\vartheta$ induces a closed immersion of formal schemes over $\Spf O_{\breve{E}}$,
\begin{equation*}\label{theta5}
\vartheta: \widetilde{\mathcal{M}}_{\mathbf{K}^{\circ}_{\wt G}} \rightarrow \mathcal{M}_{r, O_{\breve{E}}}\times {\mathcal{M}}_{0, O_{\breve{E}}} .
\end{equation*}
 In fact  
\begin{equation*}
\vartheta: \widetilde{\mathcal{M}}_{\mathbf{K}^{\circ}_{\wt G}}(i) \isoarrow \mathcal{M}_{r, O_{\breve{E}}}(i)\times {\mathcal{M}}_{0, O_{\breve{E}}}(i) , \quad \forall i\in\BZ .
\end{equation*}
Furthermore, $\vartheta$ is equivariant for the action of $\wt{J}(\BQ_{p})$ and  is compatible with the natural Weil descent data down to $O_E$ on both sides.
\end{lemma}
\begin{proof}
By Remark \ref{simi-factor-rmk}, the space $\widetilde{\mathcal{M}}_{\mathbf{K}^{\circ}_{\wt G}}(i)$ can be identified with the fiber product 
\begin{equation*}
\mathcal{M}_{r, O_{\breve{E}}}(i)\times {\mathcal{M}}_{0, O_{\breve{E}}}(i) . 
\end{equation*}
Hence $\vartheta$ is a closed embedding of formal schemes. The equivariance claim is obvious. 

Let $\Spec R$ be a scheme over $\Spf O_{\breve{E}}$.  We need to verify the commutativity of the following  diagram,
\begin{equation}\label{weil-M-Ku-diagram}
\begin{tikzcd}
\wt{\CM}_{\bK^{\circ}_{\wt G}}(i)(\bar{R}) \arrow[r, "\vartheta"] \arrow[d, "\omega_{\wt{\CM}_{\bK^{\circ}_{\wt G}}}"] & {\CM}_{r, O_{\breve{E}}}(i)\times\CM_{0, O_{\breve{E}}}(i)(\bar{R})  \arrow[d, "\omega_{{\CM}_{r}, E}\times \omega_{\CM_{0}, E}"] \\
\wt{\CM}_{\bK^{\circ}_{\wt G}}(i+f_{E})(\bar{R}_{[\tau_{E}]}) \arrow[r, "\vartheta"] &   {\CM}_{r, O_{\breve{E}}}(i+f_{E})\times\CM_{0, O_{\breve{E}}}(i+f_{E})(\bar{R}_{[\tau_{E}]}).\\         
\end{tikzcd}
\end{equation}
But this is clear by construction. 
\end{proof}

We consider the local reflex norm \eqref{reflex-norm-0} with respect to the datum $(\Phi^{+}_{0}, E)$:
\begin{equation}\label{reflex-norm-E}
\mathfrak{r}_{\Phi^{+}_{0}, E}: \mathrm{Res}_{E/\BQ_{p}}\mathbb{G}_{m, E}\longrightarrow \T.
\end{equation}
Let 
\begin{equation}\label{wtw}
\wt{w}={\mathfrak{r}}_{\Phi^{+}_{0}, E}(\varpi_{E})\in \T(\BQ_{p}) ,
\end{equation}
be the special element in $\T(\BQ_{p})$ relative to $(\Phi^{+}_{0}, E)$ for the uniformizer $\varpi_{E}$ as in Construction \ref{sp-elt-loc}.

\begin{proposition}\label{LS-weil-special}
There is an isomorphism
\begin{equation*}
\widetilde{\mathcal{M}}_{\mathbf{K}^{\circ}_{\wt G}} \overset{\sim}{\longrightarrow}
(\hat{\Omega}_{F} \times_{\Spf O_{F}} \Spf O_{\breve{E}}) \times
\wt G(\mathbb{Q}_p)/\mathbf{K}^{\circ}_{\wt G}
 \end{equation*}
compatible with natural actions of $\wt{J}(\BQ_{p})$ on both sides. Here $\wt{J}(\BQ_{p})$ acts on $(\hat{\Omega}_{F} \times_{\Spf O_{F}} \Spf O_{\breve{E}}) \times\wt G(\mathbb{Q}_p)/\mathbf{K}^{\circ}_{\wt G}$ via the composition $$\wt J(\BQ_p)\to J(\BQ_p)\to  J_{\mathrm{ad}}(\BQ_{p})\simeq \PGL_{2}(F)$$ on the first factor and via $ \wt J(\BQ_p)\to J_{\T}(\BQ_{p})\simeq \T(\BQ_{p})$ by multiplication on the second factor (which we identify with $\T(\BQ_{p})/\bfK^{\circ}_{\T}$, cf. \eqref{wtverT}).

The Weil descent datum $\omega_{\wt{\mathcal{M}}_{\mathbf{K}^{\circ}_{\wt G}}}$
on the left-hand side corresponds on the
right-hand side to the Weil descent datum given by
\begin{equation*}
(\xi, g) \mapsto (\omega_{\Omega, E}(\xi), \wt{w}g), \quad g \in \wt{G}(\mathbb{Q}_p) ,
\end{equation*} 
where $\wt{w}$ is the special element relative to $(\Phi^{+}_{0}, E)$ for our fixed uniformizer $\varpi_{E}$, viewed as an element in the center of $\wt{G}(\mathbb{Q}_p)$.
\end{proposition}

\begin{proof}
We first construct an isomorphism 
\begin{equation*}
\widetilde{\mathcal{M}}_{\mathbf{K}^{\circ}_{\wt G}} \overset{\sim}{\longrightarrow}
(\hat{\Omega}_{F} \times_{\Spf O_{F}} \Spf O_{\breve{E}}) \times
\wt G(\mathbb{Q}_p)/\mathbf{K}^{\circ}_{\wt G}.
 \end{equation*}
 We consider the composite map $\CM_{r}\overset{\sim}\longrightarrow (\hat{\Omega}_{F} \times_{\Spf O_{F}} \Spf O_{\breve{E}_{r}})\times \BZ\rightarrow \BZ$, where the last map is the projection to the second factor. By the construction of Proposition \ref{Mr-thm}, this map maps $\wt\CM_{r}(i)$ to $i$.  Therefore we have isomorphisms
\begin{equation*}
\begin{aligned}
\widetilde{\mathcal{M}}_{\mathbf{K}^{\circ}_{\wt G}} &\overset{\sim}{\longrightarrow} \CM_{r, O_{\breve{E}}}\times_{\BZ} \CM_{0, O_{\breve{E}}}\\
&\overset{\sim}{\longrightarrow}\big(\hat{\Omega}_{F} \times_{\Spf O_{F}} \Spf O_{\breve{E}} \times \BZ\big)\times_{\BZ} \CM_{0, O_{\breve{E}}}\\
&\overset{\sim}{\longrightarrow}\big(\hat{\Omega}_{F} \times_{\Spf O_{F}} \Spf O_{\breve{E}} \big)\times \CM_{0, O_{\breve{E}}}\\
&\overset{\sim}{\longrightarrow}\big(\hat{\Omega}_{F} \times_{\Spf O_{F}} \Spf O_{\breve{E}} \big)\times \T(\BQ_{p})/\bK^{\circ}_{\T}\\
&\overset{\sim}{\longrightarrow}\big(\hat{\Omega}_{F} \times_{\Spf O_{F}} \Spf O_{\breve{E}} \big)\times \wt{G}(\BQ_{p})/\bK^{\circ}_{\wt{G}}.\\
\end{aligned}
\end{equation*}

The equivariance claim is obvious. Let $\Spec R$ be a scheme over $\Spf O_{\breve{E}}$.  From Lemma \ref{wtM-immersion} and Proposition \ref{Mr-thm}, we  obtain the commutativity of the following  diagram 
\begin{equation}\label{weil-M-Ku-diagram}
\begin{tikzcd}
\wt{\CM}_{\bK^{\circ}_{\wt G}}(i)(\bar{R}) \arrow[r, "\sim"] \arrow[d, "\omega_{\wt{\CM}_{\bK^{\circ}_{\wt G}}}"] & \big(\hat{\Omega}_{F} \times_{\Spf O_{F}} \Spf O_{\breve{E}} \times\{i\}\big)\times_{\BZ} \CM_{0, O_{\breve{E}}}(i)(\bar{R})  \arrow[d, "\omega_{\Omega,E}\times \omega_{\CM_{0},E}"] \\
\wt{\CM}_{\bK^{\circ}_{\wt G}}(i+f_{E})(\bar{R}_{[\tau_{E}]}) \arrow[r, "\sim"] &  \big(\hat{\Omega}_{F} \times_{\Spf O_{F}} \Spf O_{\breve{E}} \times\{i+f_{E}\}\big)\times_{\BZ} \CM_{0, O_{\breve{E}}}(i+f_{E})(\bar{R}_{[\tau_{E}]}).         
\end{tikzcd}
\end{equation}

For any $j\in\BZ$, the right vertical map is given under the natural isomorphism 
\begin{equation*}
\big(\hat{\Omega}_{F} \times_{\Spf O_{F}} \Spf O_{\breve{E}} \times\{j\}\big)\times_{\BZ} \CM_{0, O_{\breve{E}}}(j)\simeq \big(\hat{\Omega}_{F} \times_{\Spf O_{F}} \Spf O_{\breve{E}} \big)\times \CM_{0, O_{\breve{E}}}(j) 
\end{equation*}
simply  by $\omega_{\Omega, {E}}\times \omega_{\CM_{0}, E}$.  Finally, by Proposition \ref{LS-tori-weil}, the descent datum $\omega_{\CM_{0}, E}$ corresponds to multiplication by the special element $\wt{w}$ under the isomorphism $\CM_{0, O_{\breve{E}}}\simeq \T(\BQ_{p})/\bK^{\circ}_{\T}$. This finishes the proof.
 \end{proof}

\subsection{Formal moduli space of RSZ CM-triples: the banal case}
Let $r$ be a local CM-type of rank $2$. We now assume that $r$ is banal, i.e., $r_{\varphi}\in\{0,2\}$ for all $\varphi\in \Phi$. Let $r_{0}$ still be a local CM-type of rank $1$ with associated local CM-type $\Phi^{+}_{0}$.

Let $\bfK^{\circ}_{\T}$ still be the maximal open compact subgroup of $\T(\BQ_{p})$ but now let $\bfK_{\U}$ be an arbitrary open compact subgroup of $\U(\BQ_{p})$. Let $\bfK_{\wt{G}}= \bfK_{\U}\times\bfK^{\circ}_{\T}$, it is an open compact subgroup of $\wt{G}(\BQ_{p})$. The open compact subgroup $\bfK_{\wt{G}}$ defines an open compact subgroup $\bK_{G}$ of $G(\BQ_{p})$ as the image of $\bfK_{\wt{G}}$ under the map $\wt{G}(\BQ_{p})\rightarrow G(\BQ_{p})$.

We fix a local CM-triple $(\mathbb{X}_{0}, \iota_{\mathbb{X}_{0}}, \lambda_{\mathbb{X}_{0}})$ of type $(K/F, r_{0})$ over $\bar{\kappa}$ satisfying the condition $({\rm KC}_{r_{0}})$ and a local CM-triple $(\mathbb{X}, \iota_{\mathbb{X}}, \lambda_{\mathbb{X}})$ of type $(K/F, r)$ over $\bar{\kappa}$ which satisfies the conditions $({\rm KC}_{r})$ and $({\rm EC}_{r})$ and is $(V, \varsigma)$-principally polarized.   We make the assumption $\inv^r(\mathbb{X}, \iota_{\mathbb{X}}, \lambda_{\mathbb{X}})=\inv(V)$. 
\begin{definition}\label{RSZ-LS-banal}
We define for each $i\in \BZ$ a functor $\wt{\CM}_{\bK_{\wt G}}(i)$ on $\Nilp_{O_{\breve{E}}}$. For a scheme $S$ over $\Spf O_{\breve{E}}$, a point of $\wt{\CM}_{\bK_{\wt G}}(i)(S)$ consists of the following data:
\begin{enumerate}
\item[(1)] A local CM-triple $(X_{0}, \iota_{0}, \lambda_{0})$ of type $(K/F, r_{0})$ satisfying the condition $({\rm KC}_{r_{0}})$;
      
\item[(2)] An $O_{K}$-linear quasi-isogeny
\begin{equation*}    
\rho_{0}: \bar{X}_{0} := X_{0} \times_{S} \bar{S} \longrightarrow \mathbb{X}_{0} \times_{\Spec \bar{\kappa}} \bar{S}
\end{equation*}
such that, locally on $\bar S$, there exists $u\in\BZ_p^\times$ with  $\rho_{0}^{\ast}(\lambda_{\BX_{0}})=up^{i}{\lambda}_{0}$;
\item[(3)] A local CM-triple $(X, \iota, \lambda_{X})$ of type $(K/F ,r)$ over $S$
which satisfies the conditions $({\rm KC}_{r})$ and $({\rm EC}_{r})$ and is $(V, \varsigma)$-principally polarized;

\item[(4)] An $O_{K}$-linear quasi-isogeny 
\begin{equation*}
\rho: \bar{X} := X \times_{S} \bar{S} \longrightarrow \mathbb{X}\times_{\Spec \bar {\kappa}} \bar{S} 
\end{equation*}
such that $\rho^{\ast}(\lambda_{\BX})=u p^{i}{\lambda_{\bar{X}}}$, locally on $\bar S$,  with the same  $u\in\BZ_p^\times$ as in datum $(2)$;    
\item[(5)] A class
  $\bar{\eta}$ of isomorphisms of \'etale sheaves 
  \begin{equation*}
    \eta: (\underline{\Hom}_{O_{K}}(X_{0}, X), \mathfrak{e})\isoarrow (\Lambda, \varsigma) \;
    \mathrm{mod}\; \bK_{\U}
  \end{equation*}
which respect the bilinear forms  on both sides.  
Here $(\underline{\Hom}_{O_{K}}(X_{0}, X)$ is the $p$-adic \'etale sheaf of \cite[Definition 7.4.1]{KRZI}, and the bilinear form $\mathfrak{e}:\underline{\Hom}_{O_{K}}(X_{0}, X)\times \underline{\Hom}_{O_{K}}(X_{0}, X)\rightarrow \BZ_{p}$ is defined as in \cite[Proposition  7.4.7]{KRZI}. 
\end{enumerate} 
\end{definition}

We denote these data by a tuple $(X_{0}, \iota_{0}, \lambda_{0},\rho_{0}, X, \iota, \lambda_{X}, \rho, \bar{\eta})$. We will call datum $(5)$ a {\it strict CL-level structure} on $(X_{0}, \iota_{0}, \lambda_{0},\rho_{0}, X, \iota, \lambda_{X}, \rho)$. Two tuples $(X_{0}, \iota_{0}, \lambda_{0}, \rho_{0}, X, \iota, \lambda_{X}, \rho)$ and $(X^{\prime}_{0}, \iota^{\prime}_{0}, \lambda^{\prime}_{0}, \rho^{\prime}_{0}, X^{\prime}, \iota^{\prime}, \lambda_{X^{\prime}}, \rho^{\prime}, \bar{\eta}^{\prime})$ define the same point of $\wt{\CM}_{\bK_{\wt G}}(i)$ if there exist isomorphisms  of $O_{K}$-modules  $\alpha_{0}: (X_{0}, \iota_{0})\isoarrow (X^{\prime}_{0}, \iota^{\prime}_{0})$ and $\alpha: (X, \iota)\isoarrow  (X^{\prime}, \iota^{\prime})$ such that $\rho^{\prime}_{0}\circ\alpha_{0, \bar{S}}=\rho_{0}$ and $\rho^{\prime}\circ\alpha_{\bar{S}}=\rho$, and which pulls back  $\bar{\eta'}$ to $\bar{\eta}$.

The formal scheme over $\Spf O_{\breve{E}}$ representing this functor is also denoted by $\wt{\CM}_{\bK_{\wt G}}(i)$. We note that Remark \ref{simi-factor-rmk} applies to this formal moduli space.
We define the formal scheme
\begin{equation}\label{wtCM-banal}
\wt{\CM}_{\bK_{\wt G}}=\bigsqcup\nolimits_{i\in\BZ}\wt{\CM}_{\bK_{\wt G}}(i). 
\end{equation}
Then $\wt{\CM}_{\bK_{\wt G}}$ parametrizes tuples $(X_{0}, \iota_{0}, \lambda_{0}, \rho_{0}, X, \iota, \lambda_{X}, \rho, \bar{\eta})$ such that $\rho_{0}$ respects the polarization $\lambda$ on $\bar{X}_{0}$ and ${\lambda}_{\mathbb{X}_{0}}$ on $\BX_{0,\bar{S}}$ up to a factor in $\mathbb{Q}_p^{\times}$ and $\rho$ respects the polarization $\lambda_{\bar{X}}$ on $\bar{X}$ and ${\lambda}_{\mathbb{X}}$ on $\BX_{\bar{S}}$ up to the same factor in $\mathbb{Q}_p^{\times}$.

The framing objects define as in \eqref{def:Js} the groups $J_{\T}(\BQ_{p})$ and $J(\BQ_{p})$ and $\wt{J}(\BQ_{p})$. Note, as in the special case,  that $\wt{\CM}_{\bK_{\wt G}}$ affords an action of $\wt{J}(\BQ_{p})$. However, in the present banal case  $J(\BQ_{p})$ is isomorphic to the unitary similitude group $G(\BQ_{p})$, cf.  \cite[(7.2.8)]{KRZI}. Also, as before,   $J_{\T}(\BQ_{p})$ is isomorphic to $\T(\BQ_{p})$, see \eqref{Ti-group}. Therefore, in the banal case, we have an isomorphism $\wt{G}(\BQ_{p})\simeq \wt{J}(\BQ_{p})$. For varying $\bK_{\wt{G}}$, the group $\wt{G}(\BQ_{p})$ acts on the tower of formal schemes $\wt{\CM}_{\bK_{\wt G}}$ in the usual way.

Let $S = \Spec R$ be an affine scheme over $\Spf O_{\breve{E}}$. The same formula as in \eqref{Weil-map-special} defines the map 
\begin{equation}\label{WD-diamond-banal}
\omega_{\wt{\CM}_{\bK_{\wt G}}}: \wt{\CM}_{\bK_{\wt G}}(i)(\bar{R})\rightarrow \wt{\CM}_{\bK_{\wt G}}(i+f_{E})(\bar{R}_{[\tau_{E}]})
\end{equation}
and hence defines the Weil descent datum
\begin{equation}\label{WD-banal}
\omega_{\widetilde{\CM}_{\bK_{\wt G}}}: \widetilde{\CM}_{\bK_{\wt G}} \longrightarrow
{\widetilde{\CM}^{(\tau_{E})}_{\bK_{\wt G}}}
\end{equation}
on $\widetilde{\CM}_{\bK_{\wt G}}$.

We recall the definition of the formal moduli space ${\CM}_{r, \bK_{G}}$ of local CM triples of type $(K/F, r)$ when $r$ is banal, cf.  \cite[Definition 7.4.3]{KRZI}.
\begin{definition}\label{recall-M-r-banal}
We define for each $i\in \BZ$ a functor $\CM_{r, \bK_{G}}(i)$ on $\Nilp_{O_{\breve{E}_{r}}}$. For a scheme $S$ over $\Spf O_{\breve{E}_{r}}$, a point of $\CM_{r, \bK_{G}}(i)(S)$ consists of the following data:
\begin{enumerate}
\item[(1)] A local CM-triple $(X, \iota, \lambda_{X})$ of type
$(K/F ,r)$ over $S$ which satisfies the conditions $({\rm KC}_{r})$ and $({\rm EC}_{r})$ and is $(V, \varsigma)$-principally polarized;

\item[(2)] An $O_{K}$-linear quasi-isogeny 
\begin{equation*}
\rho: \bar{X} := X\times_{S} \bar{S} \longrightarrow \mathbb{X}\times_{\Spec \bar {\kappa}} \bar{S} 
\end{equation*}
such that, locally on $\bar S$, there exists $u\in\BZ_p^\times$ with  $\rho_{0}^{\ast}(\lambda_{\BX})=up^{i}{\lambda}_{\bar{X}}$;

\item[(3)] A class
 $\bar{\eta}_{X}$ of isomorphisms of \'etale sheaves
\begin{equation*}
{\eta}_{X}: (\underline{\Hom}_{O_{K}}(\wt\BX_{0}, X), \mathfrak{e}_{X})\isoarrow (\Lambda, \varsigma) \;
\mathrm{mod}\; \mathbf{K}_{G}
\end{equation*}
which respect the bilinear forms on both sides up to a unit in $\BZ_{p}$. Here $\wt\BX_{0}$ is the lift of $\BX_{0}$ over $O_{\breve{E}_{r}}$ as in \cite[(7.4.14)]{KRZI} (where it is denoted by $X_0$) and $(\underline{\Hom}_{O_{K}}(\wt{\BX}_{0}, X)$ is the $p$-adic \'etale sheaf of \cite[Definition 7.4.1]{KRZI} and  we define the alternating form $\mathfrak{e}_{X}$ as in \cite[Proposition 7.4.7]{KRZI}. 
\end{enumerate}
 We denote these data by a tuple $(X, \iota, \lambda_{X}, \rho, \bar{\eta}_{X})$. The datum  $(3)$ is called a {\it (relaxed) CL-level structure} on $(X, \iota, \lambda, \rho)$. Here we sometimes use the qualifier ``relaxed'' to distinguish from a strict CL-level structure. Two tuples $(X, \iota, \lambda_{X}, \rho, \bar{\eta}_{X})$ and $(X^{\prime}, \iota^{\prime}, \lambda_{X^{\prime}}, \rho^{\prime}, \bar{\eta}^{\prime}_{X^{\prime}})$ define the same point of $\CM_{r, \bK_{G}}(i)$ iff there exists an isomorphism  of $O_{K}$-modules  $\alpha: (X, \iota)\isoarrow  (X^{\prime}, \iota^{\prime})$ such that $\rho^{\prime}\circ\alpha_{\bar{S}}=\rho$ and such that the  pullback of $\bar\eta'_{X'}$ under  $\alpha$ is equal to $\bar\eta_{X}$.  Note that the requirement on $\rho$ in $(2)$ above implies that 
\begin{equation*}
2\mathrm{height}(\rho)=\mathrm{height}(p^{i}\mid X)=4di. 
\end{equation*}

The formal scheme over $\Spf O_{\breve{E}_{r}}$ representing this functor is also denoted by $\CM_{r}(i)$.  We define the formal scheme $\CM_{r, \bK_{G}}$ over $\Spf O_{\breve{E}_{r}}$ by
\begin{equation}
\CM_{r, \bK_{G}}=\bigsqcup\nolimits_{i\in\BZ}\CM_{r, \bK_{G}}(i). 
\end{equation}
\end{definition}
Then $\CM_{r, \bK_{G}}$ parametrizes tuples $(X, \iota, \lambda_{X}, \rho, \bar{\eta}_{X})$ such that $\rho$ respects the polarization $\lambda_{\bar{X}}$ on $\bar{X}$ and ${\lambda}_{\mathbb{X}}$ on $\BX$ up to a factor in $\mathbb{Q}_p^{\times}$, and a CL level structure on $X$. 
 We define a map 
\begin{equation}\label{frakvr}
\mathfrak{v}_{r}: {\CM}_{r, \bK_{G}}\rightarrow \BZ
\end{equation}
sending the tuple $(X, \iota, \lambda_{X}, \rho, \bar{\eta}_{X})$ to $(2d)^{-1}\mathrm{height}(\rho)$ and hence $\mathfrak{v}_{r}({\CM}_{r, \bK_{G}}(i))=i$.

The group $J(\BQ_{p})$ acts naturally on $\CM_{r, \bK_{G}}$ via the framing $\rho$ by $(X, \iota,\lambda_{X}, \rho,  \bar{\eta}_{X})\mapsto (X, \iota,\lambda_{X}, \alpha\circ\rho, \bar{\eta}_{X})$. For varying $\bK_{G}$, the group $G(\BQ_{p})$ acts on the tower of formal schemes $\CM_{r, \bK_{G}}$

The same formula as \eqref{Weil-map-special} defines a Weil descent datum $\omega_{\CM_r}$ down to $O_{E_r}$. We denote by $\CM_{r, \bK_{G}, O_{\breve{E}}}$ the formal scheme $\CM_{r, \bK_{G}}\times_{\Spf O_{E_r}}\Spf O_{\breve E}$, with its Weil descent datum $\omega_{\CM_r, E}$ down to $O_{E}$.

\begin{proposition}\label{Mr-banal}
There exists an isomorphism of formal schemes over $O_{\breve{E}}$
\begin{equation*}
\CM_{r, \bK_{G},  O_{\breve{E}}}\overset{\sim}\longrightarrow G(\BQ_{p})/\bK_{G}
\end{equation*}
which is equivariant with respect to the action of $J(\BQ_{p})$ on both sides.  For varying $\bK_{G}$, the isomorphism is equivariant for the natural Hecke actions of $G(\BQ_{p})$ on both sides.

This isomorphism is compatible with the Weil descent datum $\omega_{\CM_{r}, E}$ down to $O_E$ on the left hand side and the Weil descent datum on the right hand side given by multiplication by the special element $\wt{w}$ relative to $(\Phi^{+}_{0}, E)$  for our fixed uniformizer $\varpi_{E}$,  viewed as an element in the center of ${G}(\mathbb{Q}_p)$.
\end{proposition}

This statement is to be compared with \cite[Corollaries 6.3.4, 6.4.5, 6.5.3]{KRZI}. In loc.~cit., the Weil descent data of $\CM_r$ down to $O_{E_r}$ is not explicit. The key point here is that we can make explicit the Weil descent datum down to $O_E$. To prove Proposition \ref{Mr-banal}, we need to go back to \cite[\S 6]{KRZI}.

Let $S$ be a scheme over $\Spf O_{\breve{E}_{r}}$. Given a point $(X, \iota, \lambda_{X}, \rho, \bar{\eta}_{X})\in {\CM}_{r, \bK_{G}}(S)$, we consider the tuple $(C, \iota, \phi,\rho, \bar{\eta}_{X})$ obtained by applying to it the polarized contraction functor $\mathfrak{C}_{r}^{\rm pol}$ \cite[Theorem  4.5.11]{KRZI}. Let $(\underline{C}_{\mathbb{X}}, \iota_{\mathbb{X}}, \phi_{\mathbb{X}})$ be the image of $(\mathbb{X}, \iota_{\mathbb{X}}, \lambda_{\mathbb{X}})$ under $\mathfrak{C}_{r}^{\rm pol}$. Here $\underline{C}_{\BX}$ is a constant $p$-adic \'etale sheaf on $\Spec W(\bar{\kappa})$ associated to a module $C_{\BX}$ of $\BZ_{p}$-rank $4d$ equipped with an action $\iota_\BX: O_{K}\rightarrow \End_{\BZ_{p}}(C_{\BX})$, and $\phi_\BX: C_{\BX}\times C_{\BX}\rightarrow O_{F}$ is an $O_{F}$-linear alternating form such that $\phi_{\BX}(\iota(a) c_1, c_2) = \phi (c_1, \iota(\bar{a})c_2)$ for
  $c_1, c_2 \in C_{\BX}$ and $a \in O_{K}$ and such that $\phi_{\BX}$ is perfect unless $K/F$ is unramified and $\inv(V)=-1$ in which case $\mathrm{ord}_{\pi}(\mathrm{det}(\phi_{\BX}))=2$.

 For $i\in\BZ$,  we consider the functor $\mathcal{G}_{r, \bK_{G}}(i)$ on the category of schemes $S$ over $\Spf O_{\breve{E}_{r}}$. A point of $\mathcal{G}_{r, \bK_{G}}(i)$ is given
by the following data:
\begin{enumerate}
\item[(1)]A locally
  constant $p$-adic \'etale sheaf $C$ on $S$ which is
  $\mathbb{Z}_p$-flat with $\rank_{\mathbb{Z}_p} C = 4d$ and with an action
\begin{equation*}
\iota: O_{K} \longrightarrow \End_{\mathbb{Z}_p} C;
\end{equation*}

\item[(2)]An alternating $O_F$-bilinear pairing
  \begin{equation*}
  \phi: C \times C \longrightarrow O_F, 
  \end{equation*}
  such that $\phi(\iota(a) c_1, c_2) = \phi (c_1, \iota(\bar{a})c_2)$ for
  $c_1, c_2 \in C$ and $a \in O_{K}$ and such that $\phi$ is perfect unless $K/F$ is unramified and $\inv(V)=-1$ in which case $\mathrm{ord}_{\pi}(\mathrm{det}(\phi))=2$;
  
\item[(3)]A quasi-isogeny of $O_K$-module sheaves on $S$ 
\begin{equation*}
\rho: (C, \iota) \longrightarrow (\underline{C}_{\mathbb{X}}, \iota_{\BX})
\end{equation*}
such that $\rho$ respects the bilinear form $p^{i}\phi$ on $C$ and the bilinear form $\phi_{\BX}$ on $\underline{C}_{\mathbb{\BX}}$ up to unit in $\BZ_{p}$;

\item[(4)]A class $\bar{\eta}_{C}$ of isomorphisms of $p$-adic \'etale sheaves on $\bar{S}$ 
\begin{equation*}
{\eta}_{C}: (C, \xi)\isoarrow (\Lambda, \varsigma)\; \mathrm{mod}\; \mathbf{K}_{G}
\end{equation*}
which respect the bilinear forms on both sides up to a unit in $\BZ_{p}$. Here we define the alternating form $\xi: C\times C\rightarrow \BZ_{p}$ by $\xi:=\mathrm{Tr}_{F/\BQ_{p}}\vartheta^{-1}\phi$ where $\vartheta$ is the different of the extension $F/\BQ_{p}$.
\end{enumerate}

We denote these data by a tuple $(C, \iota, \phi,\rho, \bar{\eta}_{C})$ and refer to the datum $(4)$ as a (relaxed) CL-level structure on $(C, \iota, \phi,\rho)$. Another set of data $(C', \iota',\phi', \rho', \bar{\eta}^{\prime}_{C'})$ defines the same point
iff there is an isomorphism $\alpha: (C,\iota) \isoarrow (C^{\prime}, \iota^{\prime})$ such that
$\rho' \circ \alpha = \rho$ and such that the pullback  of $\bar\eta'_{C'} $ under $\alpha$ is equal to $\bar\eta_{C}$. Then $\alpha$ respects $\phi$ and $\phi'$ up
to a factor in $\BZ^{\times}_{p}$. 
We set 
 \begin{equation}
 \mathcal{G}_{r, \bK_{G}}=\coprod_i \mathcal{G}_{r, \bK_{G}}(i) .
 \end{equation}

\begin{proof}[Proof of Proposition \ref{Mr-banal}]
The polarized contraction functor $\mathfrak{C}_{r}^{\rm pol}$ of \cite[Theorem  4.5.11]{KRZI} gives an isomorphism of functors between $\CM_{r, \bK_{G}}$ and $\mathcal{G}_{r, \bK_{G}}$ (here we have used \cite[Proposition 7.4.7]{KRZI} for the equivalence of CL-level structures on both sides). On the other hand, the proof of \cite[(7.4.11)]{KRZI} shows that there is a natural isomorphism of formal schemes over $\Spf O_{\breve{E}_{r}}$
\begin{equation*}
\CG_{r, \bK_{G}}\simeq G(\mathbb{Q}_p)/\mathbf{K}_{G} .
\end{equation*}
 Combining the two isomorphisms, we obtain  an isomorphism of formal schemes over $\Spf O_{\breve{E}_{r}}$ 
\begin{equation*}
\CM_{r, \bK_{G}}\overset{\sim}\longrightarrow \CG_{r, \bK_{G}}\simeq G(\mathbb{Q}_p)/\mathbf{K}_{G} ,
\end{equation*}
 which is equivariant for the action of $J(\BQ_{p})\simeq G(\BQ_{p})$. For varying $\bK_{G}$, this isomorphism is  by construction equivariant for the natural Hecke actions of $G(\BQ_{p})$ on both sides.

Next we explain the compatibility of Weil descent data down to $O_E$.  Let $S = \Spec R$ be an affine scheme over $\Spf O_{\breve{E}}$ and $(X, \iota, \lambda_{X}, \rho, \bar{\eta}_{X})$ be a point on $\CM_{r, \bK_{G}, O_{\breve{E}}}(i)(\bar{R})$ for some integer $i$. Let the quasi-isogeny
\begin{equation}\label{Frob}
X_{\bar{R}}\overset{\rho}\longrightarrow \BX_{\bar{R}}\xrightarrow{F_{\BX, \tau_{E}}} \tau_{E\ast}\BX_{\bar{R}}
\end{equation}
correspond via the contraction functor $\mathfrak{C}_{r}^{\rm pol}$  to the quasi-isogeny 
\begin{equation*}
C\overset{\rho}\longrightarrow \underline{C}_{\BX}\xrightarrow{F_{\BX, \tau_{E}}} \tau_{E\ast}\underline{C}_{\BX}\simeq \underline{C}_{\BX}
\end{equation*}
for a tuple $(C, \iota, \phi,  \rho:C \rightarrow C_{\BX}, \bar{\eta}_{C})$ in $\CG_{r}(\bar{R})$. We want to make the map $\rho\mapsto F_{\BX, \tau_{E}}\circ\rho$ explicit. 

We recall the following notation from the beginning of \S \ref{ss:specialbanal}.  
In the case where $K/F$ is a ramified extension of local fields
(ramified case) we choose a prime element $\Pi \in O_K$ 
such that
$\bar{\Pi} = - \Pi$. Then $\pi = -\Pi^2$ is a prime element 
of $F$. In the case where $K/F$ is an unramified extension of local
fields (unramified case) or where $K \cong F \times F$ (split case) we
choose a prime element $\pi \in F$ and we set $\Pi = \pi$. 

The map sending the quasi-isogeny $C\overset{\rho}\longrightarrow \underline{C}_{\BX}$ to the composite 
$C\overset{\rho}\longrightarrow \underline{C}_{\BX}\xrightarrow{F_{\BX, \tau_{E}}} \underline{C}_{\BX}$
corresponds to the map on $\CG_{r, \bK_{G}, O_{\breve{E}}}$ given  
\begin{enumerate}
\item when $K/F$ is ramified, by
\begin{equation*}
\CG_{r, \bK_{G}, O_{\breve{E}}}(i)(\bar{R})\xrightarrow{\Pi^{ef_{E}}}\CG_{r, \bK_{G}, O_{\breve{E}}}(i+ef_{E})(\bar{R}),
\end{equation*}
cf. \cite[Lemma 6.3.3]{KRZI};
\item
 when $K/F$ is unramified, by
 \begin{equation*}
\CG_{r, \bK_{G}, O_{\breve{E}}}(i)(\bar{R})\xrightarrow{\pi^{ef_{E}/2}}\CG_{r, \bK_{G}, O_{\breve{E}}}(i+ef_{E})(\bar{R}) ,
\end{equation*}
cf.  \cite[Lemma 6.4.4]{KRZI};
\item when $K/F$ is split, by 
 \begin{equation*}
\CG_{r, \bK_{G}, O_{\breve{E}}}(i)(\bar{R})\xrightarrow{(\pi^{a_{1,E}}, \pi^{a_{2,E}})}\CG_{r, \bK_{G}, O_{\breve{E}}}(i+ef_{E})(\bar{R}) ,
\end{equation*}
cf.  \cite[Lemma 6.5.2]{KRZI}. Here $2a_{i}=\dim \BX_{i}$ and $fa_{i, E}=f_{E}a_{i}$ for $i\in\{1,2\}$.
\end{enumerate}
Here these maps over the arrows are the natural translation maps on the moduli space $\CG_{r, O_{\breve{E}}}$ given by multiplication by the corresponding element of $K$. More precisely, in loc.~cit., the Weil descent datum down to $O_{E_r}$ is given. We use that the elements for the  Weil descent datum down to $O_{E}$ are obtained from these elements by raising them to the $f_E/f_{E_r}$-th power (for the split case, use the identity $f a_{i, E_r}=f_{E_r} a_i$, cf. \cite[below  (6.5.1)]{KRZI}). 
We denote these elements by $\wt{w}^{\prime}$.  By \cite[Remark 7.2.6]{KRZI}, we can modify the element $\wt{w}^{\prime}$ by a unit in $O_{K}$ without affecting the conclusion. Therefore, we can replace the element $\wt{w}^{\prime}$ by the element $\wt{w}$ appearing in the statement of the proposition. Indeed, $\wt{w}^{\prime}$ has norm $\Nm_{K/F}(\wt{w}^{\prime})=\pi^{ef_{E}}$ and  $\wt{w}$ has norm $\Nm_{K/F}(\wt{w})=up^{f_{E}}$ for some unit $u\in\BZ^{\times}_{p}$. The proposition follows.
\end{proof}

We have a morphism of formal moduli spaces over $\Spf O_{\breve{E}}$
\begin{equation}\label{thetaba}
\vartheta: \widetilde{\mathcal{M}}_{\mathbf{K}_{\wt G}} \rightarrow \mathcal{M}_{r, \bK_{G}, O_{\breve{E}}}\times {\mathcal{M}}_{0, O_{\breve{E}}} ,
\end{equation}
sending $(X_{0}, \iota_{0}, \lambda_{0}, \rho_{0}, X, \iota, \lambda_{X}, \rho, \bar{\eta})$ to $(X, \iota, \lambda_{X}, \rho, \bar{\eta}_{X})\times (X_{0}, \iota_{0}, \lambda_{0}, \rho_{0})$. 
Here  we define the relaxed CL-level structure $\bar{\eta}_{X}$ as follows. Given $(X_{0}, \iota_{0}, \lambda_{0}, \rho_{0})$, we can associate to it a unique class of isomorphisms of $p$-adic \'etale sheaves $\eta_{0}: C_{\wt{\BX}_{0}}\xrightarrow{\sim} C_{X_{0}}\mod \mathbf{K}^{\circ}_{\T}$ which respect the natural bilinear forms on both sides up to a unit in $\BZ_{p}$ (one should regard this as a CL-level structure on $(X_{0}, \iota_{0}, \lambda_{0}, \rho_{0})$). To see this, recall from \cite[(7.4.14)]{KRZI} that $C_{\wt\BX_0}=O_K$. Consider the moduli problem of ${\mathcal{M}}_{0, O_{\breve{E}}}$ in \S \ref{s:Cmod}. The proof of Proposition \ref{LS-tori-weil} shows that there is a unique $g\in \mathrm{T}(\BQ_{p})/\bK^{\circ}_{\T}$ such that there is an isomorphism $(gC_{\wt{\BX}_{0}}, (p^{i})^{-1}\xi_{\wt{\BX}_{0}})\simeq (C_{X_{0}}, \xi_{X_{0}})$ which respects the bilinear forms up to a unit in $\BZ_{p}$, where $i$ is determined by $\xi_{\wt\BX_{0}}(gc_{1}, gc_{2})=p^{i}\xi_{\BX_{0}}(c_{1}, c_{2})$ for $c_{1}, c_{2}$ in $C_{\wt\BX_{0}}$. Since multiplication by $g$ gives an isomorphism  $(C_{\wt{\BX}_{0}}, \xi_{\wt{\BX}_{0}})\simeq (gC_{\wt{\BX}_{0}}, (p^{i})^{-1}\xi_{\wt{\BX}_{0}})$, the composition of the above two isomorphisms give us the desired isomorphism 
$\eta_{0}: C_{\wt{\BX}_{0}}\xrightarrow{\sim} C_{X_{0}}\mod \mathbf{K}^{\circ}_{\T}$ which respects the natural bilinear forms on both sides up to a unit $u$ in $\BZ_{p}$. We note that this unit is the same as the one appearing in datum $(2)$ of Definition \ref{RSZ-LS-Tori}.
 
Given an isomorphism of $p$-adic \'etale sheaves
\begin{equation*}
\eta: (\underline{\Hom}_{O_{K}}(X_{0}, X), \mathfrak{e})\isoarrow (\Lambda, \varsigma) 
\end{equation*}
which strictly respects the bilinear forms on both sides, we choose an isomorphism $\eta_{0}: C_{\wt{\BX}_{0}}\xrightarrow{\sim} C_{X_{0}}$ in the above class to obtain an isomorphism
\begin{equation*}
\eta_{X}: (\underline{\Hom}_{O_{K}}(\wt\BX_{0}, X), \mathfrak{e}_{X})\isoarrow (\Lambda, \varsigma) 
\end{equation*}
where $\mathfrak{e}_{X}$ is defined similarly as $\mathfrak{e}$. Here we use the fully faithfulness of the polarized contraction functor. This isomorphism respects the bilinear forms on both sides up to a unit in $\BZ_{p}$ and its class under the action of $\bK_{G}$ gives the desired relaxed CL-level structure on $(X, \iota, \lambda_{X}, \rho)$. One easily checks that choosing a different $\eta_{0}$ gives an isomorphism $\eta_{X}$ in the same $\bK_{G}$-orbit as the original one. This defines $\vartheta$.

Conversely, if we are given an isomorphism 
\begin{equation*}
\eta_{X}: (\underline{\Hom}_{O_{K}}(\wt\BX_{0}, X), \mathfrak{e}_{X})\isoarrow (\Lambda, \varsigma) 
\end{equation*}
which respects the bilinear forms on both sides up a unit in $\BZ_{p}$ and an isomorphism $\eta_{0}: C_{\wt{\BX}_{0}}\xrightarrow{\sim} C_{X_{0}}$ which respects the bilinear forms on both sides up to the same unit in $\BZ_{p}$, then we can reverse the above procedure to obtain a strict CL-level structure on $(X_{0}, \iota_{0}, \lambda_{0}, \rho_{0}, X, \iota, \lambda_{X}, \rho, \bar{\eta})$. Indeed, $(\underline{\Hom}_{O_{K}}(\wt\BX_{0}, X), \mathfrak{e}_{X})$ can be identified with $(C_{X}, \xi_{X})$ therefore we obtain an isomorphism $\eta_{X}: (C_{X}, \xi_{X})\simeq (\Lambda, \varsigma)$ which respects the bilinear forms on both sides up to a unit $u_1\in \BZ_{p}$. This unit is the same unit as the one in datum $(2)$ of Definition \ref{recall-M-r-banal}. Now using the isomorphism $\eta_{0}: C_{\wt\BX_{0}}\isoarrow C_{X_{0}}$, we can still identify $\underline{\Hom}_{O_{K}}(X_{0}, X)$ with $C_{X}$ but the bilinear form $\mathfrak{e}$ is identified with $\xi_{X}$ up to the inverse of the unit $u_2$ that comes from comparing the bilinear forms under $\eta_{0}: C_{\wt\BX_{0}}\simeq C_{X_{0}}$. We note that again $u_{2}$ is the same unit as the one appearing in datum $(2)$ of Definition \ref{RSZ-LS-Tori}. Therefore under the isomorphism $(\underline{\Hom}_{O_{K}}(X_{0}, X), \mathfrak{e})\simeq (\Lambda, \varsigma)$ the two bilinear forms differ by  $u^{-1}_{2}u_{1}$. In order to obtain a strict CL-level structure, we need to have $u_{1}=u_{2}$. In this case, we will say $\eta_{X}$ and $\eta_{0}$ are compatible. Suppose now  that $(X_{0}, \iota_{0}, \lambda_{0}, \rho_{0}, X, \iota, \lambda_{X}, \rho, \bar{\eta})$ lies on $\widetilde{\mathcal{M}}_{\mathbf{K}_{\wt G}}(i)$ for some $i$. Then $\eta_{X}$ and $\eta_{0}$ are  compatible by definition of the moduli problem, see datum $(2)$ and $(4)$ of Definition \ref{RSZ-LS-banal}.

The above discussion and Remark \ref{simi-factor-rmk} show that \eqref{thetaba} is a closed immersion of formal schemes, equivariant with respect to the action of $\wt{J}(\BQ_{p})$, 
and that this closed immersion is compatible with the natural Weil descent data down to $O_E$ on both sides. In fact, it shows that \eqref{thetaba} induces an isomorphism \
$$ \vartheta: \widetilde{\mathcal{M}}_{\mathbf{K}_{\wt G}}(i) \isoarrow \mathcal{M}_{r, \bK_{G}, O_{\breve{E}}}(i)\times {\mathcal{M}}_{0, O_{\breve{E}}}(i), \quad i\in\BZ,
$$ 
since the image of $\vartheta$ consists of those tuples $(X, \iota, \lambda_{X}, \rho, \bar{\eta}_{X})\times (X_{0}, \iota_{0}, \lambda_{0}, \rho_{0})$ such that $\eta_{X}$ and $\eta_{0}$ are compatible which is automatically true by the above discussion.

We equip the right-hand side of \eqref{thetaba} with the natural descent datum given by $\omega_{\CM_{r},E}\times\omega_{\CM_{0},E}$ and with an action of $\wt{J}(\BQ_{p})$ via the map $\wt{J}(\BQ_{p})\rightarrow J(\BQ_{p})\times J_{\T}(\BQ_{p})$.

\begin{proposition}\label{LS-weil-banal}
There is an isomorphism of $O_{\breve{E}}$-formal schemes
\begin{equation*}
\widetilde{\mathcal{M}}_{\mathbf{K}_{\wt G}} \overset{\sim}{\longrightarrow}
\wt G(\mathbb{Q}_p)/\mathbf{K}_{\wt G} ,
 \end{equation*}
which is compatible with natural actions of $\wt{J}(\BQ_{p})$ on both sides. For varying $\bK_{\wt{G}}$, the isomorphism is equivariant for the natural Hecke actions of $\wt{G}(\BQ_{p})$ on both sides.

The Weil descent datum $\omega_{\wt{\mathcal{M}}_{\mathbf{K}_{\wt G}}}$
on the left-hand side corresponds on the right-hand side to the Weil descent datum down to $O_E$ given by  multiplication by the special element $\wt{w}$ relative to $(\Phi^{+}_{0}, E)$ for our fixed uniformizer $\varpi_{E}$, viewed as an element in the center of $\wt{G}(\mathbb{Q}_p)$.
\end{proposition}
\begin{proof}
We have
\begin{equation*}
\widetilde{\mathcal{M}}_{\mathbf{K}_{\wt G}}(i)\overset{\sim}{\longrightarrow} \mathcal{M}_{r, \bK_{G}, O_{\breve{E}}}(i)\times{\mathcal{M}}_{0, O_{\breve{E}}}(i),\quad i\in\BZ .
\end{equation*} 
By Proposition \ref{Mr-banal} and Remark \ref{simi-factor-rmk}, we therefore obtain the desired isomorphism.
\begin{equation*}
\begin{aligned}
\widetilde{\mathcal{M}}_{\mathbf{K}_{\wt G}}&=\coprod_i\mathcal{M}_{r, \bK_{G}, O_{\breve{E}}}(i)\times {\mathcal{M}}_{0, O_{\breve{E}}}(i)\\
&\overset{\sim}{\longrightarrow} G(\BQ_{p})/\bK_{G}\times_{\BQ^{\times}_{p}/\BZ^{\times}_{p}} \T(\BQ_{p})/\bK^{\circ}_{\T}\\
&\overset{\sim}{\longrightarrow} \wt{G}(\BQ_{p})/\bK_{\wt{G}}.
\end{aligned}
\end{equation*}
For varying $\bK_{\wt{G}}$, the isomorphism is equivariant for the natural Hecke actions of $\wt{G}(\BQ_{p})$ on both sides by construction and Proposition \ref{Mr-banal}.

Let $\Spec R$ be a scheme over $\Spf O_{\breve{E}}$.  From the compatibility of Weil descent data of the closed immersion $\vartheta$, we obtain the commutativity of the following  diagram
\begin{equation*}
\begin{tikzcd}
\wt{\CM}_{\bK_{\wt G}}(i)(\bar{R}) \arrow[r, "\sim"] \arrow[d, "\omega_{\wt{\CM}_{\bK^{\circ}_{\wt G}}}"] & \mathcal{M}_{r, \bK_{G}, O_{\breve{E}}}(i)\times \CM_{0, O_{\breve{E}}}(i)(\bar{R})  \arrow[d, "\omega_{\CM_{r}, E}\times \omega_{\CM_{0}, E}"] \\
\wt{\CM}_{\bK_{\wt G}}(i+f_{E})(\bar{R}_{[\tau_{E}]}) \arrow[r, "\sim"] &  \mathcal{M}_{r, \bK_{G}, O_{\breve{E}}}(i+f_{E})\times\CM_{0, O_{\breve{E}}}(i+f_{E})(\bar{R}_{[\tau_{E}]}) .        
\end{tikzcd}
\end{equation*}
 Thus we have the commutativity of the following  diagram 
\begin{equation}
\begin{tikzcd}
\wt{\CM}_{\bK_{\wt G}}(\bar{R}) \arrow[r, "\sim"] \arrow[d, "\omega_{\wt{\CM}_{\bK^{\circ}_{\wt G}}}"] & G(\BQ_{p})/\bK_{G}\times_{\BQ^{\times}_{p}} \T(\BQ_{p})/\bK^{\circ}_{\T}  \arrow[d, "\omega_{\CM_{r}, E}\times \omega_{\CM_{0}, E}"] \\
\wt{\CM}_{\bK_{\wt G}}(\bar{R}_{[\tau_{E}]}) \arrow[r, "\sim"] &  G(\BQ_{p})/\bK_{G}\times_{\BQ^{\times}_{p}} \T(\BQ_{p})/\bK^{\circ}_{\T}.\\         
\end{tikzcd}
\end{equation}
But  by Proposition \ref{LS-tori-weil} (via Remark \ref{sp-elt-rmk}) and Proposition \ref{Mr-banal}, $\omega_{\CM_{r},E}\times \omega_{\CM_{0},E}$ is given by multiplication by $(\wt{w}, \wt{w})$ on $G(\BQ_{p})/\bK_{G}\times_{\BQ^{\times}_{p}} \T(\BQ_{p})/\bK^{\circ}_{\T}$. Under the isomorphism $\wt{G}(\BQ_{p})/\bK_{\wt{G}}=G(\BQ_{p})/\bK_{G}\times_{\BQ^{\times}_{p}} \T(\BQ_{p})/\bK^{\circ}_{\T}$, this is given simply by multiplication by $\wt{w}$.
\end{proof}

\section{$p$-adic uniformization of RSZ Shimura curves}\label{s:globRSZ}
In this section, we prove Theorem \ref{maintilde}. We fix the notation as in that theorem: we therefore have the CM-extension $K/F$, the fixed embedding $\varphi_0:K\to\bar\BQ$, the CM-type $r$ of rank $2$ special relative to the restriction $w_0$ of $\varphi_0$ to $F$, with its  reflex field $E_{r}$;  we also have a CM-type $r_{0}$ of rank $1$ with its  associated classical CM-type $\Phiplus_{0}$ for which $\varphi_0\in\Phi_0^+$ and the composite  $E=E_{\Phi^{+}_{0}}E_{r}=E_{0}E_{r}$. We further have the prime number  $p$, a fixed embedding  $\nu:\bar{\BQ}\rightarrow \bar{\BQ}_{p}$, the induced place $\nu$ of $E$ and the induced place $v_0$ of $F$ via the embedding ${\varphi_0}_{|_F}: F\to E$.  We denote by $\kappa_{\nu}:=\kappa_{E_{\nu}}$ the residue field of $E_{\nu}$ whose cardinality is $p^{f_{E_{\nu}}}$, where $f_{E_{\nu}}$ is the inertia index. We fix a uniformizer $\varpi_{\nu}:=\varpi_{E_{\nu}}$ of the ring of integers $O_{E_{\nu}}$ of $E_{\nu}$.

We also have the $K/F$-anti-hermitian vector space $(V, \varkappa)$ of signature $r$; it  is assumed that $v_0$ does not split in $K$ and that $V_{v_0}$ is non-split. We will refer to $v_0$ as the special prime, the other $p$-adic places $v$ of $F$ will be called the banal primes. Finally, we fix an almost selfdual lattice $\Lambda_v\subset V_{v}=V\otimes_F F_v$ for each $p$-adic place $v$ of $F$ in the sense of the Notation section at the end of the Introduction. 

We will only consider compact open subgroups $\bK_{\wt G}\subset \wt G(\BA_f)$ which, in terms of the product decomposition \eqref{tildeG}, are of the form
\begin{equation}\label{decprod1}
\bK_{\wt G}=\bK_\U \times \bK^{\circ}_{Z^\BQ},
\end{equation} 
where $\bK_\U\subset \U(\BA_f)$ is a open compact subgroup of the form  
\begin{equation*}
{\bK}_\U = {\bK}_{\U, p} \cdot {\bK}^{p}_{\U},
\end{equation*}
where $\bK_{\U}^{p}$ is an open compact subgroup of $\U(\BA^p_{f})$ and $\bK_{\U, p}$ is an open compact subgroup of $\U(\BQ_{p})$.  Then we assume that $\bK_{\U, p}$ further decomposes as $\bK_{\U, p}=\bK_{\U, v_{0}}\cdot\bK^{v_{0}}_{\U, p}$, where ${\bK}_{\U, p}^{v_{0}}$ is an  open compact subgroup of  $\U^{v_{0}}(\BQ_{p})=\prod_{v\neq v_{0}}\U_{v}(\BQ_{p})$. We assume that  ${\bK}_{\U, v_{0}}$ is a maximal open compact and therefore ${\bK}_{\U, v_{0}}= \U_{v_{0}}(\BQ_{p})$, as follows from our assumption that $ \U_{v_{0}}(\BQ_{p})$ is a  compact group. Let ${\bK}^{v_{0}}_\U = {\bK}^{v_{0}}_{\U, p} \cdot {\bK}^{p}_{\U}$ and  ${\bK}^{v_{0}}_{\wt{G}} = {\bK}^{v_{0}}_{\U} \times \bK^{\circ}_{Z^{\BQ}}$ and $\bK^{v_{0}}_{\wt G, p}={\bK}_{\U, p}^{v_{0}}\times \bK^{\circ}_{Z^\BQ, p}$.

Let $(X, \iota, \lambda)$ be a semi-local CM-triple, cf. \S \ref{ss:semiglob}. We say $(X, \iota, \lambda)=\prod_{v\mid p}(X_{v}, \iota_{v}, \lambda_{v})$ is $(V, \varsigma)$-principally polarized if each local CM-triple $(X_{v}, \iota_{v}, \lambda_{v})$ is $(V_{v}, \varsigma_{v})$-principally polarized. Sometimes, we also simply say that $\lambda$ is a $(V, \varsigma)$-principal polarization.

\subsection{Semi-local formal moduli space of RSZ CM-triples}\label{ss:semloc}

We will need a semi-local version of the notion of a special element from Construction \ref{sp-elt-loc}. 

\begin{construction}\label{consemil}
Let $\Phi^+$ be a classical CM-type. Let $E\subset \bar\BQ$ be a finite extension of $E_{\Phi^+}$ and $\nu: \bar{\BQ}\rightarrow \bar{\BQ}_{p}$ be a fixed embedding.  Then $\Phi^+$ defines a homomorphism $\mu:\BG_{m, E}\to Z^\BQ_{E}$ and hence by base change 
\[\mu:\BG_{m, E_\nu}\to Z^\BQ_{E_\nu}.
\]
We obtain the semi-local reflex norm map which is the composition 
\[
\frak{r}_{\Phi^+, E_\nu}: \Res_{E_\nu/\BQ_p}(\BG_{m, E_\nu})\overset{\mu}\longrightarrow  \Res_{E_\nu/\BQ_p}(Z^\BQ_{ E_\nu})\overset{\Nm_{E_\nu/\BQ_p}}\longrightarrow Z^\BQ_{\BQ_p} .
\]
Then the special element relative to $(\Phi^{+}, E_{\nu})$ for the uniformizer $\varpi_{\nu}$ is given by
\begin{equation}
\wt w=\frak{r}_{\Phi^+, E_\nu}(\varpi_\nu)\in Z^\BQ(\BQ_p).
\end{equation}
 Its image in $Z^\BQ(\BQ_p)/Z^\BQ(\BZ_p)$ is independent of $\varpi_\nu$. 
 
 Using the embedding $\nu: \bar\BQ\to \bar\BQ_p$, we have the decomposition  $\Phi^+=\coprod\nolimits_{v\mid p} \Phi_v^+$, cf. \eqref{Phiplusv}. The local reflex field $E(v)$ for $\Phi_v^+$ is contained in $E_\nu$. 
Under the map  
\begin{equation}\label{loc-sp}
 Z^{\BQ}(\BQ_{p})\to\prod\nolimits_{v\mid p} \big(Z^{\BQ}_{v}(\BQ_{p})/Z^{\BQ}_{v}(\BZ_{p})\big) ,
\end{equation}
the image of $\wt{w}$ coincides with   $({\wt w}_{v})$, where $\wt{w}_v$ is the special element in  $Z^{\BQ}_{v}(\BQ_{p})$ relative to $(\Phi^+_v, E_\nu)$ for the uniformizer $\varpi_\nu$. 
\end{construction}

Consider the integral Shimura variety $\wt\CA_{\bK_{\wt G}}$ in Definition \ref{def:overOE} and consider it as a scheme over $O_{E_{\nu}}$.  We fix a point 
\begin{equation}\label{fixbas}
(\bm{A}_0, \bm{\iota}_0, \bm{\lambda}_0, \bm{A}, \bm{\iota}, \bm{\lambda}, \bar{\bm{\eta}}^p, \bar{\bm{\eta}}_{p}^{v_{0}}) \in \wt \CA_{\bK_{\wt G}}(\bar{\kappa}_{\nu}).
\end{equation}

We denote by $\BX_0$, resp. $\BX$, the $p$-divisible group of $\bm{A}_0$, resp. $\bm{A}$. These are equipped with an action $\iota_{\BX_0}$, resp. $\iota_\BX$, of $O_{K}\otimes \BZ_{p}$ induced by $\bm{\iota}_{0}$, resp. $\bm{\iota}$. They are also equipped with a polarization $\lambda_{\BX_0}$, resp. $\lambda_\BX$ induced by $\bm{\lambda}_0$ and $\bm{\lambda}$. Here, because the ideal $\frak a$ from \S \ref{ss:modoverE} is prime to $p$,  the  polarization $\lambda_{\BX_0}$ is  principal and $\lambda_{\BX}$ is a $(V, \varsigma)$-principal polarization. 
 
In analogy with \eqref{def:Js}, we set 
 \begin{equation}\label{Jsemi}
\begin{aligned}
J_{Z^{\BQ}}(\BQ_{p})&=\{\alpha\in \mathrm{Aut}^{\circ}_{K\otimes \BQ_{p}}(\BX_{0}): \alpha^{\ast}(\lambda_{\BX_{0}})=\mu_{0}(\alpha){\lambda}_{\BX_{0}}, \; \mu_{0}(\alpha)\in \BQ^{\times}_{p}\};\\
J_{\U}(\BQ_{p})&=\{\alpha\in \mathrm{Aut}^{\circ}_{K\otimes\BQ_{p}}(\BX): \alpha^{\ast}(\lambda_{\BX})={\lambda}_{\BX}\};\\
J(\BQ_{p})&=\{\alpha\in \mathrm{Aut}^{\circ}_{K\otimes\BQ_{p}}(\BX): \alpha^{\ast}(\lambda_{\BX})=\mu(\alpha){\lambda}_{\BX},\;\mu(\alpha)\in \BQ^{\times}_{p}\};\\
\wt{J}(\BQ_{p})&=J(\BQ_{p})\times_{\BQ^{\times}_{p}} J_{\T}(\BQ_{p})\simeq J_{\U}(\BQ_{p})\times J_{\T}(\BQ_{p}).
\end{aligned}
\end{equation}
Then $J_{\U}(\BQ_{p})=\prod_{v\mid p}J_{\U, v}(\BQ_{p})$, where $J_{\U, v}(\BQ_{p})=\{\alpha_{v}\in \mathrm{Aut}^{\circ}_{K_{v}}(\BX_{v}): \alpha^{\ast}_{v}(\lambda_{\BX_{v}})={\lambda}_{\BX_{v}}\}$.

For $v\neq v_{0}$, $J_{\U, v}(\BQ_{p})$ is isomorphic to $\U(V_{v})(F_{v})$ and  $J_{\U, v_{0}}(\BQ_{p})$ is isomorphic to $\U(\bar{V}_{v_{0}})(F_{v_{0}})$, cf. \cite[(7.2.6), (7.2.8)]{KRZI}. We define 
 \begin{equation}
 \begin{aligned}
 &J^{v_{0}}_{\U}(\BQ_{p})=\prod\nolimits_{v\neq v_{0}, v\mid p}J_{\U, v}(\BQ_{p}), \quad &\wt{J}^{v_{0}}(\BQ_{p})=J^{v_{0}}_{\U}(\BQ_{p})\times J_{Z^{\BQ}}(\BQ_{p}),\\
 &\U^{v_{0}}(\BQ_{p})=\prod\nolimits_{v\neq v_{0}, v\mid p}\U_{v}(\BQ_{p}), \quad 
& \wt{G}^{v_{0}}(\BQ_{p})=\U^{v_{0}}(\BQ_{p})\times Z^{\BQ}(\BQ_{p}).
 \end{aligned}
 \end{equation}
 Then we have an isomorphism $\wt{J}^{v_{0}}(\BQ_{p})\simeq\wt{G}^{v_{0}}(\BQ_{p})$. For each $v\mid p$, we set 
 \begin{equation}
 \wt{G}_{v}(\BQ_{p})=\U_{v}(\BQ_{p})\times Z^{\BQ}_{v}(\BQ_{p}), \quad \bK_{\wt G, v}=\bK_{\U, v}\times Z^{\BQ}_{v}(\BZ_{p}).
 \end{equation}
 \begin{definition}\label{RSZ-RZ}
We define for each $i\in \BZ$ a functor $\wt{\CM}_{\bK_{\wt G, p}}(i)$ on $\Nilp_{O_{\breve{E}_\nu}}$. Let $S$ be a scheme over $\Spf O_{\breve{E}_\nu}$. A point of $\wt\CM_{\bK_{\wt G, p}}(i)(S)$ is given by the following data:
\begin{enumerate}
\item[(1)] A CM-triple $(X_0, \iota_0, \lambda_0)$ of type $(K\otimes \BQ_{p}/F\otimes \BQ_{p}, r_{0})$ satisfying the condition $({\rm KC}_{r_{0}})$;
      
\item[(2)] An $O_K\otimes \mathbb{Z}_p$-linear quasi-isogeny
\begin{equation*}    
\rho_0: \bar{X}_0 := X_0 \times_{S} \bar{S} \longrightarrow \mathbb{X}_0 \times_{\Spec k} \bar{S}
\end{equation*}
such that locally on $\bar{S}$, there exists $u\in\BZ^{\times}_{p}$ such that $\rho^{\ast}_{0, v}({\lambda}_{\BX_{0, v}})=up^{i}\lambda_{0,v}$ for each $v\mid p$; 
\item[(3)] A CM-triple $(X, \iota, \lambda)$ of type
$(K \otimes \mathbb{Q}_p / F \otimes \mathbb{Q}_p, r)$ over $S$
which satisfies the conditions $({\rm KC}_r)$ and $({\rm EC}_r)$ and is $(V, \varsigma)$-principally polarized.

\item[(4)] An $O_K\otimes \mathbb{Z}_p$-linear quasi-isogeny 
\begin{equation*}
\rho: \bar{X} := X \times_{S} \bar{S} \longrightarrow \mathbb{X}\times_{\Spec \bar{\kappa}} \bar{S} 
\end{equation*}
such that, locally on $\bar{S}$, $\rho^{\ast}_{v}({\lambda}_{\BX_{v}})=up^{i}\lambda_{\bar{X}_{v}}$  for each $v\mid p$, with the same $u\in\BZ^{\times}_{p}$ as in datum $(2)$;
    
\item[(5)]  A class $\bar{\eta}_p^{v_{0}}$ of isomorphisms of \'etale sheaves
\begin{equation*}
\eta_p^{v_{0}}: \underline{\Hom}_{O_{K^{v_{0}}_{p}}}(X_0^{v_{0}}, X^{v_{0}})\isoarrow \Lambda\otimes_{O_K}O_{K^{v_{0}}_{p}} \;
\mathrm{mod}\; \bK^{v_{0}}_{\U, p}
\end{equation*}
 which respect the bilinear forms on $\underline{\Hom}_{O_{K_{v}}}(X_{0, v}, X_{v})$ and $\Lambda\otimes_{O_K}O_{K_{v}}$ for each $p$-adic place $v\neq v_{0}$.  
\end{enumerate} 
\end{definition}
We denote these data by a tuple $(X_{0}, \iota_{0}, \lambda_{0},\rho_{0}, X, \iota, \lambda_{X}, \rho)$. Two tuples $(X_{0}, \iota_{0}, \lambda_{0}, \rho_{0}, X, \iota, \lambda_{X}, \rho)$ and $(X^{\prime}_{0}, \iota^{\prime}_{0}, \lambda^{\prime}_{0}, \rho^{\prime}_{0}, X^{\prime}, \iota^{\prime}, \lambda_{X^{\prime}}, \rho^{\prime})$ define the same point of $\wt{\CM}_{\bK_{\wt G, p}}(i)$ iff there exist isomorphisms $\alpha_{0}: (X_{0}, \iota_{0})\isoarrow (X^{\prime}_{0}, \iota^{\prime}_{0})$ and $\alpha: (X, \iota)\isoarrow (X^{\prime}, \iota^{\prime})$ of $O_{K}$-modules such that $\rho^{\prime}_{0}\circ\alpha_{0, \bar{S}}=\rho_{0}$ and $\rho^{\prime}\circ\alpha_{\bar{S}}=\rho$ and such that $(\alpha, \alpha_0)$ pulls back $\bar{\eta}_p^{\prime v_{0}}$ to $\bar{\eta}_p^{v_{0}}$.

\begin{remark}\label{simi-fact-rmk-g}
We could replace the condition in $(2)$ by $\rho_{0}^{\ast}(\lambda_{\BX_{0}})=p^{i}{\lambda}_{\bar{X}_{0}}$ and the condition in $(4)$ by $\rho^{\ast}(\lambda_{\BX})=p^{i}\lambda_{\bar{X}}$ without changing the formal moduli space. Similarly, if $p^i$ is replaced by $up^i$ with $u\in O_F^\times$ in both equations, the formal moduli space will not change. This follows from the same reasoning as in \cite[Remark 7.2.3]{KRZI}.
\end{remark}

The formal scheme over $\Spf O_{\breve{E}}$ representing this functor is also denoted by $\wt{\CM}_{\bK_{\wt G, p}}(i)$. We define the formal scheme
\begin{equation}
\wt{\CM}_{\bK_{\wt G, p}}=\bigsqcup\nolimits_{i\in\BZ}\wt{\CM}_{\bK_{\wt G, p}}(i). 
\end{equation}
Thus $\wt{\CM}_{\bK_{\wt G, p}}$ parametrizes $(X_{0}, \iota_{0}, \lambda_{0},\rho_{0}, X, \iota, \lambda_{X}, \rho)$ such that $\rho_{0}$ respects the polarization $\lambda_{0}$ on $\bar{X}_{0}$ and ${\lambda}_{\mathbb{X}_{0}}$ on $\BX_{0}$ up to a factor in $\mathbb{Q}_p^{\times}$ and such that $\rho$ respects the polarization $\lambda_{X}$ on $\bar{X}$ and ${\lambda}_{\mathbb{X}}$ on $\BX$ up to the same factor in $\mathbb{Q}_p^{\times}$. Note that $\wt{\CM}_{\bK_{\wt G, p}}$ affords an action of $\wt{J}(\BQ_{p})=J(\BQ_{p})\times_{\BQ^{\times}_{p}} Z^{\BQ}(\BQ_{p})$. We refer to $\wt{\CM}_{\bK_{\wt G, p}}$ as the semi-local formal moduli space of RSZ CM-triples. Following the same procedure as in \eqref{WD7e}, we define a Weil descent datum 
\begin{equation}\label{WD7esemi}
\omega_{\wt{\CM}_{\bK_{\wt G, p}}}: \wt{\CM}_{\bK_{\wt G, p}} \longrightarrow
\wt{\CM}^{(\tau_{E_{\nu}})}_{\bK_{\wt G, p}}
\end{equation}
where the upper index $(\tau_{E_{\nu}})$ denotes the base change via
$\Spec \tau_{E_{\nu}} : \Spf O_{\breve{E}_{\nu}} \longrightarrow \Spf O_{\breve{E}_{\nu}}$. 
  
We fix an isomorphism $J_{\U, v_{0}, \mathrm{ad}}(\BQ_{p})$ with $\PGL_{2}(\BQ_{p})$. Via this identification, we then obtain an action of $\wt{J}(\BQ_{p})=J_{\U}(\BQ_{p})\times Z^{\BQ}(\BQ_{p})$ on $\hat{\Omega}_{F_{v_{0}}}$ via $\wt J(\BQ_p)\to J_\U(\BQ_p)\to  J_{\U, v_{0}, \mathrm{ad}}(\BQ_{p})$. Note that $\wt G(\mathbb{Q}_p)/\mathbf{K}_{\wt G,p}\simeq \U^{v_{0}}(\BQ_{p})/\bK^{v_{0}}_{\U}\times Z^{\BQ}(\BQ_{p})/\bK^{\circ}_{Z^{\BQ}, p}$ since $\U_{v_{0}}(\BQ_{p})=\bK_{\U, v_{0}}$. Then we obtain an action of $\wt{J}(\BQ_{p})$ on $\wt G(\mathbb{Q}_p)/\mathbf{K}_{\wt G,p}$ via the factor $\wt{J}^{v_{0}}(\BQ_{p})\simeq\wt{G}^{v_{0}}(\BQ_{p})$.
  
\begin{proposition}\label{RZ-p-weil}
There exists an isomorphism of formal schemes over $O_{\breve{E}_\nu}$
\begin{equation*}
\widetilde{\mathcal{M}}_{\mathbf{K}_{\wt G, p}} \overset{\sim}{\longrightarrow}
(\hat{\Omega}_{F_{v_{0}}} \times_{\Spf O_{F_{v_{0}}}} \Spf O_{\breve{E}_{\nu}}) \times
\wt G(\mathbb{Q}_p)/\mathbf{K}_{\wt G,p}
 \end{equation*}
which is equivariant with respect to the action of $\wt J(\mathbb{Q}_p)$ on both sides.  For varying $\bK^{v_{0}}_{\wt G, p}$, this isomorphism is compatible with Hecke action of $\wt{G}^{v_{0}}(\BQ_{p})$ on both sides.

The Weil descent datum $\omega_{\wt{\mathcal{M}}_{\mathbf{K}_{\wt G, p}}}$ down to $O_{E_\nu}$ on the left-hand side corresponds on the right-hand side to the Weil descent datum given by
\begin{equation*}
(\xi, g) \mapsto (\omega_{\Omega, {E_{\nu}}}(\xi), \wt{w}g), \quad g \in \wt{G}(\mathbb{Q}_p) .
\end{equation*} 
Here $\wt{w}$ is the special element relative to $(\Phi^{+}_{0}, E_{\nu})$ for the uniformizer $\varpi_{\nu}$, as in Construction \ref{consemil},  considered as an element in the center of $\wt{G}(\mathbb{Q}_p)$.
\end{proposition}

\begin{proof} 
Let $(X_{0}, \iota_{0}, \lambda_{0},\rho_{0}, X, \iota, \lambda_{X}, \rho)$ be a tuple parametrized by $\wt{\CM}_{\bK_{\wt G, p}}$. Then the decomposition 
\begin{equation*}
X_{0}=\prod\nolimits_{v\mid p}X_{0, v}, \quad X=\prod\nolimits_{v\mid p}X_{v}
\end{equation*}
extends to the tuple 
\begin{equation*}
(X_{0}, \iota_{0}, \lambda_{0},\rho_{0}, X, \iota, \lambda_{X}, \rho)=\prod\nolimits_{v\mid p}(X_{0, v}, \iota_{0, v}, \lambda_{0, v},\rho_{0, v}, X_{v}, \iota_{v}, \lambda_{X_{v}}, \rho_{v}, \bar\eta_{v}).
\end{equation*}
This in turn gives rise to a map
\begin{equation}\label{se-loc}
\wt{\CM}_{\bK_{\wt G, p}}\longrightarrow \prod\nolimits_{v\mid p} \wt{\CM}_{\bK_{\wt G, v}}\times_{\Spf O_{\breve{E}(v)}} \Spf O_{\breve{E}_\nu}.
\end{equation}
Here $\breve E(v)$ is the completion of the maximal unramified extension of the composite $E(v)$ of the local reflex fields for $r_v$ and $r_{0, v}$ (recall that $\nu$ defines the decomposition $\Phi=\bigsqcup_{v}\Phi_v$ and the restrictions $r_v$ and $r_{0, v}$ of $r$ and $r_0$). Then $\breve E(v)\subset \breve E_\nu$. Furthermore,  $\wt{\CM}_{\bK_{\wt G, v_{0}}}$ is given by the formal scheme $\wt{\CM}_{\bK^{\circ}_{\wt{G}}}$ defined in \eqref{wtCM-special};  for $v\neq v_{0}$,  $\wt{\CM}_{\bK_{\wt G, v}}$ is given by the formal scheme $\wt{\CM}_{\bK_{\wt{G}}}$ defined in Definition \eqref{wtCM-banal}.

 The map \eqref{se-loc} is an isomorphism. Indeed, comparing the moduli descriptions, we see that the RHS is characterized by the fact that the compatibility scalars in 2) and 4)  are the same for any $v\mid p$, whereas the LHS is characterized by the fact that the compatibility scalars in 2) and 4)  are the same for \emph{all} $v\mid p$. But appealing, as in Remark \ref{simi-factor-rmk} to \cite[Remark 7.2.3]{KRZI}, we see that in passing to isomorphism classes, the difference between these conditions is irrelevant.  The Hecke equivariance of \eqref{se-loc} is obvious. 
Now the claimed  isomorphism over $O_{\breve E_{\nu}}$ follows from Proposition \ref{LS-weil-special}  and Proposition \ref{LS-weil-banal}.

The Weil descent datum on $ \wt{\CM}_{\bK_{\wt G, v_0}}$ from $O_{\breve E({v_0})}$ down to $O_{E({v_0})}$ is given under   the isomorphism  with $(\hat{\Omega}_{F_{v_0}} \times_{\Spf O_{F_{v_0}}} \Spf O_{\breve E({v_0})}) \times \wt{G}_{v_{0}}(\mathbb{Q}_p)/\bK_{\wt{G}, v_{0}}$ of Proposition \ref{LS-weil-special} by
\begin{equation*}
(\xi, g) \mapsto (\omega_{\Omega, {E({v_0})}}(\xi), \wt{w}_{v_0}g), \quad g \in \wt{G}_{v_{0}}(\mathbb{Q}_p)/\bK_{\wt{G}, v_{0}},
\end{equation*} 
where $\wt{w}_{v_0}$ is the special element relative to $(\Phi^{+}_{0, v_0}, E({v_{0})})$ for the uniformizer $\varpi_{E({v_0})}$, as in Construction \ref{sp-elt-loc}. A similar formula follows from  Proposition \ref{LS-weil-banal} for $v\neq v_0$. The result follows now from the compatibility \eqref{loc-sp} of $\wt w$ with localization and Remark \ref{sp-elt-rmk} on the behavior of special elements under passage to bigger fields (applied to $E_\nu\supset E(v)$).

\end{proof}

\subsection{$p$-adic uniformization} 
Recall from \eqref{fixbas} the base point $(\bm{A}_0, \bm{\iota}_0, \bm{\lambda}_0, \bm{A}, \bm{\iota}, \bm{\lambda}, \bar{\bm{\eta}}^p, \bar{\bm{\eta}}_{p}^{v_{0}}) \in \wt \CA_{\bK_{\wt G}}(\bar{\kappa}_{\nu})$. We define the algebraic groups over $\BQ$
\begin{equation}\label{Jglob}
\begin{aligned}
J_{Z^{\BQ}}(\BQ)&=\{\alpha\in \mathrm{Aut}^{\circ}_{K}(\mathbf{A}_{0}): \alpha^{\ast}(\bm{\lambda}_{0})=\mu_{0}(\alpha) \bm{\lambda}_{0}, \; \mu_{0}(\alpha)\in \BQ^{\times}\};\\
J_{\U}(\BQ)&=\{\alpha\in \mathrm{Aut}^{\circ}_{K}(\mathbf{A}): \alpha^{\ast}(\bm{\lambda})=\bm{\lambda}\};\\
J(\BQ)&=\{\alpha\in \mathrm{Aut}^{\circ}_{K}(\mathbf{A}): \alpha^{\ast}(\bm{\lambda})=\mu(\alpha)\bm{\lambda},\;\mu(\alpha)\in \BQ^{\times}\};\\
\wt{J}(\BQ)&=J(\BQ)\times_{\BQ^{\times}} J_{Z^{\BQ}}(\BQ)=J_{\U}(\BQ_{p})\times J_{Z^{\BQ}}(\BQ).
\end{aligned}
\end{equation}
The Dieudonn\'e modules of $\mathbf{A}_{0}$ and $\mathbf{A}$ are isoclinic (for $\mathbf{A}_{0}$ this is obvious, for $\mathbf{A}$, this follows from \cite[Proposition 7.1.11]{KRZI}). This implies that the $\BQ_p$-valued points $J_{Z^{\BQ}}(\BQ_p)$, $J_\U(\BQ_p)$, $J(\BQ_p)$   and $\wt J(\BQ_p)$ coincide with the groups \eqref{Jsemi}, cf. \cite[Cor. 6.29]{RZ}.

Now we can prove the main $p$-adic uniformization theorem for the RSZ Shimura curve introduced in Theorem \ref{maintilde}.

\begin{theorem}\label{main}
Let $(\wt\CA_{\bK_{\wt G}})^\wedge$ be the formal completion of $\wt\CA_{\bK_{\wt G}}$ along its special fiber, which is a formal scheme over $\Spf O_{E_\nu}$. Then there exists an isomorphism of formal schemes over $\Spf O_{\breve {E}_\nu}$,
\begin{equation*}
(\wt\CA_{\bK_{\wt G}})^\wedge\times_{\Spf O_{E_{\nu}}} \Spf O_{\breve {E}_\nu}  \simeq   \wt{J}(\BQ)\bs\big[\big(\widehat{\Omega}_{F_{v_{0}}} \times_{\Spf O_{F_{v_{0}}}} \Spf O_{\breve {E}_\nu}\big) \times
 \wt G(\BA_f)/{\bK}_{\wt G}\big].
\end{equation*}
The quotient on the right-hand-side is defined via the composition $$\wt J(\BQ)\to J_\U(\BQ)\to J_{\U, v_0}(\BQ_p)\to J_{\U, v_0, {\rm ad}}(\BQ_p)$$ using an identification of the adjoint group ${J}_{\U, v_{0}, {\rm ad}}(\BQ_{p})$ with $\PGL_2(F_{v_{0}})$ and through an action of 
$\wt{J}(\BQ)$ on $\wt G(\BA_f)/{\bK}_{\wt G}$. For varying $\bK^{v_{0}}_{\wt{G}}$, the isomorphism is equivariant for the action of the group $\wt{G}^{v_{0}}(\BA_{f})$ which acts on both sides by Hecke correspondences.

Furthermore, the natural descent datum down to $O_{E_{\nu}}$ on the left-hand side is given on the right-hand side by 
\begin{equation*}
(\xi, h)\rightarrow (\omega_{\omega_{\Omega, E_{\nu}}}(\xi), \wt{w}h), \quad h\in \wt G(\BA_f) , 
\end{equation*}
where $\wt{w}\in \wt G(\BQ_p)$ is the special element relative to $(\Phi^{+}_{0}, E_{\nu})$ for the uniformizer $\varpi_{\nu}$, as in Construction \ref{consemil},  considered as an element in the center of $\wt G(\BA_f)$.
\end{theorem}
\begin{proof}
We define as follows the uniformization morphism 
\begin{equation}\label{theta}
  \Theta:  \wt\CM_{\bK_{\wt G, p}} \times \wt{G}(\mathbb{A}_f^{p})/{\bK}^{p}_{\wt{G}} \longrightarrow
 (\wt\CA_{\bK_{\wt G}})^{\wedge} \times_{\Spf O_{E_{\nu}}} \Spf O_{\breve{E}_{\nu}} .
\end{equation}
Let ${S}=\Spec {R}$,   and let $(X_{0}, \iota_{0}, \lambda_{0}, \rho_{0}, X, \iota, \lambda, \rho, \bar{\eta}^{v_{0}}_{p})$ be a point of  $\wt\CM_{\bK_{\wt G, p}}(i)$ over ${R}$ and let
$g \in \wt{G}(\mathbb{A}_f^{p})$. The quasi-isogeny $\rho$ over  $\bar R=R\otimes_{O_{E_\nu}}\bar{\kappa}_{\nu}$
extends uniquely to a quasi-isogeny of abelian varieties
\begin{equation*}\label{rho-A}
\rho: \bar{A} \longrightarrow \bm{A} \times_{\Spec \bar{\kappa}_{{\nu}}} \Spec \bar{R}.
\end{equation*}
Because $O_K$ acts on $\bar{X}=X\otimes_R \bar R$, we obtain a map
$O_K\to \End(\bar A)\otimes\BZ_{(p)}$. Moreover, the polarization
$\bm{\lambda}: \bm{A} \longrightarrow \bm{A}^{\vee}$
induces on $\bar{A}$ a quasi-polarization
$\lambda'_{\bar{A}}: \bar{A} \longrightarrow \bar{A}^{\vee}$.
On the $p$-divisible groups, $\lambda'_{\bar{A}}$ is given by $p^i\lambda$  and therefore
$\lambda_{\bar{A}} := p^{-i} \lambda'_{\bar{A}}$ satisfies the condition  (ii)
in the Definition \ref{def:overOE} of the functor $\wt{\mathcal{A}}_{\bK_{\wt{G}}}$. Similarly, we obtain  $(\bar{A}_{0}, \iota_{\bar{A}_{0}}, \lambda_{\bar{A}_{0}})$ and the quasi-isogeny
\begin{equation*}\label{rho0}
\rho_{0}: \bar{A}_{0} \longrightarrow \bm{A}_{0} \times_{\Spec \bar{\kappa}_{{\nu}}} \Spec \bar{R}.
\end{equation*}
The quasi-isogenies $\rho_{0}$ and $\rho$ induce an isomorphism $\widehat{\RV}^{p}(\bar{A}_{0}, \bar{A})\xrightarrow{\sim} \widehat{\RV}^{p}({\bm A}_{0}, {\bm A})$ given by   $x\mapsto   \widehat{\RV}^p(\rho)^{-1}\circ x\circ\widehat{\RV}^p(\rho_0)$. Hence  $\bm{\bar{\eta}}^{p}$ induces an isomorphism
\begin{equation*}\label{prime-p-level}
{\eta}^p: \widehat{\RV}^{p}(\bar{A}_{0}, \bar{A})\longrightarrow \wt{V} \otimes \mathbb{A}_{K,f}^p \; \mod {\bK}^p_{\U} .
\end{equation*}

The same construction furnishes the isomorphism
\begin{equation*}\label{prime-to-v0-level}
{\eta}^{v_{0}}_{p}: \widehat{\RV}^{v_{0}}_{p}(\bar{A}_{0}, \bar{A})\longrightarrow \wt{V} \otimes_{K} K^{v_{0}}_{p} \; \mod {\bK}^{v_{0}}_{\U, p}. 
\end{equation*}

We associate to the tuple $(X_{0}, \iota_{0}, \lambda_{0}, \rho_{0}, X, \iota, \lambda, \rho, \bar{\eta}^{v}_{p}, g)$ from the left-hand side of (\ref{theta})
the point
\begin{equation}\label{image-of-Theta1}
  (\bar{A}_{0}, \iota_{\bar{A}_{0}}, \lambda_{\bar{A}_{0}}, \bar{A}, \iota_{\bar{A}}, \lambda_{\bar{A}}, \bar{\eta}^{p}g, \bar{\eta}^{v_{0}}_{p})
  \in \wt{\mathcal{A}}_{\bK_{\wt{G}}}(\bar{R}). 
\end{equation}
The semi-local CM-triples $(X_{0}, \iota_{0}, \lambda_{0})$ and $(X, \iota, \lambda)$ over $R$ define by the Serre-Tate theorem
a lifting of (\ref{image-of-Theta1}) to a point of $\wt{\mathcal{A}}_{\bK_{\wt{G}}}(R)$. The map is functorial in $R$ and defines
the uniformization morphism $\Theta$ in
(\ref{theta}). 

Recall $\wt J(\BQ)$ from \eqref{Jglob}.  We define an action of $\wt J(\BQ)$ on the left hand side of \eqref{theta}. 
Let $\wt{\alpha}=({\alpha}_0, \alpha)\in \wt{J}(\BQ)$, then we define $\omega(\wt{\alpha})\in \wt{G}(\BA^{p}_{f})$ by the equation
\begin{equation}\label{omega-map-defn}
\bm{\eta}^{p}\circ{\widehat\RV}^{p}(\alpha^{-1})\circ {\widehat\RV}^{p}(\alpha_{0})=\omega(\wt{\alpha})\circ \bm{\eta}^{p} .
\end{equation}
This defines a homomorphism
\begin{equation*}
\omega: \wt{J}(\BQ)\rightarrow \wt{G}(\BA^{p}_{f}).
\end{equation*}
We already noted that $\wt{J}(\BQ_{p})$ coincides with \eqref{Jsemi}, and hence acts on  $\wt{\CM}_{\bK_{\wt G, p}}$ by
\begin{equation*}
\wt\alpha=(\alpha_0, \alpha): (X_{0}, \iota_{0}, \lambda_{0},\rho_{0}, X, \iota, \lambda_{X}, \rho, \bar{\eta})\mapsto(X_{0}, \iota_{0}, \lambda_{0},\alpha_{0}\circ\rho_{0}, X, \iota, \lambda_{X}, \alpha\circ\rho, \bar{\eta}).
\end{equation*}
Let $((X_{0}, \iota_{0}, \lambda_{0},\rho_{0}, X, \iota, \lambda_{X}, \rho, \bar{\eta}), g)$ be a point of $ \wt\CM_{\bK_{\wt G, p}} \times \wt{G}(\mathbb{A}_f^{p})$ over $\bar{S}$ and let
\begin{equation*}
(\bar{A}_{0}, \iota_{\bar{A}_{0}}, \lambda_{\bar{A}_{0}}, \bar{A}, \iota_{\bar{A}}, \lambda_{\bar{A}}, g\bar{\eta}^{p}, \bar{\eta}^{v_{0}}_{p})
\in \wt{\mathcal{A}}_{\bK_{\wt{G}}}(\bar{R}) 
\end{equation*}  
be its image under $\Theta$. For $\wt{\alpha}=(\alpha_{0}, \alpha)\in \wt{J}(\BQ_{p})$,  the quasi-isogenies $\alpha_{0}\circ\rho_{0}$ and $\alpha\circ\rho$ extend to quasi-isogenies of abelian varieties
\begin{equation*}
\bar{A}_{0}\overset{\rho_{0}} \longrightarrow \bm{A}_{0,\bar{R}}\overset{\alpha_{0}} \longrightarrow \bm{A}_{0,\bar{R}}\quad\text{and}\quad\bar{A}\overset{\rho} \longrightarrow \bm{A}_{\bar{R}}\overset{\alpha} \longrightarrow \bm{A}_{\bar{R}}
\end{equation*}
The image of $((X_{0}, \iota_{0}, \lambda_{0},\alpha_{0}\circ\rho_{0}, X, \iota, \lambda_{X}, \alpha\circ\rho, \bar{\eta}), g)$ under the map $\Theta$ is given by
\begin{equation*}
(\bar{A}_{0}, \iota_{\bar{A}_{0}}, \lambda_{\bar{A}_{0}}, \bar{A}, \iota_{\bar{A}}, \lambda_{\bar{A}}, \omega({\wt\alpha}^{-1})g\bar{\eta}^{p}, \bar{\eta}^{v_{0}}_{p})
\in \wt{\mathcal{A}}_{\bK_{\wt{G}}}(\bar{R}).
\end{equation*}  
Hence, if we define the action of $\wt{J}(\BQ)$ on the source of  $\Theta$ by
\begin{equation*}
((X_{0}, \iota_{0}, \lambda_{0}, \rho_{0}, X, \iota, \lambda_{X}, \rho, \bar{\eta}), g)\mapsto ((X_{0}, \iota_{0}, \lambda_{0},\alpha_{0}\circ\rho_{0}, X, \iota, \lambda_{X}, \alpha\circ\rho, \bar{\eta}), \omega(\wt\alpha)g),
\end{equation*}
then $\Theta$ is equivariant. The uniformization theorem   for PEL-type Shimura varieties \cite[Theorem  6.2.3]{RZ} shows now that  $\Theta$ induces  an isomorphism
\begin{equation*}
\wt{J}(\BQ)\bs\big[\widetilde{\mathcal{M}}_{\mathbf{K}_{\wt G, p}} \times\wt G(\BA^{p}_f)/{\bK}^{p}_{\wt G}\big]\overset{\sim}\longrightarrow (\wt\CA_{\bK_{\wt G}})^\wedge\times_{{\rm Spf}\,O_{E_{\nu}}}{\rm Spf}\, O_{\breve {E}_\nu} .
\end{equation*}
Note here we are using the fact that $\wt\CA_{\bK_{\wt G}}(\bar{\kappa}_{\nu})$ has only one isogeny class, as follows from \cite[Proposition  7.11]{KRZI}.  By Proposition \ref{RZ-p-weil}, we can rewrite the left hand side of this isomorphism as 
\begin{equation*}
\begin{aligned}
&\wt{J}(\BQ)\bs\big[(\hat{\Omega}_{F_{v_{0}}} \times_{\Spf O_{F_{v_{0}}}} \Spf O_{\breve{E}_{\nu}}) \times
\wt G(\mathbb{Q}_p)/\mathbf{K}_{\wt G,p} \times\wt G(\BA^{p}_f)/{\bK}^{p}_{\wt G}\big]\\
=&\wt{J}(\BQ)\bs\big[(\hat{\Omega}_{F_{v_{0}}} \times_{\Spf O_{F_{v_{0}}}} \Spf O_{\breve{E}_{\nu}})\times\wt G(\BA_f)/{\bK}_{\wt G}\big].
\end{aligned}
\end{equation*}
The Hecke equivariance statement follows from that of \cite[Theorem  6.2.3]{RZ} and Proposition \ref{RZ-p-weil}.
  
  Finally we prove that the natural Weil descent data on the two sides of \eqref{theta} are compatible. This will conclude the proof. The Weil descent datum $\omega_{\widetilde{\mathcal{M}}_{\mathbf{K}_{\wt G, p}}}$ is obtained by changing $\rho_{0}$ to $\wt{\rho}_{0}$ and $\rho$ to $\wt{\rho}$ where
\begin{equation*}
\wt{\rho}_{0}: X_{0} \overset{\rho_{0}}{\longrightarrow}\varepsilon_{\ast} \mathbb{X}_{0} \overset{\varepsilon_{\ast}F_{\mathbb{X}_{0},\tau_{E_{\nu}}}}{\longrightarrow} \varepsilon_{\ast}(\tau_{E_{\nu}})_{\ast} \mathbb{X}_{0},\quad\text{ resp.}\quad
\wt{\rho}: X \overset{\rho}{\longrightarrow}\varepsilon_{\ast} \mathbb{X} \overset{\varepsilon_{\ast}F_{\mathbb{X},\tau_{E_{\nu}}}}{\longrightarrow} \varepsilon_{\ast}(\tau_{E_{\nu}})_{\ast} \mathbb{X}. 
\end{equation*}
This gives a point  $(X_{0}, \iota_{0}, \lambda_{0}, \wt{\rho}_{0}, X, \iota, \lambda, \wt{\rho}, \bar{\eta}^{v_{0}}_{p}) \in \widetilde{\mathcal{M}}_{\mathbf{K}_{\wt G, p}}$. The point $(X_{0}, \iota_{0}, \lambda_{0}, \rho_{0}, X, \iota, \lambda, \rho, \bar{\eta}^{v_{0}}_{p})$ defines a quasi-isogeny of abelian varieties
\begin{equation*}
\rho_{0}: A_{0} \longrightarrow \varepsilon_{\ast }\bm{A}_{0}, \quad\text{ resp.}\quad \rho: A \longrightarrow \varepsilon_{\ast }\bm{A} ,
\end{equation*}
as explained in the definition of $\Theta$. The point
$(X_{0}, \iota_{0}, \lambda_{0}, \wt\rho_{0}, X, \iota, \lambda, \wt\rho, \bar{\eta}^{v_{0}}_{p})$ defines in the same way the quasi-isogeny of
abelian varieties over $R_{[\tau_{E_{\nu}}]}$,
\begin{equation}\label{desab}
{A}_{0} \overset{\rho_{0}}{\longrightarrow} \varepsilon_{\ast} \bm{A}_0 \overset{\varepsilon_{\ast} F_{\bm{A}_0, \tau_{E_{\nu}}}}{\longrightarrow} (\varepsilon \tau_{E_{\nu}})_{\ast} \bm{A}_0,\quad\text{ resp.}\quad A \overset{\rho}\longrightarrow \varepsilon_{\ast} \bm{A} \overset{\varepsilon_{\ast} F_{\bm{A}, \tau_{E_{\nu}}}}{\longrightarrow} (\varepsilon \tau_{E_{\nu}})_{\ast} \bm{A}.
\end{equation}
Here $A_{0}$ and $A$  with their additional structure are regarded as  points of
$ (\wt\CA_{\bK_{\wt G}})^{\wedge}(R_{[\tau_{E_{\nu}}]})$. This makes sense because to be a
point of  $(\wt\CA_{\bK_{\wt G}})^{\wedge}(R)$ depends only on the
$\kappa_{{\nu}}$-algebra structure on $R$. In other words 
\begin{equation*}
(\wt\CA_{\bK_{\wt G}})(R) = (\wt\CA_{\bK_{\wt G}})(R_{[\tau_{E_{\nu}}]}).
\end{equation*}
But \eqref{desab}  is exactly how the Weil descent datum on the right hand side of (\ref{theta}) is defined. 
\end{proof}

\begin{corollary}\label{stable}
Assume that $\bK^{p}_{\U}$ is sufficiently small, then the functor $\wt\CA_{\bK_{\wt G}}$ is representable by a projective flat $O_{E, (\nu)}$-scheme $\wt\CA_{\bK_{\wt G}}$ which is a stable relative curve.
\end{corollary}
\begin{proof}
This follows from the same argument as in the proof of \cite[Theorem 7.3.3]{KRZI}.
\end{proof}

\begin{corollary}\label{main-cor}
 There exists an isomorphism of formal schemes over $\Spf O_{\breve {E}_\nu}$, 
\begin{equation*}
\scalebox{0.93}{
($\wt\CA_{\bK_{\wt G}})^\wedge\times_{\Spf O_{E_{\nu}}} \Spf O_{\breve {E}_\nu}  \simeq   J_{\U}(\BQ)\bs\big[\big(\widehat{\Omega}_{F_{v_{0}}} \times_{\Spf O_{F_{v_{0}}}} \Spf O_{\breve {E}_\nu}\big) \times  \U(\BA_f)/{\bK}_{\U}\big]\times Z^\BQ(\BQ)\bs Z^\BQ(\BA_f)/\bK^{\circ}_{Z^\BQ}$.}
\end{equation*}
The quotient in the first factor on the right-hand-side is defined using an identification of the adjoint group $J_{\U, v_{0}, {\rm ad}}(\BQ_{p})$ with $\PGL_2(F_{v_{0}})$ and through an action of $J_{\U}(\BQ)$ on $\U(\BA_f)/{\bK}_{\U}$.  For varying $\bK^{v_{0}}_{\wt{G}}$, the isomorphism is equivariant for the action of the group $\wt{G}^{v_{0}}(\BA_{f})$ which acts on both sides by Hecke correspondences.

Furthermore, the natural descent datum down to $O_{E_{\nu}}$ on the left-hand side is given on the right-hand side by 
\begin{equation*}
(\xi, h, z)\rightarrow (\omega_{\omega_{\Omega, E_{\nu}}}(\xi), h, \wt{w}z), \quad h\in \U(\BA_f),\quad z\in Z^\BQ(\BA_f)
\end{equation*}
where $\wt{w}\in Z^{\BQ}(\BQ_p)$ is the special element relative to $(\Phi^{+}_{0}, E_{\nu})$ for the uniformizer $\varpi_{\nu}$, as in Construction \ref{consemil}.  
\end{corollary}
\begin{proof}
This follows from the previous theorem by using the decomposition $\wt{J}=J_{\U}\times Z^{\BQ}$ and $\wt{G}=\U\times Z^{\BQ}$.
\end{proof}

\part{KRZ Shimura curves revisited}
\section{$p$-adic uniformization of KRZ Shimura curves revisited}\label{s:revi}
Let $K/F$ be a CM field. Let $r$ be a CM-type of rank $2$ special relative to the archimedean place $w_{0}$ of  $F$ whose extension to $K$ corresponds to two embeddings $\varphi_{0}$ and $\bar{\varphi}_{0}$ of $K$ in $\BC$. The reflex field of $r$ will be denoted by $E_{r}$. 

Let $p$ be a prime number and $\nu:\bar{\BQ}\rightarrow \bar{\BQ}_{p}$ be an embedding. The $p$-adic place which is induced on any subfield of $\bar{\BQ}$ will always be denoted by $\nu$. We fix a uniformizer $\varpi_{r,\nu}:=\varpi_{E_{r,\nu}}$ of $E_{r,\nu}$ and denote by $\kappa_{r, \nu}:=\kappa_{E_{r, \nu}}$ the residual class field of $E_{r,\nu}$, and by $f_{E_{r, \nu}}$ the inertia index.  We denote by $v_{0}$ the $p$-adic place of $F$ induced by 
\begin{equation*}
F\xrightarrow{\varphi_{0}} E_{r}\xrightarrow{\nu}\bar{\BQ}_{p}
\end{equation*}
which we will refer to as the special prime; the other $p$-adic places $v$ of $F$ will be called the banal primes. It is assumed that $v_0$ does not split in $K$. 

\subsection{The KRZ Shimura curve}
Let $(V, \varsigma)$ be the two-dimensional $K$-vector space considered in \eqref{varsigma} and let $G$ be the unitary similitude group of $(V, \varsigma)$ with rational similitude factor. For each place $v \mid p$ in $F$, let $G_{v}$ be the group of unitary similitudes of $(V_{v}, \varsigma_{v})$ with similitude factor in $\BQ_p^\times$. Then $G(\BQ_p)$ is the subgroup of $\prod_{v\mid p}G_{v}(\BQ_{p})$ where the similitudes of all factors are identical.  For each place $v \mid p$ in $F$, we fix a lattice $\Lambda_{v}\subset V_{v}$ which is almost self-dual with respect to $\varsigma_{v}$ (see the Notation section at the end of the Introduction for the meaning of this terminology).

We fix an open compact subgroup $\bK_{G}$ of $G(\BA_f)$ of the form 
\begin{equation*}
\bK_{G}=\bK_{G}^p\cdot \bK_{G, p} 
\end{equation*}
where $\bK_{G}^p\subset G(\mathbb{A}^{p})$ and $\bK_{G, p}\subset G(\BQ_{p})$. We impose the following assumptions on $\bK_{G, p}$, cf.  \cite[(7.1.8)]{KRZI}:
\begin{equation*}\label{boldK1e}
\bK_{G,p}=G(\BQ_p)\cap\big( \prod\nolimits_{v\mid p} \bK_{G, v}\big ), 
\end{equation*}
where $\bK_{G, v}=\{g\in G_{v}(\BQ_{p}): g\Lambda_{v}=\Lambda_{v}\}$ is an open compact subgroup of $G_{v}(\BQ_{p})$.

\begin{definition}\label{def:CAbK}
We define the category ${\CA}_{r,\bK_{G}}$ fibered in groupoids on $({\rm LNSch})/O_{E_{r}, (\nu)}$ which associates to  each  $O_{E_{r}, ({\nu})}$-scheme $S$ the groupoid of tuples $(A, \iota, \bar{\lambda}, \bar{\eta}^{p})$ consisting of the following data: 
\begin{enumerate}
\item[(1)] An abelian scheme $A$ over $S$, up to isogeny of degree
  prime to $p$, with an algebra homomorphism
  \begin{equation*}
    \iota: O_{K, (p)} \longrightarrow \End_{(p)}(A)
  \end{equation*}
  such that $\Lie A\otimes_{O_{E_{r}, (\nu)}}O_{E_{r,\nu}}$
  satisfies the condition
  $({\rm KC}_r)$;
\item[(2)] A $\mathbb{Q}$-homogeneous polarization $\bar{\lambda}$ of $A$
  such that the Rosati involution induces the conjugation of $K/F$;
\item[(3)] A class of $O_K$-linear isomorphisms
  \begin{equation*}
    \bar{\eta}^p: \widehat\RV^{p}(A)\rightarrow V \otimes \mathbb{A}_f^{p}  \;
    \text{mod} \; {{\bK}}^{p}_{G}  
  \end{equation*}
 which respect the forms on both sides up to a constant in
  $\mathbb{A}_f^p(1)^{\times}$ (and satisfies the usual Galois invariance property, comp. Definition \ref{def:overOE}). 
  \end{enumerate}
  We impose the following  conditions.
  \begin{enumerate}
  \item[(i)]  The action of $O_{K, (p)}$ on $\Lie A$ satisfies the Eisenstein condition  $({\rm EC}_r)$
\item[(ii)] There exists a polarization
$\lambda \in \bar{\lambda}$ such that, for the $p$-primary part  of $\ker(\lambda)$ we have $\ker(\lambda)_{p}=\prod_{v\mid p}\ker(\lambda)_{v}$, where $\ker(\lambda)_{v}$ is trivial unless $v$ is unramified in $K$ and $\inv(V_{v})=-1$, in which case ${\pi_v}_{|_{\ker(\lambda)_{v}}}=0$ and $\ker(\lambda)_{v}$ has rank $\#(\Lambda_v^\vee/\Lambda_v)$, comp. Definition \ref{def:overOE}, (ii);
\item[(iii)] There is an identity of invariants, 
\begin{equation*}\label{invcond}
\inv_{v}^r(A, \iota, \lambda)=\inv(V_{v}), \quad \text{ for all }\, v|p ,
\end{equation*}
cf. \cite[Definition  7.1.2]{KRZI}. 
\end{enumerate}
An isomorphism of such data $(A, \iota, \bar{\lambda}, \bar{\eta}^{p})\to (A', \iota', \bar{\lambda}', \bar{\eta}'^{p})$  is given by a $O_K$-linear quasi-isogeny $\phi\colon A\to A'$ of degree prime to $p$ compatible with the $\BQ$-homogeneous polarizations and the level structures.  
\end{definition}

By \cite[Proposition  7.1.5]{KRZI}, the functor ${\CA}_{r,\bK_{G}}$ is representable by  a Deligne-Mumford stack which we denote by the same symbol  ${\CA}_{r,\bK_{G}}$. It is representable by a projective scheme over $\Spec O_{E_{r}, (\nu)}$ if $\bK_G^p$ is small enough.

\subsection{Semi-local formal moduli space of KRZ CM-triples} 
We fix a point $({\bm A}, {\bm \iota}, \bar{\bm \lambda},  \bar{\bm \eta}^{p})$
of $\mathcal{A}_{r,\bK_{G}}(\bar{\kappa}_{{r, \nu}})$.
We denote by $\mathbb{X}$ the $p$-divisible group of ${\bm A}$. The action $\bm \iota$
induces an action $\iota_{\mathbb{X}}$ of $O_K\otimes\BZ_p$ on $\mathbb{X}$ and $\bm \lambda$
induces a polarization $\lambda_{\mathbb{X}}$ on $\mathbb{X}$. This construction gives a CM-triple $(\BX, \iota_{\BX}, \lambda_{\BX})$ of type  $(K\otimes \BQ_p /F\otimes\BQ_p, r)$ over $\bar{\kappa}_{{r, \nu}}$.  We introduce a semi-local version of the space $\CM_{r}$ in Definition \ref{recall-M-r} and \ref{recall-M-r-banal}.

\begin{definition}\label{KRZ-RZ-space}
We define for each $i \in \mathbb{Z}$ the functor $\mathcal{M}_{r}(i)$ on the category $\Nilp_{O_{\breve{E}_{r,\nu}}}$. Let $S$ be a scheme over $\Spf {O_{\breve{E}_{r,\nu}}}$. A point of $\mathcal{M}_{r}(i)(S)$ is given by the following data: 
\begin{enumerate}
\item[(1)] A CM-triple $(X, \iota, \lambda)$ of type
$(K \otimes \mathbb{Q}_p / F \otimes \mathbb{Q}_p, r)$ over $S$
which satisfies the conditions $({\rm KC}_r)$ and $({\rm EC}_r)$ and is $(V, \varsigma)$-principally polarized;
  
\item[(2)] An $O_K\otimes\BZ_p$-linear quasi-isogeny
\begin{equation*}
\rho: \bar{X} := X \times_{S}  \bar{S} \longrightarrow \mathbb{X} \times_{\Spec \bar{\kappa}_{r, {\nu}}}  \bar{S}  
\end{equation*}
such that, locally on $\bar S$,  $\rho$ respects the polarization $p^i \lambda$ on $\bar X$ and $\lambda_{\mathbb{X}}$ up to a factor in $(O_F \otimes \mathbb{Z}_p)^{\times}$. 
\end{enumerate}
The formal scheme over $\Spf O_{\breve{E}_{r,\nu}}$ representing this functor is also denoted by $\CM_{r}(i)$.
We  define the formal scheme
\begin{equation}
\mathcal{M}_r = \coprod\nolimits_{i \in \mathbb{Z}} \mathcal{M}_r(i) .
\end{equation}
We define the algebraic group $J$ over $\BQ_p$ such that
\begin{equation}\label{def:locJ}
J(\mathbb{Q}_p) = \{\alpha \in \Aut^{\circ}_{K \otimes \mathbb{Q}_p}(\mathbb{X})
\; | \; \alpha^{\ast} (\lambda_{\mathbb{X}}) = \mu(\alpha) \lambda_{\mathbb{X}}, \;
\mu(\alpha)\in \mathbb{Q}_p^{\times}  \}.
\end{equation}
This group acts naturally on $\mathcal{M}_r$ via the framing datum $\rho$.
\end{definition}

We consider the decomposition $(\BX, \iota_{\BX}, \lambda_{\BX})=\prod_{v\mid p}(\BX_{v}, \iota_{\BX_{v}}, \lambda_{\BX_{v}})$ and  we set
\begin{equation}\label{Endo9e} 
  J_{v}(\BQ_{p}) = \{\alpha_v \in \Aut^{\circ}_{K_{v}}
 (\mathbb{X}_{v}) \; |
  \; \alpha_v^{\ast} (\lambda_{v}) = \mu(\alpha_v) \lambda_{v},
  \; \mu(\alpha_v) \in \mathbb{Q}_p^{\times} \}.
\end{equation}

 Following \cite[\S 7.2]{KRZI}, we recall an explicit description of these groups. For this, it is
convenient to replace the bilinear form $\varsigma_{v}$ by the
$F_{v}$-bilinear form
\begin{equation*}
  \tilde{\varsigma}_{v}: V_{v} \times V_{v} \rightarrow
  F_{v}, 
\end{equation*}
which is defined by
\begin{equation*}
  \T(a\tilde{\varsigma}_{v}(x_1, x_2)) =
  \varsigma_{v}(a x_1, x_2), \quad a \in F_{v}, 
  \end{equation*}
for $\T(a) = \Trace_{F_{v}/\mathbb{Q}_p} \vartheta^{-1}a$, where as usual $\vartheta \in O_F$ is the different of $F/\mathbb{Q}_p$.  

Then $J_{v}$ is explicitly given as follows:
\begin{enumerate}
\item For the special prime $v_{0}$, we obtain from the local CM triple $(\BX_{v_{0}}, \iota_{\BX_{v_{0}}}, \lambda_{\BX_{v_{0}}})$ a two-dimensional $K_{v_{0}}$-vector space $\bar{V}_{v_{0}}$ with an anti-hermitian alternating form
\begin{equation*}
  \bar{\varsigma}_{v_{0}}: \bar{V}_{v_{0}} \times
  \bar{V}_{v_{0}} \rightarrow F_{v_{0}}  .
\end{equation*}
  This comes from the construction in \cite[\S 3, p. 1209]{KRalt} when $K/F$ is unramified, resp.  the construction in \cite[\S 4, p. 1215]{KRalt} in the ramified case (in  loc.~cit, the hermitian space is denoted by $C$ and its anti-hermitian form is denoted by $h$). In either case $\inv(\bar{V}_{v_{0}})=1$. The contraction functor \cite[\S 5.2, \S 5.3]{KRZI} furnishes an isomorphism between $J_{v_{0}}(\BQ_p)$ and 
\begin{equation*}\label{Endo12e}
  G(\bar{V}_{v_{0}},
  \bar{\varsigma}_{v_{0}})(\BQ_p):= \{g \in \Aut_{K_{v_{0}}} (\bar{V}_{v_{0}}) \; |
  \; \bar{\varsigma}_{v_{0}}(gx, gy) = \mu_{v_{0}}(g) \bar{\varsigma}_{v_{0}}(x,y), \; 
  \; \mu_{v_{0}}(g) \in \mathbb{Q}_p^{\times}\} ,
\end{equation*}
cf. \cite[(7.2.6)]{KRZI}.

\item For a banal prime $v$, we consider the image $(C_{\mathbb{X}_{v}}, \phi_{\mathbb{X}_{v}})$ by the polarized contraction functor 
$\mathfrak{C}_{r_{v}, k}^{\rm pol}$ of $(\BX_{v}, \iota_{\BX_{v}}, \lambda_{\BX_{v}})$
\cite[Theorem 4.5.11]{KRZI}. There is an isomorphism
$ (C_{\mathbb{X}_{v}}, \phi_{\mathbb{X}_{v}}) \cong
(\Lambda_{v}, \tilde{\varsigma}_{v})$
and we obtain an isomorphism between $J_{v}(\BQ_p)$ and
\begin{equation*}
G(V_{v}, \tilde{\varsigma}_{v})(\BQ_p):= \{g \in \Aut_{K_{v}} (\bar{V}_{v}) \; |
\; \tilde{\varsigma}_{v}(gx, gy) = \mu_{v}(g) \tilde{\varsigma}_{v}(x,y), \; 
 \; \mu_{v}(g) \in \mathbb{Q}_p^{\times}  \} ,
\end{equation*}
cf. \cite[(7.2.8)]{KRZI}.
\end{enumerate}

These isomorphisms are induced by isomorphisms of algebraic groups over $\BQ_p$. To use a uniform notation, we set $(\bar{V}_{v}, \bar{\varsigma}_{v}) = (V_{v}, \tilde{\varsigma}_{v})$  for $v \neq v_{0}$ and set
$\bar{G}_{v} =  G(\bar{V}_{v}, \bar{\varsigma}_{v})$ for all $v\mid p$, in particular, $G_{v} = \bar{G}_{v}$ for $v\neq v_{0}$. We thus have fixed an isomorphism $J_{v} \cong \bar{G}_{v}$
for all $v\mid p$.  

We define
\begin{equation*}
\bar{V}_p = \bigoplus\nolimits_{v\mid p} \bar{V}_{v} ,
\end{equation*}
which is an $K \otimes \mathbb{Q}_p$-module. Let 
\begin{equation*}
  \bar{\varsigma}_p: \bar{V}_{p} \times \bar{V}_{p}
  \rightarrow F \otimes \mathbb{Q}_p 
 \end{equation*}
be the orthogonal sum of the forms $\bar{\varsigma}_{v}$.  

We define the algebraic group $\bar{G}$ over $\mathbb{Q}_p$ by
\begin{equation*}
\bar{G}(\mathbb{Q}_p) := G(\bar{V}_p, \bar{\varsigma}_p)= 
\{g \in \Aut_{K \otimes \mathbb{Q}_p} (\bar{V}_p) \; | \;
\bar{\varsigma}_p (gx, gy) = \mu(g) \bar{\varsigma}_p (x, y),
\;  \; \mu(g) \in \mathbb{Q}_p^{\times}\} .
\end{equation*}
The previous discussions imply that  we have $\bar{G} \simeq J$ as algebraic groups over $\BQ_{p}$,  
and 
\begin{equation*}
\bar{G}(\mathbb{Q}_p)\subset\prod\nolimits_{v\mid p} \bar{G}_{v}(\BQ_{p}) 
\end{equation*}
is the subgroup of elements whose components have the same similitude factor. 

In order to discuss the Weil descent datum, we define the group $\bar{G}'(\BQ_{p})$ by
\begin{equation*}
\bar{G}^{\prime}(\BQ_{p}) = \{g \in \Aut_{K \otimes \mathbb{Q}_p} (\bar{V}_p) \; |
  \; \bar{\varsigma}_p(gx, gy) = \mu(g) \bar{\varsigma}_p(x,y), \; 
  \; \mu(g) \in p^{\mathbb{Z}} (O_F \otimes \mathbb{Z}_p)^{\times}  \} .
\end{equation*}
Note that this is a misuse of notation since $\bar{G}^{\prime}(\BQ_{p})$ is not the set of $\BQ_p$-points of an algebraic group. Then  
\begin{equation*}
\bar{G}^{\prime}(\BQ_{p})\subset\prod\nolimits_{v\mid p}\bar{G}'_v(\BQ_{p}) ,
\end{equation*}
where we define $\bar{G}^{\prime}_{v}(\BQ_{p})$ by
\begin{equation*}
\bar{G}^{\prime}_{v}(\BQ_{p}) = \{g \in \Aut_{K_{v}} (\bar{V}_{v}) \; | \; \bar{\varsigma}_{v}(gx, gy) =\mu_{v}(g) \bar{\varsigma}_{v}(x,y), \; \text{for}
\; \mu_{v}(g) \in p^{\mathbb{Z}} O_{F_{v}}^{\times} \} . 
\end{equation*}
Note that $\bar{G}^{\prime}(\BQ_{p})\supset\bar{G}(\mathbb{Q}_p)$.

\subsection{The KRZ descent datum revisited}
Let $\tau_{E_{r, \nu}} \in \Gal(\breve{E}_{r, \nu}/E_{r,\nu})$ be the Frobenius
automorphism. 
The relative Frobenius morphism 
\begin{equation*}
F_{\mathbb{X},\tau_{E_{r, \nu}}}: \BX\rightarrow (\tau_{E_{r,\nu}})_{\ast}\BX
\end{equation*}
defines a Weil descent datum on the functors, 
\begin{equation*}\label{KneunU14e}
\omega_{\mathcal{M}_r}: \mathcal{M}_r(i)(S) \longrightarrow \mathcal{M}_r(i + f_{E_{r, \nu}})(S_{[\tau_{E_{r, \nu}}]}).
\end{equation*}
The goal of this section is to give a more explicit description of the descent datum of \cite[Corollary 7.2.7]{KRZI}.

We define the group $\hat{G}'(\mathbb{Q}_p)$ as the union of the following sets $\hat{G}^{\prime}(i)$ for $i \in \mathbb{Z}$,
\begin{equation*}
  \hat{G}^{\prime}(i) = \{ (c, g_{v}) \in
  p^{i} O_{F_{v_{0}}}^{\times} \times \prod\nolimits_{v\mid p, v\neq v_{0}}
  \bar{G}^{\prime}_{v}(\BQ_{p}) \; | \;
  \mu_{v}(g_{v}) \in p^{i} O_{F_{v}}^{\times} \;
  \text{for all} \; v\} .
\end{equation*}
The similitude map $\mu_{v_{0}}: \bar{G}'_{v_{0}}(\BQ_{p}) \rightarrow p^{\mathbb{Z}} O_{F_{v_{0}}}^{\times}$
induces  homomorphisms
\begin{equation*}
\bar{G}^{\prime}(\BQ_{p}) \rightarrow \hat{G}'(\mathbb{Q}_p)  \quad \text{ and } \quad G(\mathbb{Q}_p) \rightarrow \hat{G}'(\mathbb{Q}_p).
\end{equation*}   
We define $\hat{G}^{\prime}(\mathbb{Z}_p)\subset \hat{G}^{\prime}(\mathbb{Q}_p)$ to be the subgroup of elements
$(c, g_{v})$ such that $c \in O_{F_{v_{0}}}^{\times}$ and
$g_{v}\Lambda_{v}= \Lambda_{v}$. 

For the next definition, we recall from \S \ref{ss:specialbanal}  the definition of $\Pi_v$ (when $K_v/F_v$ is ramified) and of $\pi_v$ (when $K_v/F_v$ is unramified or split). 
\begin{definition}\label{Unif3d}
We consider the following element 
${w}'_r = (w'_{r, v_{0}} ,{w}'_{r, v}) \in \hat{G}'(\mathbb{Q}_p)$ where
\begin{enumerate}
\item[(1)] $w'_{r, v_{0}} = p^{f_{E_{r, \nu}}}$.

\item[(2)] If $K_{v}/F_{v}$ is ramified, $w'_{r,v}$ is the multiplication
\begin{equation*}
\Pi_{v}^{e_{v}f_{E_{r,\nu}}}: \bar V_{v} \longrightarrow  \bar V_{v}.
\end{equation*}
    
\item[(3)] If $K_{v}/F_{v}$ is unramified, we define $w'_{r,v}$ as the multiplication
\begin{equation*}
\pi_{v}^{e_{v}f_{E_{r,\nu}}/2}: \bar V_{v}
\longrightarrow \bar V_{v}.
\end{equation*}

\item[(4)] If $K_{v} = F_{v} \times F_{v}$ is split, we have the decomposition 
$\bar V_{v}=\bar V_{v, 1}\oplus\bar V_{v, 2}$.
We consider the numbers $a_{1, v}$ and $a_{2, v}$  defined by $\dim \BX_{v, 1}=2a_{1, v}$ and  $\dim \BX_{v,2}=2a_{2, v}$. We set
\begin{equation*}
a_{1,E_{r, \nu}} = a_{1, v}\frac{f_{E_{r, \nu}}}{f_{v}}, \quad \text{ resp. }\quad a_{2,E_{r, \nu}} = a_{2, v}\frac{f_{E_{r,\nu}}}{f_{v}} .
\end{equation*}
Then we have $a_{1,E_{r, \nu}} + a_{2,E_{r, \nu}} = e_{v}f_{E_{r, \nu}}$. We define $w'_{v}$ to be the multiplication by $\pi_{v}^{a_{1,E_{r,\nu}}}$ on $\bar V_{v, 1}$ and the
multiplication by $\pi_{v}^{a_{2,E_{r,\nu}}}$ on $\bar V_{v, 1}$. 
\end{enumerate}
\end{definition}
For a banal prime $v$, the element $w'_{r, v}$ is clearly an element of the center of $\bar{G}'_{v}(\mathbb{Q}_p)$ which is considered as a subgroup of $K^{\times}_{v}$. The similitudes of these elements are given by
\begin{equation*}
\mu_{v}(w'_{r, v})=\pi^{e_{v}f_{E_{r,\nu}}}_{v}.
\end{equation*}
Therefore $\mu_{v}(w'_{r, v})=p^{f_{E_{r,\nu}}}u_v$ for some $u_v\in O^{\times}_{F_{v}}$. 

Let $v\neq v_{0}$. We consider the local CM type  $\Phi^{+}_{v} =\Phi^+_{v, r}= \{\varphi \in  \Phi_{v} \mid r_{\varphi} =2 \}$.  
 Let 
\begin{equation*}
\mathfrak{r}_{\Phi^{+}_{v}, E_{r,\nu}}: \mathrm{Res}_{E_{r,\nu}/\BQ_{p}}\mathbb{G}_{m}\rightarrow Z^{\BQ}_{v}
\end{equation*}
be the local reflex norm with respect to the data $(\Phi^{+}_{v}, E_{r,\nu})$. We introduce the special element with respect to $(\Phi^{+}_{v}, E_{r,\nu})$ for the uniformizer $\varpi_{r,\nu}$ in $Z^{\BQ}_{v}(\BQ_{p})\subset K^{\times}_{v}$, 
\begin{equation*}
w_{r, v}:=\mathfrak{r}_{\Phi^{+}_{v}, E_{r,\nu}}(\varpi_{r,\nu}) ,
\end{equation*} 
cf. Construction \ref{sp-elt-loc}.  Note that $Z^{\BQ}_{v}(\BQ_{p})$ is contained in the center of $\bar{G}_{v}(\BQ_{p})$.

\begin{lemma}\label{w-transfer}
Let $v\neq v_{0}$, then the element $w_{r, v}$ agrees with $w'_{r, v}$ up to a unit in $O_{K_{v}}$.
\end{lemma}

\begin{proof}
We first remark that both $w_{r, v}$ and $w'_{r, v}$ are elements of the center of $\bar{G}'_{v}(\BQ_{p})$ which is a subgroup of $K^{\times}_{v}$, therefore the statement of the lemma makes sense. But we have that $\mu_{v}(w'_{r, v})=p^{f_{E_{r,\nu}}}u_v$ for some $u_v\in O^{\times}_{F_{v}}$. On the other hand,
\begin{equation}\label{co:sp}
\mu_{v}(w_{r,v})=\Nm_{E_{r, \nu}/\BQ_p}(\varpi_{r,\nu})=p^{f_{E_{r,\nu}}}u,\quad  u\in\BZ^\times_p.
\end{equation}
The lemma follows.
\end{proof} 

We introduce the group
\begin{equation*}
 \hat{G}(\mathbb{Q}_p) = \{ (c, g_{v}) \in
 \mathbb{Q}_{p}^{\times} \times \prod\nolimits_{v\mid p,\; v\neq v_{0}} G_{v}
 \; | \;
 \mu_{v}(g_{v}) = c, \; \forall v\neq v_{0}\}. 
\end{equation*}
There are natural homomorphisms
\begin{equation*}
G(\mathbb{Q}_p) \rightarrow \hat{G}(\mathbb{Q}_p) \quad\text{ and }\quad \bar{G}(\mathbb{Q}_p) \rightarrow \hat{G}(\mathbb{Q}_p) .
\end{equation*}  
For the second map, we used that  in the definition of $\hat{G}(\mathbb{Q}_p)$ we can replace $G_{v}$ by
$\bar{G}_{v}$. In particular
the groups $G(\mathbb{Q}_p)$ and 
$J(\mathbb{Q}_p)$   act on $\hat{G}(\mathbb{Q}_p)$. 
We denote by
$\hat{G}(\mathbb{Z}_p) \subset \hat{G}(\mathbb{Q}_p)$ the subgroup of all 
$(c, g_{v})$ such that $c \in \mathbb{Z}_p^{\times}$ and 
$g_{v} \Lambda_{v} = \Lambda_{v}$. 

\begin{construction}
We have $w'_{r, v_{0}}=p^{f_{E_{r,\nu}}}$ (cf. Definition \ref{Unif3d}) and, by \eqref{co:sp},  $\mu_{v}(w_{r, v})=p^{f_{E_{r,\nu}}}u$ for $u\in \BZ_p^\times$ independent of $v\neq v_0$. Hence we can define the element
\begin{equation}\label{wr}
w_{r}=(uw'_{r, v_{0}}, w_{r, v})\in \hat{G}(\BQ_{p}) ,
\end{equation}
which we refer to as the \emph{special element of $\hat{G}(\BQ_{p})$} for the uniformizer $\varpi_{r,\nu}$ of $E_{r, \nu}$.
\end{construction}

By \cite[Corollaries 6.3.4, 6.4.5, 6.5.3]{KRZI}, the inclusion  $ \hat{G}(\mathbb{Q}_p)\subset  \hat{G}'(\mathbb{Q}_p)$ induces a bijection 
\begin{equation}\label{hat-to-prime}
  \hat{G}(\mathbb{Q}_p)/\hat{G}(\mathbb{Z}_p) \isoarrow
  \hat{G}'(\mathbb{Q}_p)/\hat{G}'(\mathbb{Z}_p).
\end{equation}
 By Lemma \ref{w-transfer},  the class of the special element $w_{r}$  maps under the bijection \eqref{hat-to-prime} to the class of $w'_{r}$. 

\begin{proposition}\label{compfo}
There exists an isomorphism
\begin{equation*}
{\mathcal{M}}_{r} \isoarrow  (\wh{\Omega}_{F_{v_{0}}} \times_{\Spf O_{F_{v_{0}}}} \Spf O_{\breve{E}_{r, \nu}}) \times \hat{G}(\mathbb{Q}_p)/\hat{G}(\mathbb{Z}_p) 
\end{equation*}
which is equivariant with respect to the natural action of $\bar{G}(\BQ_{p})$
on both sides. 
  
The Weil  descent datum $\omega_{\mathcal{M}_r}$ on the left-hand side corresponds on the right hand side to
\begin{equation*}
(\xi, g) \mapsto (\omega_{\Omega, E_{r,\nu}}(\xi), w_rg), \quad g \in \hat{G}(\mathbb{Q}_p) ,
\end{equation*} 
where $w_{r}$ is the special element  in $\hat{G}(\BQ_{p})$ for the uniformizer $\varpi_{r,\nu}$.
\end{proposition}
\begin{proof} 
By \cite[Proposition   7.2.5]{KRZI}, there is an isomorphism
\begin{equation}\label{hat-G'-weil}
{\mathcal{M}}_{r} \isoarrow  (\wh{\Omega}_{F_{v_{0}}} \times_{\Spf O_{F_{v_{0}}}} \Spf O_{\breve{E}_{r, \nu}}) \times \hat{G}'(\mathbb{Q}_p)/\hat{G}'(\mathbb{Z}_p) 
\end{equation}
which is equivariant with respect to the action of $\hat{G}^{\prime}(\BQ_{p})$
on both sides. By loc.~cit., the Weil  descent datum $\omega_{{\mathcal{M}}_r}$ relative to
$O_{E_{r, \nu}}$ on the left-hand side corresponds on the
right-hand side to
\begin{equation*}
 (\xi, g) \mapsto (\omega_{\Omega, {E_{r,\nu}}}(\xi), w'_rg), \quad
 g \in \hat{G}'(\mathbb{Q}_p). 
\end{equation*} 

We may multiply each $w'_{r, v}$ by a unit in $K^{\times}_{v}$ in the $v$-component of $w'_r$, cf. Definition \ref{Unif3d}.  By \cite[Remark  7.2.6]{KRZI}, this does not change the assertion on the descent datum on both sides of \eqref{hat-G'-weil}. By Lemma \ref{hat-to-prime},  we may replace $w'_{r}$ by $w_{r}$ and the proposition follows. 
\end{proof}

Now we can modify the main $p$-adic uniformization theorem  \cite[Theorem 7.3.3]{KRZI} in the following way. Recall $({\bm A}, {\bm \iota}, \bar{\bm \lambda},  \bar{\bm \eta}^{p})\in \mathcal{A}_{r,\bK_{G}}(\bar{\kappa}_{{r, \nu}})$. We define 
\begin{equation}
J(\BQ)= \{\alpha \in \Aut^{\circ}_{K }({\bm A})
\; | \; \alpha^{\ast} ({\bm \lambda}) = \mu(\alpha) {\bm \lambda}, \;
\mu(\alpha)\in \mathbb{Q}^{\times}  \}.
\end{equation}
Then $J(\BQ)$ is the set of $\BQ$-rational points of an algebraic group $J$ over $\BQ$ such that $J(\BR)$ is compact modulo center and the $\BQ_p$-points can be identified with $J(\BQ_p)$ from \eqref{def:locJ}, cf. \cite[proof of Proposition 7.1.11]{KRZI}. Moreover $J(\BQ)$ acts on $G(\BA_{f})$ as in \cite[(7.3.6)]{KRZI} which defines an isomorphism $J(\BA^{p}_{f})\simeq G(\BA^{p}_{f})$.
\begin{theorem}\label{KRZ-improve-1}
There is an isomorphism of formal schemes over $\Spf O_{\breve{E}_{r, \nu}}$,
\begin{equation*}\label{unifor}
\widehat{\CA}_{r,\bK_{G}}\times_{\Spf O_{E_{r,\nu}}} \Spf O_{\breve{E}_{r,\nu}} \simeq J(\BQ) \backslash [(\widehat{\Omega}_{F_{v_{0}}}\times_{\Spf O_{F_{v_{0}}}} \Spf O_{\breve{E}_{r,\nu}})\times \hat{G}(\BQ_{p})/\hat{G}(\BZ_{p})\times {G}(\mathbb{A}^{p}_{f})/\bK^{p}_{G}] .
\end{equation*}
 For varying $\bK^{p}_{G}$, the isomorphism is compatible with the action of $G(\BA_f^p)$ through Hecke correspondences on both sides. 

This isomorphism is compatible with the Weil descent data if we equip the right-hand side with the Weil descent datum
\begin{equation*}
(\xi, h, g)\mapsto (\omega_{\Omega, {E_{r,\nu}}}(\xi), w_{r}h, g)\quad h\in \hat{G}(\BQ_{p}), \quad g\in {G}(\mathbb{A}^{p}_{f})
\end{equation*}
where $w_{r}$ is the special element  \eqref{wr}  in $\hat{G}(\BQ_{p})$, for the uniformizer $\varpi_{r,\nu}$.
\end{theorem}
\begin{proof}
By \cite[Theorem  7.3.3]{KRZI}, we have the isomorphism
\begin{equation}\label{unifor}
\widehat{\CA}_{r,\bK_{G}}\times \Spf O_{\breve{E}_{r,\nu}} \simeq J(\BQ) \backslash [(\widehat{\Omega}_{F_{v_{0}}}\times_{\Spf O_{F_{v_{0}}}} \Spf O_{\breve{E}_{r,\nu}})\times \hat{G}^{\prime}(\BQ_{p})/\hat{G}^{\prime}(\BZ_{p})\times G(\mathbb{A}^{p}_{f})/\bK^{p}_{G}].
\end{equation}
whose descent datum is given by 
\begin{equation*}
(\xi, h, g)\mapsto (\omega_{\Omega, {E_{r,\nu}}}(\xi), w^{\prime}_{r}h, g).
\end{equation*}

We can modify the $v$-component $w^{\prime}_{r, v}$ of $w^{\prime}_{r}$ by a unit in $O_{K_{v}}$ for each banal prime $v$, cf.  \cite[Remark 7.2.6]{KRZI};  and we can modify $w^{\prime}_{r, v_0}$ by a unit in $\BZ_p$.  By Proposition \ref{compfo},  we can therefore rewrite the descent datum as 
\begin{equation*}
(\xi, h, g)\mapsto (\omega_{\Omega, {E_{r,\nu}}}(\xi), w_{r}h, g).
\end{equation*}
Now we have an isomorphism 
\begin{equation*}
\hat{G}(\BQ_{p})/\hat{G}(\BZ_{p})\cong \hat{G}^{\prime}(\BQ_{p})/\hat{G}^{\prime}(\BZ_{p})
\end{equation*}
and the class of $w_{r}$ in $\hat{G}(\BQ_{p})/\hat{G}(\BZ_{p})$ is sent to the class of $w^{\prime}_{r}$ in $\hat{G}^{\prime}(\BQ_{p})/\hat{G}^{\prime}(\BZ_{p})$. The assertion follows.
\end{proof}

Using the isomorphism $G(\BQ_{p})/G(\BZ_{p})\simeq \hat{G}(\BQ_{p})/\hat{G}(\BZ_{p})$ as explained in \cite[(7.3.10)]{KRZI}, we have the following isomorphisms
\begin{equation*}
\begin{aligned}
&J(\BQ) \backslash [(\widehat{\Omega}_{F_{v_{0}}}\times_{\Spf O_{F_{v_{0}}}} \Spf O_{\breve{E}_{r,\nu}})\times \hat{G}(\BQ_{p})/\hat{G}(\BZ_{p})\times G(\mathbb{A}^{p}_{f})/\bK^{p}_{G}]\\
\simeq &J(\BQ) \backslash [(\widehat{\Omega}_{F_{v_{0}}}\times_{\Spf O_{F_{v_{0}}}} \Spf O_{\breve{E}_{r,\nu}})\times {G}(\BQ_{p})/{G}(\BZ_{p})\times G(\mathbb{A}^{p}_{f})/\bK^{p}_{G}]\\
\simeq &J(\BQ) \backslash [(\widehat{\Omega}_{F_{v_{0}}}\times_{\Spf O_{F_{v_{0}}}} \Spf O_{\breve{E}_{r,\nu}})\times G(\mathbb{A}_{f})/\bK_{G}].\\
\end{aligned}
\end{equation*}
In the resulting isomorphism 
\begin{equation}\label{beaut}
\widehat{\CA}_{r, \bK_{G}}\times_{\Spf O_{E_{r,\nu}}}\Spf O_{\breve{E}_{r,\nu}}\simeq J(\BQ) \backslash [(\widehat{\Omega}_{F_{v_{0}}}\times_{\Spf O_{F_{v_{0}}}} \Spf O_{\breve{E}_{r,\nu}})\times G(\mathbb{A}_{f})/\bK_{G}]
\end{equation}
 the element $w_{r}$ induces a non-explicit automorphism of ${G}(\BQ_{p})/{G}(\BZ_{p})$ (its $v_0$-component is not well-determined) and thus the descent datum of the isomorphism
is non-explicit in this formulation. 

Choose an extension $\varphi_0: K\to \bar\BQ$ of $\varphi_0: F\to\bar\BQ$. We define a classical CM-type by
\begin{equation}
\Phi^+=\Phi^{+}_{\varphi_0, r}=\{\varphi_0\} \cup \{\varphi\in\Phi\mid r_\varphi=2\} .
\end{equation}
Let $E_{\Phi^+}$ be the reflex field of $\Phi^+$. We define $E=E_{\Phi^{+}}\varphi_0(K)=E_r\varphi_0(K)$. 

We next explain that if we extend the isomorphism in \eqref{beaut} to $\Spf O_{\breve{E}_{\nu}}$, then the descent datum down to $O_{E_\nu}$ becomes explicit.  For this, we recall the  semi-local  special element $\wt w=\frak{r}_{\Phi^+, E_\nu}(\varpi_\nu)\in Z^\BQ(\BQ_p)$, relative to $(\Phi^{+}, E_{\nu})$ for the uniformizer $\varpi_{\nu}$ of $E_\nu$, cf. Construction \ref{consemil}.  For the next statement, we note that $Z^{\BQ}(\BQ_{p})$  lies in the center of $G(\BQ_{p})$. 
\begin{theorem}\label{KRZ-improve-2}
There is an isomorphism of formal schemes over $\Spf O_{\breve{E}_{\nu}}$,
\begin{equation*}\label{unifor}
\begin{aligned}
\widehat{\CA}_{r,\bK_{G}}\times_{\Spf O_{{E}_{r,\nu}}} \Spf O_{\breve{E}_{\nu}} &\simeq J(\BQ) \backslash [(\widehat{\Omega}_{F_{v_{0}}}\times_{\Spf O_{F_{v_{0}}}} \Spf O_{\breve{E}_{\nu}})\times G(\BQ_{p})/\bK_{G, p} \times  G(\mathbb{A}^{p}_{f})/\bK^{p}_{G}] \\
&=J(\BQ) \backslash [(\widehat{\Omega}_{F_{v_{0}}}\times_{\Spf O_{F_{v_{0}}}} \Spf O_{\breve{E}_{\nu}})\times  G(\mathbb{A}_{f})/\bK_{G}] .
\end{aligned}
\end{equation*}
For varying $\bK^{p}_{G}$, this isomorphism is equivariant for the action of $G(\BA^{p}_{f})$ through Hecke correspondences on both sides.

This isomorphism is compatible with the Weil descent data down to $O_{E_\nu}$ if we equip the right-hand side with the Weil descent datum
\begin{equation*}
(\xi, h, g)\mapsto (\omega_{\Omega,{E_{\nu}}}(\xi), \tilde{w}h, g), \quad h\in {G}(\BQ_{p}),\quad g\in {G}(\mathbb{A}^{p}_{f}).
\end{equation*}
\end{theorem}

\begin{proof}
We extend the isomorphism of Theorem \ref{KRZ-improve-1} to $\Spf O_{\breve{E}_{\nu}}$ and obtain
\begin{equation*}
\widehat{\CA}_{r,\bK_{G}}\times \Spf O_{\breve{E}_{\nu}} \simeq J(\BQ) \backslash [(\widehat{\Omega}_{F_{v_{0}}}\times_{\Spf O_{F_{v_{0}}}} \Spf O_{\breve{E}_{\nu}})\times \hat{G}(\BQ_{p})/\hat{G}(\BZ_{p})\times {G}(\mathbb{A}^{p}_{f})/\bK^{p}_{G}].
\end{equation*}
This isomorphism is compatible with the Weil descent data if we equip the right-hand side with the Weil descent datum
\begin{equation*}
(\xi, h, g)\mapsto (\omega_{\Omega,{E_{\nu}}}(\xi), w_{r}^{f_{E_\nu}/f_{E_{r, \nu}}}h, g)
\end{equation*}
where $w_{r}$ is the special element given by \eqref{wr}. Note that, regarding $\hat{G}(\BQ_{p})$ as a subset of $\prod_{v\mid p}\hat{G}_{v}(\BQ_{p})$, the $v_{0}$-component of $w_{r}$ is given by $up^{f_{E_{r, \nu}}}$, where $u\in\BZ_p^\times$ is given by $\Nm_{K_{v}/F_{v}}({w}_{v})=p^{f_{E_{r,\nu}}}u$, for any banal $v$. By Remark \ref{sp-elt-rmk},  $\wt{w}_v$ differs from $w_{r, v}^{f_{E_\nu}/f_{E_{r, \nu}}}$ by a unit,  for any banal $v$. Also ${\rm Nm}_{K_{v_0}/\BQ_p}(\wt w_{v_0})=p^{f_{E_{ \nu}}}$, up to a unit in $\BZ_p^\times$. It follows that the above Weil descent is given by
\begin{equation*}
(\xi, h, g)\mapsto (\omega_{\Omega,{E_{\nu}}}(\xi), \wt{w}h, g) ,
\end{equation*}
as claimed.

\end{proof}

\begin{remark}
The above theorem is compatible with Theorem \ref{main} in the case of the CM-type $\Phi^+=\Phi^+_{\varphi_0, r}$ used above. More precisely,  there is a commutative diagram compatible with all the Weil descent data down to $O_{E_\nu}$, 
\begin{equation*}
\begin{tikzcd}
(\wt\CA_{\bK_{\wt G}})^\wedge\times_{\Spf O_{E_{\nu}}} \Spf O_{\breve {E}_\nu} \arrow[r, "\sim"] \arrow[d] &   \wt{J}(\BQ)\bs\big[\big(\widehat{\Omega}_{F_{v_{0}}} \times_{\Spf O_{F_{v_{0}}}} \Spf O_{\breve {E}_\nu}\big) \times
 \wt G(\BA_f)/{\bK}_{\wt G}\big]  \arrow[d] \\
\widehat{\CA}_{r, \bK_{G}}\times_{\Spf O_{{E}_{r, \nu}}} \Spf O_{\breve{E}_{\nu}} \arrow[r, "\sim"]           & J(\BQ) \backslash \big[\big(\widehat{\Omega}_{F_{v_{0}}}\times_{\Spf O_{F_{v_{0}}}} \Spf O_{\breve{E}_{\nu}}\big)\times  G(\mathbb{A}_{f})/\bK_{G}\big].    
\end{tikzcd}
\end{equation*}
Here the left-vertical map is the natural map  $\wt{\CA}_{\bK_{\wt G}}\rightarrow \CA_{r, \bK_{G}}$ induced by $(A_0, \iota_0, \lambda_0, A, \iota, \lambda, \bar\eta^p)\mapsto (A, \iota, \lambda, \bar\eta^p)$ and the right vertical map is induced by the natural map from $\wt{G}$ to $G$, resp., from $\wt{J}$ to $J$.

\end{remark}

\section{KRZ Shimura curves vs RSZ Shimura curves}\label{s:versus}

Now we compare the Shimura curves in \cite{KRZI} that are associated to the group $G$ of unitary similitudes with the Shimura curves associated to $\wt G$. 
Recall the exact sequence of algebraic groups over $\BQ$,
$$
1\to \U\to G\overset{\mu}{\longrightarrow} \BG_m\to 1 .
$$
To make the connection between the moduli stacks $ \wt{\CA}_{\bK_{\wt G}}$ and $ {\CA}_{\bK_{G}}$, we make the assumption  
\begin{equation}\label{inthyp}
\bK_\U=\bK_G\cap \U(\BA_f), \quad \mu(\bK_G)=\wh \BZ^\times .
\end{equation}  
Then $\bK_G$ and $\bK_\U$ determine each other uniquely. Indeed, $\bK_G$ is the maximal open compact subgroup of $G(\BA_f)$ whose intersection with $\U(\BA_f)$ is equal to $\bK_\U$. 

We now pass to deeper level structures at places over  $p$. For each place $v\mid p$
of $F$, we consider first the  open compact subgroup $\mathbf{K}_{G, v} \subset G_{v}(\BQ_{p})$ which is defined as the quasi-parahoric given by $\mathrm{Stab}_{G_{v}}(\Lambda_{v})$.
We will assume that there exist places $v$ which are banal
since our deeper level structures exist only in this case. 

For the special prime $v_{0}$ we set
$\mathbf{K}_{G, v_0}^{\star} = \mathbf{K}_{G,v_{0}}$. For each  banal $v$, we choose an open subgroup of  
$\mathbf{K}_{G,v}^{\star} \subset \mathbf{K}_{G, v}$.
We set
\begin{equation*}
  \mathbf{K}_{G,p}^{\star} = \{g = (g_{v}) \in G(\mathbb{Q}_p) \; | \;
  g_{v} \in \mathbf{K}_{G,v}^{\star}, \forall v\mid p \}. 
\end{equation*}
This says that $\mu_{v}(g_{v})$ is independent of
$v$. We also introduce
\begin{equation}\label{banK}
  \mathbf{K}_{G, p}^{\star, {v_{0}}} = \{(g_{v}) \in \prod\nolimits_{v\mid p, v\neq v_{0}} \mathbf{K}^{\star}_{G, v}\; | \;
  \mu_{v}(g_{v}) = c \in \mathbb{Z}_p^{\times}, \;
  \text{independent of} \; v\} .
\end{equation}
This is a subgroup of
\begin{equation*}
  G^{v_{0}}(\mathbb{Q}_p) = \{(g_{v}) \in 
  \prod\nolimits_{v\mid p, v\neq v_{0}} G_{v}(\BQ_{p})\; | \;
  \mu_{v}(g_{v}) = c \in \mathbb{Q}_p^{\times}, \;
  \text{independent of} \; v\}. 
  \end{equation*}

\subsection{At the level of moduli problems over $E_{r}$}
Let $r$ be a special CM type of rank $2$ with reflex field $E_{r}$ and let $\Phiplus$ be a classical CM type with reflex field $E_{\Phiplus}$. Let $E_{\Phiplus}^\ab$ be the maximal abelian extension of $E_{\Phiplus}$, and let $E^\sharp=E_r E_{\Phiplus}^\ab $, an abelian extension of $E=E_r E_{\Phiplus}$. 

Let  ${\bK}^\star_{G} = {\bK}^\star_{G,p} \cdot {\bK}^p_{G}$. Consider the following moduli problem $\mathcal{A}_{r, {\bK}^\star_{G}}$ on $({\rm LNSch}/E_{r})$.  It associates to an $E_{r}$-scheme $S$ the set of isomorphism classes of tuples $(A, \iota, \bar\lambda, \bar\eta)$ where all data are the same as in Definition \ref{def:CAbK}, except that (3) is replaced by 

($3'$) A  $\bK^{\star}_{G}$-class   $\bar\eta$ of $K$-linear similitudes 
\[\eta\colon \wh \RV(A)  \isoarrow V\otimes_{K}\BA_{K, f}.
\]
Here the rational Tate module is equipped with its natural anti-hermitian form arising by contraction from its polarization form. 

When $\bK_G^p$ is sufficiently small, this moduli problem is represented by a projective scheme $\CA_{\bK^{\star}_{G}, E}$ which is the \emph{canonical model} of the Shimura variety ${\rm Sh}_{\bK^{\star}_{G}}(G, X_{G})$ over $E$. Let $A^{[\Lambda_{0}]}_{0}$ be the coarse moduli space associated to the Deligne-Mumford stack $\CA_0^{[\Lambda_0]}$. Therefore $A^{[\Lambda_{0}]}_{0}$ parametrizes the isomorphism classes of objects of $\CA_0^{[\Lambda_0]}$.

\begin{proposition} \label{iso_E}There is an isomorphism 
\begin{equation*}
\wt{\CA}_{\bK_{\wt G}, E}\times_{\Spec E} \Spec E^\sharp \simeq\coprod_{A_0^{[\Lambda_0]}(\bar\BQ)}  {\CA}_{\bK^{\star}_{G},  E_{r}}\times_{\Spec E_{r}} \Spec E^\sharp .
\end{equation*}
In other words, over $E^\sharp$, the scheme $\wt{\CA}_{\bK_{\wt G}, E}$ becomes a disjoint sum of copies of ${\CA}_{\bK^{\star}_{G},  E_{r}}$.
\end{proposition}
\begin{proof}
Let $({A_0}, {\iota_0}, {\lambda_0})\in \CA_0^{[\Lambda_{0}]}(O_{E^\sharp})$. We fix an isomorphism $\wh \RT(A_{0, E^\sharp})=\Lambda_0\otimes_{O_K}\wh{O}_K$. Since all elements of $\CA_0^{[\Lambda_{0}]}$ are isogenous, all rational Tate modules of elements of $\CA_0^{[\Lambda_{0}]}$ are identified with $V_0\otimes_K\BA_{K, f}$.  Therefore we obtain a morphism
\begin{equation}\label{morph_E}
\wt{\CA}_{\bK_{\wt G}, E}\times_{\Spec E} \Spec E^\sharp \to {\CA}_{\bK^{\star}_{G},  E_{r}}\times_{\Spec E_{r}} \Spec E^\sharp .
\end{equation}
It sends a tuple $(A_0,\iota_0,\lambda_0,A,\iota,\lambda,\bar\eta^p)$ to $(A,\iota,\lambda,\bar\eta'^p)$, where $\bar\eta'^p\colon \wh \RV^p(A)\isoarrow V\otimes_K \BA^p_{K, f}$ is induced from  
$$  \eta^p\colon \wh \RV^p(A_0, A)=\Hom_K(\wh \RV^p(A_0), \wh V^p(A)) \isoarrow \wt V\otimes_K\BA_{K, f}^p=\Hom_K( V_0,  V)\otimes_K\BA_{K, f}^p ,
$$ 
after identifying the source of the map with $\wh \RV^p(A)$ and the target with $V\otimes_K\BA_{K, f}^p$. Indeed, by hypothesis \eqref{inthyp}, the $\bK_G^{\star, p}$-class of $\eta'^p$ is well-determined by the $\bK_{\wt G}^p$-class of $\eta^p$. Furthermore, by the theory of complex multiplication, the Galois action of $\Gal(\bar \BQ/E_{\Phiplus})$ on $A_0^{[\Lambda_0]}(\bar\BQ)$ factors through its maximal abelian quotient, hence  the action of $\Gal(\bar \BQ/E^\sharp)$ on $\wh \RV^p({A}_{0, E^\sharp})$ is trivial. Hence the Galois invariance of $\bar\eta'^p$ follows from the Galois invariance of $\bar \eta^p$. Finally, the sign condition (iii) from Definition \ref{def:CAbK}  for  $(A, \iota, \bar\lambda, \bar\eta'^p)$ follows from the identities
\[
\inv^r_{v}(A_{0, \bar s},\iota_{0, \bar s},\lambda_{0, \bar s},A_{\bar s},\iota_{\bar s},\lambda_{\bar s})=\inv_{v}^r(A_{\bar s}, \iota_{\bar s}, \lambda_{\bar s}) ,\quad\quad \inv_{v}(\wt V) =\inv_{v}(V) .
\] 
Indeed, for instance, the last identity follows from the identity
\[
\wedge^2(V_0\otimes V)=V_0^{\otimes 2}\otimes \wedge^2(V) .
\] Since this construction is clearly functorial, we obtain the morphism \eqref{morph_E}. 
Letting $({A_0}, {\iota_0}, {\lambda_0})$ vary through  $A_0^{[\Lambda_0]}(\bar \BQ)$, we obtain the desired decomposition in the proposition.
\end{proof}

\subsection{At the level of moduli problems over $O_{\breve{E}_{r, \nu}}$}
Recall from \cite[Definition 7.4.5]{KRZI} the scheme $\CA^\star_{\bK^{\star}_G}$ over $\Spec O_{\breve E_{r,\nu}}$.  To define it, we first define a functor $\widehat{\mathcal{A}}^{\star}_{\mathbf{K}^{\star}_{G}}$ on the
  category of schemes $S$ over $\Spf O_{\breve{E}_{r, \nu}}$.
  A point of $\widehat{\mathcal{A}}^{\star}_{\mathbf{K}^{\star}_{G}}(S)$ consists of
  the following data:
\begin{enumerate}
\item A point $(A, \iota, \bar{\lambda}, \bar{\eta}^p)$ of $\mathcal{A}_{\mathbf{K}_{G}}(S)$;
\item A class $ \ov\eta^{v_{0}}_p$ of isomorphisms of $p$-adic \'etale sheaves  
\begin{equation*}
 \eta^{v_{0}}_p\colon \underline{\Hom}_{O_K}( \BX_0^{v_{0}}, A[p^\infty]^{v_{0}})\isoarrow \Lambda^{v_{0}}:=\prod\nolimits_{v\neq v_{0}, v\mid p}\Lambda_{v} \mod\bK_{G, p}^{\star, {v_{0}}} ,
\end{equation*}
which respect forms on both sides up to a constant in $\BZ^{\times}_{p}$, cf. \cite[just before Proposition  7.4.10]{KRZI}.
\end{enumerate}

The morphism $\widehat{\mathcal{A}}^{\star}_{\mathbf{K}^{\star}_{G}} \rightarrow \widehat{\mathcal{A}}_{\mathbf{K}_{G}}$ forgetting the datum $(2)$ (or its alternative version) 
is a finite \'etale covering of formal schemes. Since we assume that
$\mathbf{K}^p_{G}$ is small enough, $\mathcal{A}_{\mathbf{K}_{G}}$ is a proper scheme
over $\Spec O_{\breve{E}_{r,\nu}}$. By the algebraization theorem, there is a unique
finite \'etale morphism of schemes over $\Spec O_{\breve{E}_{\nu}}$
\begin{equation}\label{balevel2e}
\mathcal{A}^{\star}_{\mathbf{K}^{\star}_{G}} \rightarrow \mathcal{A}_{\mathbf{K}_{G}}\times_{\Spec\, O_{E_{r,\nu}}}\Spec\, O_{\breve E_{r, \nu}} , 
\end{equation}
such that the $p$-adic completion of $\mathcal{A}^{\star}_{\mathbf{K}^{\star}_{G}}$ is
$\hat{\mathcal{A}}^{\star}_{\mathbf{K}^{\star}_{G}}$.

\begin{proposition}\label{iso_OE} 
There is an isomorphism 
\begin{equation*}
\wt{\CA}_{\bK_{\wt G}}\times_{\Spec O_{E}} \Spec O_{\breve{E}_\nu} \simeq\coprod_{A_0^{[\Lambda_0]}(\bar {\kappa}_\nu)}  {\CA}^{\star}_{\bK^{\star}_{G}}\times_{\Spec O_{\breve E_{r, \nu}}} \Spec O_{\breve{E}_{\nu}}.
\end{equation*}
In other words, over $O_{\breve{E}_\nu}$, the scheme $\wt{\CA}_{\bK_{\wt G}}$ becomes a disjoint sum of copies of $\CA^\star_{\bK^{\star}_G}$.
\end{proposition}
\begin{proof}
Note that the index sets in Propositions \ref{iso_E} and \ref{iso_OE} can be identified, as  follows from the fact that $\CA_0^{[\Lambda_0]}$ is finite and \'etale over $\Spec O_{E_{\Phiplus}}$. Since both sides are proper over the base, we may pass to the formal completions on both sides. Note that $\BX_0^{v_{0}}$ can be identified with $A_0[p^\infty]^{v_{0}}$, cf. \eqref{primetov}. Hence $\eta_p^{v_{0}}$ defines a CL-level structure, and this extends the isomorphism of Proposition \ref{iso_E} to $O_{\breve{E}_\nu}$. 
\end{proof}

\part{Shimura curves for the unitary group}
\section{The integral LSV for the unitary group}\label{s:ILSV}

Let $F/\BQ_p$ be a finite extension. Let $K/F$ be a quadratic extension, and let $\U'$ be a $K/F$-unitary group of size $2$ over $F$ which is anisotropic. Let $\U=\Res_{F/\BQ_p}(\U')$. We also fix an embedding $\varphi_{0}: K\rightarrow \bar{\BQ}_{p}$. 

\subsection{Formulation of the result}\label{ss:formu}
Let $\Phi^{+}$ be a local CM-type for $K/F$ containing $\varphi_{0}$. We define a conjugacy class of cocharacters $\mu: \BG_{m, \bar{\BQ}_{p}}\rightarrow \U_{\bar{\BQ}_{p}}$ from $\Phi^{+}$, as follows. Namely,  we characterize  $\mu$ by the composite 
\begin{equation}\label{mu-special}
\BG_{m,\bar{\BQ}_{p}}\xrightarrow{\mu}\U_{\bar{\BQ}_{p}}\simeq \prod_{\varphi\in \Phi^{+}}\GL_{2, {\bar{\BQ}_{p}}}. 
\end{equation}
We demand that $\{\mu\}$  has component $(1,0)$ at $\varphi_{0}$ and $(0,0)$ at all other components (as conjugacy classes of cocharacters of $\GL_{2}$).  The reflex field $E(\mu)$ of $\mu$ is the subfield of $\bar\BQ_{p}$ fixed by 
$$\Gal(\bar\BQ_{p}/E(\mu)):=\{ \tau\in \Gal(\bar\BQ_{p}/\BQ_{p})\mid \tau\varphi_{0} = \varphi_{0}\}.$$
It  follows that  $E(\mu)$ is equal to  $\varphi_{0}(K)=K$. 

Let $b\in \U(\breve{\BQ}_{p})$ represent the unique element of ${B}(\U, \mu^{-1})$, which is basic. Let $\mathbf{K}^{\circ}_{\U}=\U(\BQ_{p})$ which is a maximal quasi-parahoric subgroup of $\U(\BQ_{p})$.  When $K/F$ is unramified, then $\mathbf{K}^{\circ}_{\U}$ is a parahoric; if $K/F$ is ramified, then $\mathbf{K}^{\circ}_{\U}$ contains a  parahoric with index $2$. By \cite{PRloc} we can associate to the local Shimura datum $(\U, b, \mu, \mathbf{K}^{\circ}_{\U})$, the integral local Shimura variety $\mathcal{M}^{\mathrm{int}}_{\mu, b, \mathbf{K}^{\circ}_{\U}}$. This is a formal scheme over $\Spf O_{\breve{E}(\mu)}$ with natural descent datum to $\Spf O_{E(\mu)}=\Spf O_{K}$.
 
We introduce the following finite extension of $K$ contained in $\bar\BQ_{p}$,
\begin{equation}\label{defE}
E(\varphi_0)=\bigcap\nolimits_{\varphi_0\in \Phi^{\prime, +}}E_{\Phi^{\prime, +}}K ,
\end{equation}
(intersection of the join with $K$ of the reflex fields of all CM-types $ \Phi^{\prime, +}$ of $K$ containing $\varphi_0$). We remark that when $F=\BQ_{p}$, then $E(\varphi_{0})$ agrees with $K$. 

The following is the main result of this section. The proof will be given in the next subsection.

\begin{theorem}\label{local-Shi-Drin}
Let $\mathcal{M}^{\mathrm{int}}_{\mu, b, \mathbf{K}^{\circ}_{\U}}$ be the integral local Shimura variety associated to the datum $(\U, b, \mu, \mathbf{K}^{\circ}_{\U})$ defined above. Then the Weil descent datum on $\mathcal{M}^{\mathrm{int}}_{\mu, b, \mathbf{K}^{\circ}_{\U}}$  is effective and defines a formal scheme $\mathcal{M}^{\mathrm{int}}_{\mu, b, \mathbf{K}^{\circ}_{\U}, O_{K}}$ over $\Spf O_{K}$. There is an isomorphism of  formal schemes over $\Spf O_{E(\varphi_{0})}$,
\begin{equation*}
\mathcal{M}^{\mathrm{int}}_{\mu, b, \mathbf{K}^{\circ}_{\U}, O_{K}}\times_{\Spf O_{K}} \Spf O_{E(\varphi_{0})}\simeq \hat{\Omega}_{F} \times_{\Spf O_{F}} \Spf O_{E(\varphi_{0})}.
\end{equation*}

\end{theorem}

\begin{remark}
We recall from \S \ref{ss:staU}, lines below Theorem \ref{Thm;locU},  that we conjecture that this isomorphism is induced by an isomorphism of formal schemes over $\Spf O_{K}$. 
When $F=\BQ_{p}$, then $E(\varphi_{0})=K$, and the above isomorphism is simply 
\begin{equation*}
\mathcal{M}^{\mathrm{int}}_{\mu, b, \mathbf{K}^{\circ}_{\U}, O_{K}}\simeq \hat{\Omega}_{F} \times_{\Spf O_{F}} \Spf O_{K}.
\end{equation*}
In this case, the expectation is true. 
\end{remark}

\subsection{Proof of Theorem \ref{local-Shi-Drin}} \label{proofILSV}
To prove the main theorem stated in the last subsection, we introduce two more integral local Shimura varieties.

Let $\Phi_0^+$ be a   local CM-type  with $\varphi_0\in\Phi_0^+$.  We define the torus $\mathrm{T}$ over $\BQ_p$ with 
$$\mathrm{T}(\BQ_p)=\{t\in K^{\times}: \Nm_{K/F}(t)\in\BQ^{\times}_{p}\} .$$ 
We consider the cocharacter $\mu_{0}: \BG_{m, \bar{\BQ}_{p}}\rightarrow \mathrm{T}_{\bar{\BQ}_{p}}$ obtained from the local CM-type $\Phi_0^{+}$. Namely, $\mu_{0}$ is defined so that the composite 
\[\bar{\BQ}^{\times}_{p}\overset{\mu_{0}}\longrightarrow \T(\bar{\BQ}_{p})\subset \prod\nolimits_{\varphi\in\Phi_0^{+}}(\bar{\BQ}^{\times}_{p})_{\varphi}
\] is given by the diagonal embedding. Let $E_{0}=E_{\Phi_0^+}$ be the reflex field of $\Phi_0^+$. Let $b_{0}\in \T(\breve{\BQ}_{p})$ represent the unique element $[b_{0}]$ of ${B}(\T, \mu_{0}^{-1})$. Let $\mathbf{K}^{\circ}_{\T}$ be the maximal compact open subgroup of $\T(\BQ_{p})$.  By  \cite{PRloc} we can associate  to the local Shimura datum $(\T, b_{0}, \mu_{0}, \mathbf{K}^{\circ}_{\T})$  the integral local Shimura variety $\mathcal{M}^{\mathrm{int}}_{\mu_{0}, b_{0}, \mathbf{K}^{\circ}_{\T}}$. This is a formal scheme over $\Spf O_{\breve{E}_{{0}}}$ with natural descent datum down to $\Spf O_{E_{{0}}}$. 

Let $\wt{G}=\U\times \T$. We obtain  the cocharacter 
\[\wt{\mu}=  \mu\times \mu_{0} \colon \BG_{m, , \bar{\BQ}_{p}}\longrightarrow \wt{G}_{\bar{\BQ}_{p}}.
\] Also, let $\mathbf{K}^{\circ}_{\wt{G}}= \bfK^{\circ}_{\U}\times \bfK^{\circ}_{\T}$. Then $\wt{b}=(b, b_{0})\in \wt{G}(\breve\BQ_p)= \U(\breve\BQ_p)\times \T(\breve\BQ_p)$ is a representative of  the unique basic element  in ${B}(\wt{G}, \wt{\mu}^{-1})$.  By  \cite{PRloc}, we can associate  to the local Shimura datum $(\wt{G}, \wt{b}, \wt{\mu}, \mathbf{K}^{\circ}_{\wt{G}})$ the integral local Shimura variety $\mathcal{M}^{\mathrm{int}}_{\wt{\mu}, \wt{b}, \mathbf{K}^{\circ}_{\wt{G}}}$, a formal scheme over $\Spf O_{\breve{E}}$ with natural descent datum to $\Spf O_{E}$. Here  $E=E_{0}K$ is the reflex field for this local Shimura datum.

The integral local Shimura varieties  $\mathcal{M}^{\mathrm{int}}_{\mu_{0}, b_{0}, \mathbf{K}^{\circ}_{\T}}$ and $\mathcal{M}^{\mathrm{int}}_{\wt{\mu}, \wt{b}, \mathbf{K}^{\circ}_{\wt{G}}}$ are  of PEL-type and therefore they are isomorphic to the corresponding RZ spaces, cf.  \cite[lecture 25]{Sch}\footnote{Note that in loc.~cit., the Weil descent datum is ignored.}. Let us explain this for $\mathcal{M}^{\mathrm{int}}_{\wt{\mu}, \wt{b}, \mathbf{K}^{\circ}_{\wt{G}}}$. We fix $(V, \varsigma)$ as in \eqref{varsigma-loc} such that $\U'=\U(V)$. Let $r=r_{\Phi^+_0}$ be the generalized CM-type of rank $2$, special wrt ${\varphi_0}_{| F}$,  defined by 
\[
r_{\varphi}=2, \quad \varphi\in \Phi^+_0\setminus \{\varphi_0\}. 
\] 
Then $\mathcal{M}^{\mathrm{int}}_{\wt{\mu}, \wt{b}, \mathbf{K}^{\circ}_{\wt{G}}}$ is isomorphic to $\wt{\CM}_{\bK^{\circ}_{\wt G}}$ introduced in Definition \ref{RSZ-LS-special} for the given $\Phi^+_0$ and $r$. Similarly,  
 $\mathcal{M}^{\mathrm{int}}_{\mu_{0}, b_{0}, \mathbf{K}^{\circ}_{\T}}$ is isomorphic to ${\CM}_{\bK^{\circ}_{\T}}$ as in Definition \ref{RSZ-LS-Tori}.   Note that $E=E_{0}K=E_{r}K$,   where $E_{r}$ is the reflex field of the local CM-type $r$.

\begin{proof}[Proof of Theorem \ref{local-Shi-Drin}]
The functoriality of  integral local Shimura varieties \cite[\S 2.4]{PRloc} furnishes the following product decomposition
\begin{equation}\label{int-prod1}
  {\mathcal{M}}^{\mathrm{int}}_{\wt \mu, \wt b, \mathbf{K}^{\circ}_{\wt G}}\simeq\mathcal{M}^{\mathrm{int}}_{\mu, b, \mathbf{K}^{\circ}_{\U}, O_{\breve{E}}}\times_{\Spf  O_{\breve{E}}} \mathcal{M}^{\mathrm{int}}_{\mu_{0}, b_{0}, \mathbf{K}^{\circ}_{\T}, O_{\breve{E}}},
\end{equation}  
with compatible Weil descent datum down to $O_E$ on both sides. Here, to simplify the notation,  we have indicated the base changes  by the index $O_{\breve E}$ (from $O_{\breve K}$ to $ O_{ \breve E}$ for the first factor, resp. from $O_{\breve E_{0}}$ to $ O_{\breve E}$ for the second factor). 

We have an explicit expression for the LHS of \eqref{int-prod1}. Indeed, we have ${\mathcal{M}}^{\mathrm{int}}_{\wt{\mu}, \wt{b}, \mathbf{K}^{\circ}_{\wt G}}\simeq \wt{\CM}_{\bK^{\circ}_{\wt G}} $. Hence by Proposition \ref{LS-weil-special}, we have 
\begin{equation}\label{Mint}
{\mathcal{M}}^{\mathrm{int}}_{\wt{\mu}, \wt{b}, \mathbf{K}^{\circ}_{\wt G}}\simeq (\hat{\Omega}_{F} \times_{\Spf O_{F}} \Spf O_{\breve{E}}) \times
\T(\mathbb{Q}_p)/\mathbf{K}^{\circ}_{\T}, 
\end{equation}
where we used that $\wt{G}(\mathbb{Q}_p)= \U(\BQ_{p})\times \T(\BQ_{p})$ and  $\mathbf{K}^{\circ}_{\wt G}=\U(\BQ_{p})\times \mathbf{K}^{\circ}_{\T}$.  

We also have an explicit expression for the second factor on the RHS of \eqref{int-prod1}. Indeed, $\mathcal{M}^{\mathrm{int}}_{\mu_{0}, b_{0}, \mathbf{K}^{\circ}_{\T}}\simeq {\mathcal{M}}_{\mathbf{K}^{\circ}_{\T}}$ so that, by Proposition \ref{LS-tori-weil}, we have 
\begin{equation}\label{MTint}
\mathcal{M}^{\mathrm{int}}_{\mu_{0}, b_{0}, \mathbf{K}^{\circ}_{\T}}\simeq \T(\BQ_p)/\bK^{\circ}_{\T} .
\end{equation}
Note that we have a commutative diagram of formal schemes over $\Spf O_{\breve{E}}$,
\begin{equation}\label{commutative-local}
\begin{tikzcd}
{\mathcal{M}}^{\mathrm{int}}_{\wt \mu, \wt b, \mathbf{K}^{\circ}_{\wt G}} \arrow[r, "\sim"] \arrow[d] &  \wt{\CM}_{\bK^{\circ}_{\wt G}}  \arrow[d] \\
\mathcal{M}^{\mathrm{int}}_{\mu_{0}, b_{0}, \mathbf{K}^{\circ}_{\T}, O_{\breve{E}}} \arrow[r, "\sim"] &  {\mathcal{M}}_{\mathbf{K}^{\circ}_{\T}, O_{\breve{E}}}.    
\end{tikzcd}
\end{equation}
Here the vertical map on the left is the projection morphism from \eqref{int-prod1}. The vertical map on the right comes from the moduli description of $\wt{\CM}_{\bK^{\circ}_{\wt G}}$, cf. \eqref{vartheta-special}.
Comparing the individual factors of \eqref{int-prod1} with \eqref{Mint} and using the fact that 
\begin{equation}\label{prodILSV}
 \mathcal{M}^{\mathrm{int}}_{\mu_{0}, b_{0}, \mathbf{K}^{\circ}_{\T}, O_{\breve{E}}}\simeq{\mathcal{M}}_{\mathbf{K}^{\circ}_{\T}, O_{\breve{E}}}\simeq \T(\BQ_{p})/\bK^{\circ}_{\T},
\end{equation}
it follows from considering the vertical fibers of the commutative diagram \eqref{commutative-local} that we have an isomorphism
\begin{equation}\label{MU-and-Omega}
\mathcal{M}^{\mathrm{int}}_{\mu, b, \mathbf{K}^{\circ}_{\U}, O_{\breve{E}}}\simeq \hat{\Omega}_{F} \times_{\Spf O_{F}} \Spf O_{\breve{E}}  .
\end{equation}
 By the equivariance of the vertical projections for the action of $\T(\BQ_{p})$ and the transitivity of the action on the target, this isomorphism is independent of the fiber.

Moreover, this isomorphism is compatible with Weil descent data down to $O_E$  on both sides. To see this, we use the compatibility of the Weil descent data on the two sides of the isomorphism  \eqref{int-prod1}
\begin{equation*}
\mathcal{M}^{\mathrm{int}}_{\mu, b, \mathbf{K}^{\circ}_{\U}, O_{\breve{E}}}\times_{\Spf  O_{\breve{E}}} \mathcal{M}^{\mathrm{int}}_{\mu_{0}, b_{0}, \mathbf{K}^{\circ}_{\T}, O_{\breve{E}}}\overset{\sim}{\longrightarrow} \hat{\Omega}_{F} \times_{\Spf O_{F}} \Spf O_{\breve{E}}\times \T(\BQ_p)/\bK^{\circ}_{\T} .
\end{equation*}
It follows that we have a commutative diagram
\begin{equation*}
\begin{tikzcd}
\mathcal{M}^{\mathrm{int}}_{\mu, b, \mathbf{K}^{\circ}_{\U}, O_{\breve{E}}}\times_{\Spf  O_{\breve{E}}} \mathcal{M}^{\mathrm{int}}_{\mu_{0}, b_{0}, \mathbf{K}^{\circ}_{\T}, O_{\breve{E}}} \arrow[r, "\sim"] \arrow[d, "\alpha_{1}"] & (\hat{\Omega}_{F} \times_{\Spf O_{F}} \Spf O_{\breve{E}}) \times
\T(\mathbb{Q}_p)/\mathbf{K}^{\circ}_{\T} \arrow[d, "\alpha_{2}"] \\
\big(\mathcal{M}^{\mathrm{int}}_{\mu, b, \mathbf{K}^{\circ}_{\U}, O_{\breve{E}}}\times_{\Spf  O_{\breve{E}}} \mathcal{M}^{\mathrm{int}}_{\mu_{0}, b_{0}, \mathbf{K}^{\circ}_{\T}, O_{\breve{E}}} \arrow[r, "\sim"]\big)^{(\tau_E)} &   \big((\hat{\Omega}_{F} \times_{\Spf O_{F}} \Spf O_{\breve{E}}) \times
\T(\mathbb{Q}_p)/\mathbf{K}^{\circ}_{\T}\big)^{(\tau_E)},         
\end{tikzcd}
\end{equation*}
where $\alpha_{1}$ is given by $(\xi_{1}, \xi_{2})\mapsto (\omega_{\U, E}(\xi_{1}), \omega_{\T, E}(\xi_{2}))$, with $\omega_{\U, E}$ and $\omega_{\T, E}$  the natural Weil descent data for $\mathcal{M}^{\mathrm{int}}_{\mu, b, \mathbf{K}^{\circ}_{\U}, O_{\breve{E}}}$ and $\mathcal{M}^{\mathrm{int}}_{\mu_{0}, b_{0}, \mathbf{K}^{\circ}_{\T}, O_{\breve{E}}}$  down to $O_E$, and where $\alpha_{2}$ is given by $(\xi, t)\mapsto (\omega_{\Omega, E}(\xi), \wt{w}t)$, with $\omega_{\Omega, E}$  the natural Weil descent datum for $\hat{\Omega}_{F} \times_{\Spf O_{F}} \Spf O_{\breve{E}}$ and with $\wt w$ as in the statement of Proposition \ref{LS-weil-special}.   Now  the Weil descent datum $\omega_{\T, E}$  is given in terms of the identification \eqref{MTint} by  $t\mapsto {w_{r_0}}^{f_{E}/f_{E_0}}t$, where $w_{r_0}$ is   the special element from Proposition \ref{LS-tori-weil}. But $\wt w$ and ${w_{r_0}}^{f_{E}/f_{E_0}}$ differ by a unit (both have $K/F$-norm $p^{f_E}$, up to a unit), hence have the same residue class modulo $\T(\BZ_p)$. Since the isomorphism \eqref{MU-and-Omega} is  independent of the fiber, it follows that \eqref{MU-and-Omega}  is compatible with the Weil descent data, as claimed. 

It follows that the descent datum on $\mathcal{M}^{\mathrm{int}}_{\mu, b, \mathbf{K}^{\circ}_{\U}, O_{E}}$ down to $O_E$ is effective, defining the descended formal scheme  $\mathcal{M}^{\mathrm{int}}_{\mu, b, \mathbf{K}^{\circ}_{\U}, O_{E}}$ over $\Spf O_E$, and  that the isomorphism \eqref{MU-and-Omega} is induced by an  isomorphism of formal schemes over $\Spf O_E$
\begin{equation}\label{weil-U-E}
\mathcal{M}^{\mathrm{int}}_{\mu, b, \mathbf{K}^{\circ}_{\U}, O_{E}}\simeq \hat{\Omega}_{F} \times_{\Spf O_{F}} \Spf O_{E}.
\end{equation}

Now we let $\Phi^+_0$ vary. We claim that we can descend the  isomorphism \eqref{weil-U-E} to the smaller field $E(\varphi_{0})$. To prove this, it suffices to consider two such local CM-types, say $\Phi^+_0$ and $\Phi^{+, \prime}_0$. Then we obtain two local reflex fields $E=E_{0}K$ and $E^{\prime}=E_{0}^\prime K$. Let  $L$ be the composite of $E$ and $E^{\prime}$. To establish our claim, it suffices to show that the isomorphisms  from \eqref{weil-U-E}
\begin{equation*}
\mathcal{M}^{\mathrm{int}}_{\mu, b, \mathbf{K}^{\circ}_{\U}, O_{E}}\simeq \hat{\Omega}_{F} \times_{\Spf O_{F}} \Spf O_{E}, \quad\text{resp.}\quad \mathcal{M}^{\mathrm{int}}_{\mu, b, \mathbf{K}^{\circ}_{\U}, O_{E^{\prime}}}\simeq \hat{\Omega}_{F} \times_{\Spf O_{F}} \Spf O_{E^{\prime}},
\end{equation*} 
give rise via base-change to $O_{L}$ to the same isomorphism 
\begin{equation*}
\mathcal{M}^{\mathrm{int}}_{\mu, b, \mathbf{K}^{\circ}_{\U}, O_{L}}\simeq \hat{\Omega}_{F} \times_{\Spf O_{F}} \Spf O_{L} .
\end{equation*} 
  By our construction, this boils  down to showing that the Weil descent datum on  both sides of 
\begin{equation}\label{weil-E}
\begin{aligned}
&\mathcal{M}^{\mathrm{int}}_{\wt{\mu}, \wt{b}, \mathbf{K}^{\circ}_{\wt{G}}, O_{\breve{E}}}\simeq (\hat{\Omega}_{F} \times_{\Spf O_{F}} \Spf O_{\breve{E}})\times_{\Spf  O_{\breve{E}}}  (\T(\BQ_p)/\bK^{\circ}_{\T} )_{O_{\breve{E}}},\\ \quad\text{resp.}\\
&\mathcal{M}^{\mathrm{int}}_{\wt{\mu}, \wt{b}, \mathbf{K}^{\circ}_{\wt{G}}, O_{\breve{E}^{\prime}}}\simeq (\hat{\Omega}_{F} \times_{\Spf O_{F}} \Spf O_{\breve{E}^{\prime}})\times_{\Spf  O_{\breve{E}^{\prime}}} (\T(\BQ_p)/\bK^{\circ}_{\T} )_{O_{\breve{E}^{\prime}}}
\end{aligned}
\end{equation} 
induce  the same Weil descent datum on  both sides of 
\begin{equation}\label{weil-EL}
\mathcal{M}^{\mathrm{int}}_{\wt{\mu}, \wt{b}, \mathbf{K}^{\circ}_{\wt{G}}, O_{\breve{L}}}\simeq (\hat{\Omega}_{F} \times_{\Spf O_{F}} \Spf O_{\breve{L}})\times_{\Spf  O_{\breve{L}}} (\T(\BQ_p)/\bK^{\circ}_{\T} )_{O_{\breve{L}}}.
\end{equation} 
On the left-hand side, this is clear.  By Proposition \ref{LS-weil-special}, the Weil descent datum down to $O_E$ in the first line on the right-hand side of \eqref{weil-E} is given by 
\begin{equation*}
(\xi, g) \mapsto (\omega_{\Omega,{E}}(\xi), \wt{w}_{E}g), \quad g \in \T(\mathbb{Q}_p) .
\end{equation*} 
 Here we recall $\wt{w}_{E}={\mathfrak{r}}_{\Phi_0^{+}, E}(\varpi_{E})$ is the special element relative to $(\Phi_0^{+}, E)$ for a uniformizer $\varpi_{E}$ of $E$ as in Construction \ref{sp-elt-loc}.  A similar  formula applies to the second line of \eqref{weil-E} and makes explicit the Weil descent datum down to  $O_{{E}^{\prime}}$, where we use the special element  $\wt{w}_{E^\prime}=\mathfrak{r}_{\Phi^{+,\prime}_{0}, E^{\prime}}(\varpi_{E^{\prime}})$ relative to $(\Phi_0^{+, \prime}, E^{\prime})$ for a uniformizer $\varpi_{E^{\prime}}$ of $E$. Note that the Weil descent datum depends only on the image of $\wt{w}_{\ast}$ in $\T(\mathbb{Q}_p)/\bK^{\circ}_{\T}$ for $\ast\in\{E, E^{\prime}\}$, which means we can always modify the element $\wt{w}_{\ast}$ by a unit in $O_{K}$.
 
 Let $f_L$ be the inertia degree of $L$. Then the descent datum induced on the RHS of \eqref{weil-EL} by the first line of \eqref{weil-E} is given by $(\xi, g) \mapsto (\omega_{\Omega, {L}}(\xi), \wt{w}_{E}^{f_L/f_E}g)$, where $f_E$ denotes the inertia degree of $E$. Similarly,  the descent datum induced on the RHS of \eqref{weil-EL} by the second line of \eqref{weil-E} is given by $(\xi, g) \mapsto (\omega_{\Omega, {L}}(\xi), \wt{w}_{E^\prime}^{f_L/f_{E^\prime}}g)$, where $f_{E^\prime}$ denotes the inertia degree of $E^\prime$. It suffices to show that $\wt{w}_{E}^{f_L/f_{E}}$ and $\wt{w}_{E^\prime}^{f_L/f_{E^\prime}}$ differ by a unit. This follows from the identities 
  $\Nm_{K/F}(\wt{w}_{E})=p^{f_{E}}$ and $\Nm_{K/F}(\wt{w}_{E^\prime})=p^{f_{E^\prime}}$ (up to units) which imply 
\begin{equation*}
\Nm_{K/F}(\wt{w}^{f_{L}/f_{E}}_{E})=p^{ f_{L}}=\Nm_{K/F}(\wt{w}^{f_{L}/f_{E^{\prime}}}_{E^{\prime}})
\end{equation*}
(up to units).

 At this point, we have proved the /effectivity of the Weil descent datum on $\mathcal{M}^{\mathrm{int}}_{\mu, b, \mathbf{K}^{\circ}_{\U}}$ down to the finite extension $O_{E(\varphi_{0})}$ of $O_K$, which then induces a usual descent datum on $\mathcal{M}^{\mathrm{int}}_{\mu, b, \mathbf{K}^{\circ}_{\U}, O_{E(\varphi_{0})}}$ for the finite extension $O_{E(\varphi_{0})}$ of $O_K$. This descent datum is in fact effective. Indeed, this follows from the fact that the canonical bundle on the special fiber of $\mathcal{M}^{\mathrm{int}}_{\mu, b, \mathbf{K}^{\circ}_{\U}, O_{E(\varphi_{0})}}$ is ``ample'' (i.e., the restriction to any finite type closed subscheme is ample). This latter fact is well-known for $\hat{\Omega}_{F} $ and follows from the fact that the quotient of  $\hat{\Omega}_{F} $ by a sufficiently small discrete cocompact subgroup of $\PGL_2(F)$ is a stable curve over $O_F$ in the sense of Deligne-Mumford, comp. the argument in the proof of \cite[Theorem  7.3.3]{KRZI}. We obtain the formal scheme $\mathcal{M}^{\mathrm{int}}_{\mu, b, \mathbf{K}^{\circ}_{\U}, O_{K}}$ with the property in the statement of Theorem \ref{local-Shi-Drin}.
 \end{proof}

\section{$p$-adic uniformization of Shimura curves  for the unitary group}\label{s:unigl}

Now we consider the global case.  Let $K/F$ be a CM-field, with a fixed embedding $\varphi_0\colon K\to\bar\BQ$. We denote by $w_0$ the  archimedean place of $F$ induced from the embedding $F\to\bar\BQ$. Let $p$ be a prime number and fix an embedding $\bar\BQ\to\bar\BQ_p$.  We denote by $v_0$ the $p$-adic place induced by $F\to\bar\BQ_p$. We assume that $v_0$ does not split in $K$. We will identify $F$ and $K$ with their images via $\varphi_{0}$ in $\bar\BQ$. . 

Let $\U'$ be a $K/F$-unitary group of size $2$ and set $\U=\Res_{F/\BQ}(\U')$. We assume that $\U'_{v_0}$ is anisotropic. Consider the Shimura datum ${X}_{\U}$ given by  the conjugacy class of homomorphisms
\begin{equation}\label{shi-da}
h_{\U}\colon \Res_{\BC/\BR}(\BG_m)\longrightarrow\U_{\BR} 
\end{equation}
that we define now. 
 Fix a CM type $\Phi^+$ for $K/F$ with $\varphi_0\in\Phi^+$. We choose for each $\varphi\in \Phi^+$ a $\BC$-basis of $V\otimes_{K, \varphi}\BC$  such that the anti-hermitian form is given by $\diag(i\cdot 1_{r_\varphi}, (-i)\cdot 1_{r_{\bar\varphi}})$. Then  we can write 
\begin{equation*}\label{shi-daprod}
\U\otimes_\BQ\bar\BQ=\prod\nolimits_{\Phi^+} \GL_2/\bar\BQ . 
\end{equation*}
Correspondingly $h_{\U}=(h_\varphi)_{\varphi\in\Phi^+}$. We set
\begin{equation}\label{shi-daprod}
h_\varphi(z)=
 \begin{cases}
    \begin{aligned}
      \diag(1, z/\bar z),  &\quad \varphi=\varphi_0 \\
      1  , &\quad \varphi\neq\varphi_0.\\
    \end{aligned}
  \end{cases}
\end{equation}

Then the reflex field $E(\U, X_{\U})$ of $(\U, X_{\U})$ is given by the fixed field of
$$\Gal(\bar\BQ/E(\U, X_{\U}))=\{ \tau\in \Gal(\bar\BQ/\BQ)\mid \tau\varphi_{0} = \varphi_{0}\}.$$
Therefore $E(\U, X_{\U})=\varphi_{0}(K)=K$.

For an open compact subgroup $\bK_{\U} \subset \U(\BA_f)$, there is a Shimura variety ${\rm Sh}_{\bK_{\U}}(\U, X_{\U})$ over $\Spec K$, 
whose complex points are given by 
\begin{equation*}
{\rm Sh}_{\bK_{\U}}(\U, X_{\U})(\BC) \simeq \U(\BQ) \bs [\Omega_\BR \times \U(\BA_f)/{\bK_{\U}}].
\end{equation*}
Note this Shimura variety is of abelian type but not of Hodge type. We denote by ${\rm Sh}_{\bK_{\U}}(\U, X_{\U})_{K_{v_0}}$ the base change over $K_{v_0}$, and analogously for any extension of $K_{v_0}$.

 The goal of this section is to establish  $p$-adic uniformization for this Shimura curve  over an extension of $K_{v_0}$.
\subsection{Construction of integral models}

We take  the open compact subgroup $\bK_\U$ of $\U(\BA_{f})$ of the form ${\bK}_\U = {\bK}_{\U, p} \cdot {\bK}^{p}_{\U}$, where ${\bK}^{p}_{\U}\subset \U(\BA_f^p)$ is sufficiently small. Let us write $\U_p=\U\otimes_\BQ\BQ_p$ and $\U_p(\BQ_{p})=\prod_{v\mid p}\U_{v}(\BQ_{p})$, where we denote by $\U_v={\rm Res}_{F_v/\BQ_p}(\U'_v)$ the corresponding unitary group over $\BQ_{p}$. Then we further impose that 
 \begin{equation}\label{condoutv0}
{\bK}_{\U,p} = {\bK}_{\U, v_0} \cdot {\bK}^{v_0}_{\U, p},
\end{equation}
where ${\bK}_{\U, v_0} = \U_{v_{0}}(\BQ_{p})$, and where ${\bK}^{v_0}_{\U,p}\subset \prod_{v\neq v_0}\U_v(\BQ_p)$ is arbitrary.  Note that this class of open compact subgroups is stable under conjugation by $\U(\BQ_p)$. 
Note that ${\bK}_{\U, v_0}$ is a quasi-parahoric subgroup of $\U_{v_0}(\BQ_p)$ (recall that $\U_{v_0}(\BQ_p)$ is compact by assumption). 

We first consider the case where  
\begin{equation}\label{parb}
{\bK}^{\circ}_{\U, p}={\bK}^{}_{\U, v_0} \cdot \prod_{v\neq v_0}{\bK}^{\circ}_{\U, v} .
\end{equation}
Here for $v\neq v_{0}$ we impose that ${\bK}^{\circ}_{\U, v}$ is a maximal parahoric subgroup of $\U_{v}(\BQ_{p})$ which is contained in the stabilizer of an almost selfdual lattice $\Lambda_v$ in $V_v$, cf. the Notation section in the Introduction. Then ${\bK}^{\circ}_{\U, p}$ is a quasi-parahoric subgroup of $\U(\BQ_p)$. It is even a parahoric subgroup if $v_0$ is unramified in $K$ since then ${\bK}^{}_{\U, v_0}$ is a parahoric. 

We will assume the existence of the \emph{canonical integral model} $\CA_{\bK_\U^{\circ}}$  of ${\rm Sh}_{\bK^{\circ}_{\U}}(\U, X_{\U})$ over $O_{K_{v_0}}$. This is a normal scheme proper and flat over $\Spec O_{K_{v_0}}$ which admits a unique characterization, cf.  \cite[Conjecture 4.2.2]{PR}. When $\bK^{\circ}_{\U, p}$ is a parahoric and $p\neq 2$, the canonical integral model exists, cf.  \cite{DY}. When  $\bK^{\circ}_{\U, p}$ is a quasi-parahoric or $p=2$, the existence of the canonical integral model seems still unknown.  Compare also with Remark \ref{cani} below.
\begin{proposition}
There exists a unique extension $\CA_{\bK_{\U}}$ of ${\rm Sh}_{\bK_{\U}}(\U, X_{\U})_{K_{v_0}}$ as a normal flat $O_{K_{v_0}}$-scheme, for every $\bK_{\U}$ of the form \eqref{condoutv0}, compatible with the transition morphisms for varying $\bK_{\U}$,  such that the following two conditions are satisfied.
\begin{enumerate}
\item For $\bK_{\U}=\bK^{\circ}_{\U}$, as in \eqref{parb}, the model coincides with the canonical integral model.
\item The transition morphisms $\CA_{\bK'_{\U}}\to \CA_{\bK_{\U}}$ for $\bK'_{\U}\subset\bK_{\U}$ are finite \'etale. 
\end{enumerate}
\end{proposition}
\begin{proof} The uniqueness is easy: for $\bK_{\U}\subset\bK^{\circ}_{\U}$, we see that $\CA_{\bK_{\U}}$ is the normalization of $\CA_{\bK^{\circ}_{\U}}$ in ${\rm Sh}_{\bK_{\U}}(\U, X_{\U})_{K_{v_0}}$. This defines $\CA_{\bK_{\U}}$ for all sufficiently small ${\bK_{\U}}$. If $\bK'_{\U}\subset\bK_{\U}$ is a normal subgroup, then  the action of $\bK_\U/\bK'_\U$ on ${\rm Sh}_{\bK'_{\U}}(\U, X_{\U})_{K_{v_0}}$ extends to $\CA_{\bK'_{\U}}$ and $\CA_{\bK_{\U}}=\CA_{\bK'_{\U}}/ (\bK_{\U}/\bK'_{\U})$. This extends the definition of $\CA_{\bK_{\U}}$ to all ${\bK_{\U}}$. The same argument also handles the existence. It remains to show that the transition morphisms are finite \'etale. For this, we introduce two more Shimura varieties which are of PEL type. 

Recall the    CM-type  $\Phi^+$ we used to define the Shimura datum $(\U, X_\U)$.  In analogy to the local case (where we constructed the pair $(\T, \mu_0)$), we obtain a Shimura datum $(Z^\BQ, X_{Z^\BQ})$. Let $\wt G=\U\times Z^\BQ$. We obtain the Shimura datum $(\wt G, X_{\wt G})$, where $h_{\wt{G}}= h_{\U}\times h_{Z^{\BQ}}$. The reflex field for $(\wt G, X_{\wt G})$ is equal to $E=E_{\Phi^+}K$. For the corresponding Shimura varieties we have the following isomorphism over $\Spec E$, 
\begin{equation}\label{decompgl}
{\rm Sh}_{\bK_{\wt G}}(\wt G, X_{\wt G})\simeq  {\rm Sh}_{\bK_{\U}}(\U, X_{\U})_{E} \times_{\Spec E}  {\rm Sh}_{\bK^{\circ}_{Z^\BQ}}(Z^\BQ, X_{Z^\BQ})_{E},
\end{equation}
where we define 
\begin{equation}\label{decprodK}
\bK_{\wt G}= \bK_\U\times \bK^{\circ}_{Z^\BQ}.
\end{equation} 
 We can identify the Shimura variety ${\rm Sh}_{\bK_{\wt G}}(\wt G, X_{\wt G})$ with the PEL type Shimura variety of Definition \ref{def:overE}. To this end, let $(V, \varsigma)$ as in \eqref{varsigma} such that $\U'$ is equal the isometry group $\U(V, \varsigma)$.  Let $r$ be the CM-type of rank $2$, special with respect to $\varphi_{0}$, such that its canonical CM-type is given by $\Phi^{+}$. Recall that this means $r_{\varphi_{0}}=r_{\bar{\varphi}_{0}}=1$ and  $r_\varphi=2$ for $\varphi\in\Phiplus\setminus \{\varphi_0\}$. Then  $E=E_{r}K$. We obtain an isomorphism between ${\rm Sh}_{\bK_{\wt G}}(\wt G, X_{\wt G})$
with the Shimura variety $\wt\CA_{\bK_{\wt G}, E}$ of  Definition \ref{def:overE}.

The decomposition \eqref{decompgl} 
extends to canonical integral models over $O_{E_\nu}$,
\begin{equation}\label{int-decomp-0mo}
\wt\CA_{\bK^{\circ}_{\wt G}, O_{E_\nu}}= \CA_{\bK^{\circ}_{\U}, O_{E_\nu}}\times_{\Spec O_{E_\nu}} \CA^{[\Lambda_{0}]}_{0, O_{E_\nu}},
\end{equation}
where $\bK^\circ_{\wt G}= \bK^\circ_\U\times \bK^{\circ}_{Z^\BQ}$. Here we denoted by an index the base changes  $O_{K}\to O_{E_\nu}$, resp. $O_{E_{\Phi^+}}\to O_{E_\nu}$, on the first, resp. second, factor on the RHS. 

\begin{remark}\label{cani}
Let us clarify.  We are assuming the existence of the canonical integral model for  ${\rm Sh}_{\bK_{\U}}(\U, X_{\U})$ over $O_{K_{v_0}}$. Since we are assuming that for $v\neq v_0$, the open compact $\bK_{\U,v}^\circ$ is a parahoric stabilizer of $\Lambda_v$, it follows that $\bK_{\U, p}^\circ$ is a parahoric iff $\bK_{\U,v_0}^\circ$ is a parahoric, which is true iff $v_0$ is unramified in $K$. The existence of the canonical integral model for  ${\rm Sh}_{\bK_{\U}}(\U, X_{\U})$ over $O_{K_{v_0}}$ is known when $\bK_{\U, p}^\circ$ is a parahoric and $p\neq 2$, cf. \cite{DY}.  Thus our result is conditional iff $v_0$ is ramified in $K$ or $p=2$.

On the other hand,  the existence of the canonical integral models of the PEL-type Shimura varieties ${\rm Sh}_{\bK_{\wt G}}(\wt G, X_{\wt G})$ over $O_{E_\nu}$ and of  ${\rm Sh}_{\bK^{\circ}_{Z^\BQ}}(Z^\BQ, X_{Z^\BQ})$ over $O_{E_{\Phi^+, \nu}}$ follows from   \cite[Theorem  4.2.3]{DvHKM} (there is no assumption on the quasi-parahoric $\bK_p$ in loc. cit., nor is it assumed that $p\neq 2$).  Furthermore, the canonical integral model of ${\rm Sh}_{\bK_{\wt G}}(\wt G, X_{\wt G})$  is equal to the $p$-integral model $\wt\CA_{\bK^{\circ}_{\wt G, O_{E_\nu}}}$ of Definition \ref{def:overOE}, and, similarly, the canonical integral model of ${\rm Sh}_{\bK^{\circ}_{Z^\BQ}}(Z^\BQ, X_{Z^\BQ})$  is equal to the $p$-integral model $\CA^{[\Lambda_{0}]}_{0, O_{E_\nu}}$ of \S \ref{ss:modoverE}. Indeed, this follows from  the construction in \cite[Theorem  4.2.3]{DvHKM}. 
\end{remark}

Now let $\bK_\U$ be arbitrary, and let $\bK_{\wt G}$ as in \eqref{decprodK}. We define integral models $\wt\CA_{\bK_{\wt G}}$ of ${\rm Sh}_{\bK_{\wt G}}(\wt G, X_{\wt G})$, together with their transition maps through normalization, just like for $\CA_{\bK_{ \U}}$, starting with $\wt\CA_{\bK^\circ_{\wt G}}$. Then these $p$-integral models coincide with those of Definition \ref{def:overOE} (use the normality of those models). By the regularity of $\CA^{[\Lambda_{0}]}_{0, O_{E}}$, we obtain a product decomposition 
\begin{equation}\label{int-decomp-0mo!}
\wt\CA_{\bK_{\wt G},O_{E_\nu}}= \CA_{\bK_{\U}, O_{E_\nu}}\times_{\Spec O_{E_\nu}}   \CA^{[\Lambda_{0}]}_{0, O_{E_\nu}},
\end{equation}
for  any $\bK_{\wt G}= \bK_\U\times \bK^{\circ}_{Z^\BQ}$. It follows from \cite[\S 7.4]{KRZI} that the transition maps for $\wt\CA_{\bK_{\wt G}}$ are finite \'etale.

Let $\bK'_\U\subset \bK_\U$. By the universal property of the normalization, we obtain a commutative diagram of $O_{E_\nu}$-schemes, in which the vertical maps are the transition morphisms, and the horizontal maps are the projection morphisms from the decomposition \eqref{int-decomp-0mo!}, 
\begin{equation}\label{Cart-diagmo}
\begin{tikzcd}
\wt\CA_{\bK'_{\wt G},O_{E_{\nu}}} \arrow[r] \arrow[d] &  \CA_{\bK'_{\U},O_{E_{\nu}}}  \arrow[d] \\
\wt\CA_{\bK_{\wt G},O_{E_{\nu}}} \arrow[r]           &  \CA_{\bK_{\U}, O_{E_{\nu}}}        
\end{tikzcd}
\end{equation}
 By the above, this diagram is cartesian. Since the left vertical map is finite \'etale, so is the right vertical map.
 \end{proof}

\subsection{ $p$-adic uniformization of canonical integral models} 

In this section, we prove a uniformization theorem, in the following form. 
Analogously to the local case, we introduce the  finite extension of $K$ contained in $\bar\BQ$,
\begin{equation}
E(\varphi_0)=\bigcap\nolimits_{\varphi_0\in \Phi^{\prime, +}}E_{\Phi^{\prime,+}}K .
\end{equation}

\begin{theorem}\label{unifor-para}
Let $(\CA_{\bK^{}_{\U}})^\wedge$ be the formal completion of $\CA_{\bK^{}_{\U}}$ along its special fiber, which is a formal scheme over $\Spf O_{K_{v_{0}}}$. Then there exists an isomorphism of formal schemes over $\Spf O_{E(\varphi_{0})_\nu}$,
\begin{equation*}
(\CA_{\bK^{}_{\U}})^\wedge\times_{\Spf O_{K_{v_{0}}}}{\Spf O_{E(\varphi_{0})_\nu}} \simeq   {J_\U(\BQ)}\bs\big[\big(\widehat{\Omega}_{F_{v_{0}}} \times_{\Spf O_{F_{v_{0}}}} \Spf O_{E(\varphi_{0})_\nu}\big) \times
 \U(\BA_{f})/{\bK}^{}_{\U}\big] .
\end{equation*}
Here $J_\U$ is an inner form of $\U$ which is compact at all  archimedean places and is quasi-split at $v_0$ and is locally isomorphic to $\U$ at all other places of $F$ (note that by the Hasse principle for adjoint groups, $J_\U$ is uniquely determined by these conditions). The action on the RHS is through an action on $\widehat{\Omega}_{F_{v_{0}}}$ via an identification of $J_{\U, v_0, \ad}(\BQ_p)$ with $\PGL_2(F_{v_0})$ and an action on $ \U(\BA_{ f})/{\bK}^{}_{\U}$.
\end{theorem}
\begin{proof} We proceed in analogy to the proof of Theorem \ref{local-Shi-Drin}. In the last subsection, we used the fixed CM-type $\Phi^+$ which was used to define the Shimura datum for $\U$. Now we vary this CM-type. Let $\Phi_0^+$ be any   CM-type  with $\varphi_0\in\Phi_0^+$.  As before, we define the torus $Z^\BQ$ with 
$$Z^\BQ(\BQ)=\{t\in K^{\times}: \Nm_{K/F}(t)\in\BQ^{\times}\}$$ 
and use $\Phi_0^+$ to define a Shimura datum $(Z^\BQ, X_Z)$.
Let $E_{0}=E_{\Phi_0^+}$ be the reflex field of $\Phi_0^+$.  Let $\mathbf{K}^{\circ}_{Z^{\BQ}}$ be the maximal compact open subgroup of $Z^\BQ(\BQ)$. We obtain the Shimura variety ${\rm Sh}_{\bK^\circ_{Z^{\BQ}}}(Z^\BQ, X_{Z^\BQ})$, with integral model $\CA^{[\Lambda_0]}_0$ over $O_{E_0}$. 

Let $\wt{G}= \U\times Z^\BQ$, with its Shimura datum $X_{\wt G}= X_\U\times X_Z$. The corresponding reflex field is $E=E_0K$. As in the last subsection, we can identify the Shimura variety  ${\rm Sh}_{\bK_{\wt G}}(\wt G, X_{\wt G})$ with $\wt\CA_{\bK^{\circ}_{\wt G}, E}$.

The identification \eqref{int-decomp-0mo!} (now for $\Phi_0^+$ instead of $\Phi^+$) yields  isomorphisms compatible with changes in $\bK_\U$
\begin{equation}\label{int-decomp-0mo!?}
\wt\CA_{\bK_{\wt G},O_{E_\nu}}=\CA_{\bK_{\U}, O_{E_\nu}} \times_{\Spec O_{E_\nu}}  \CA^{[\Lambda_{0}]}_{0, O_{E_\nu}}  .
\end{equation}
We have an explicit expression for the completion of the LHS along its special fiber. Indeed, by Theorem \ref{main}, we have
\begin{equation*}
(\wt\CA_{\bK_{\wt G}})^\wedge\times_{\Spf O_{E_{\nu}}} \Spf O_{\breve {E}_\nu}  \simeq   \wt{J}(\BQ)\bs\big[\big(\widehat{\Omega}_{F_{v_{0}}} \times_{\Spf O_{F_{v_{0}}}} \Spf O_{\breve {E}_\nu}\big) \times \wt G(\BA_f)/{\bK}_{\wt G}\big].
\end{equation*}
By Corollary \ref{main-cor}, we obtain an isomorphism of formal schemes over $O_{\breve E_\nu}$, with Weil descent data down to $O_{E_\nu}$,
\begin{equation}\label{prodfrom6}
\scalebox{0.93}{$
(\wt\CA_{\bK_{\wt G}})^\wedge\times_{\Spf O_{E_{\nu}}} \Spf O_{\breve {E}_\nu}\simeq {J_\U}(\BQ)\bs\big[\big(\widehat{\Omega}_{F_{v_{0}}} \times_{\Spf O_{F_{v_{0}}}} \Spf O_{\breve {E}_\nu}\big) \times
 \U(\BA_{f})/{\bK}_{\U}\big] \times Z^\BQ(\BQ)\bs Z^\BQ(\BA_f)/\bK^{\circ}_{Z^\BQ}$.}
\end{equation}
We now want to compare the product decompositions of \eqref{int-decomp-0mo!?} and \eqref{prodfrom6}. We have an isomorphism 
\begin{equation}\label{isCM}
\CA^{[\Lambda_{0}]}_{0, O_{\breve E_\nu}}\simeq Z^\BQ(\BQ)\bs Z^\BQ(\BA_f)/\bK^{\circ}_{Z^\BQ},
\end{equation}
and the second projections in \eqref{int-decomp-0mo!?} and \eqref{prodfrom6} are compatible with this identification. It
 follows from considering the fibers that we have an isomorphism 
\begin{equation}\label{AU-unifor}
(\CA_{\bK_{\U}, O_{E_\nu}})^\wedge  \times_{\Spf O_{E_\nu}} \Spf O_{\breve {E}_\nu}\simeq   {J_\U}(\BQ)\bs\big[\big(\widehat{\Omega}_{F_{v_{0}}} \times_{\Spf O_{F_{v_{0}}}} \Spf O_{\breve {E}_\nu}\big) \times
 \U(\BA_{f})/{\bK}_{\U}\big],
 \end{equation}
which, by the transitivity of the action of $Z^\BQ(\BA_f)$ on both sides of \eqref{isCM}, is independent of the fiber. It remains to compare the Weil descent data down to $O_{E_\nu}$ on both sides. The argument is identical to the corresponding point in \S \ref{proofILSV}: the Weil descent datum on the right hand side of \eqref{prodfrom6} is given by multiplication by the special element $\wt{w}$ on the second factor as in Corollary \ref{main-cor}; then one uses the transitive action of $Z^\BQ(\BA_f)$ and the fact that the isomorphism \eqref{AU-unifor}  is independent of the fibers. 

Finally, we vary the CM-types.  We have to  compare the descent data for two CM-types $\Phi_0^+$ and $\Phi_0^{+'}$ containing $\varphi_0$. Again, the argument is identical to that in \S \ref{proofILSV}, so we omit it. 
\end{proof}

\addtocontents{toc}{\protect\addvspace{15pt}}

 \end{document}